\documentclass[a4paper]{article}

\usepackage{amsthm, amssymb}
\usepackage[latin1]{inputenc}
\usepackage[T1]{fontenc}
\usepackage{xspace}
\usepackage{graphicx}
\usepackage{amsthm}
\usepackage{amssymb}
\usepackage[all]{xy}
\usepackage{subfigure}
\usepackage{enumitem}

\usepackage{savesym}
\usepackage{amsmath}
\savesymbol{iint}
\usepackage{txfonts}
\restoresymbol{TXF}{iint}

\usepackage{url} 

\newcommand\g{\gamma}
\newcommand\pr{\varmathbb{P}}
\newcommand\esp{\varmathbb{E}}

\newcounter{csterr}

\newcommand\rrr{\refstepcounter{csterr}R_{\thecsterr}}

\newtheorem{theorem}{Theorem}[section]
\newtheorem{proposition}[theorem]{Proposition}
\newtheorem{lemma}[theorem]{Lemma}
\newtheorem{corollary}[theorem]{Corollary}
\newtheorem{claim}[theorem]{Claim}

\theoremstyle{remark}
\newtheorem{remark}[theorem]{Remark}

\newcounter{cste}

\newcommand\rr{\refstepcounter{cste}R_{\thecste}}

\title{A smooth Gaussian-Kronecker diffeomorphism}

\author{Mostapha Benhenda \\ \texttt{mostaphabenhenda@gmail.com}}

\date{October 9, 2013}

\begin{document}


\maketitle

\begin{abstract}
We build a smooth Gaussian-Kronecker diffeomorphism, answering a question raised by Anatole Katok in his list of "Five Most Resistant Problems in Dynamics".
\end{abstract}

\section{Introduction}

We define Gaussian dynamical systems. We rely on \cite{delarue96}. Let $\g$ a finite symmetric measure on the circle $\mathbb{T}$. A dynamical system $(\Omega,\mathcal{A},\pr,U)$ is \textit{Gaussian of spectral measure} $\g$ if there exists a real centred Gaussian process $(X_p)_{p \in \mathbb{Z}}$ satisfying:

\begin{enumerate}
\item for any integer $p$, $X_p=X_0 \circ U^p$,
\item $\mathcal{B}( X_p, p \in \mathbb{Z})= \mathcal{A}$
\item for any integers $p, q$, $\esp[X_pX_q]= \int_0^1 e^{2i\pi(q-p)t} \; d\g(t)$
\end{enumerate}

Giving $\g$ determines all the properties of this system. In particular, $U$ is ergodic if and only if $\g$ is non-atomic, and in this case, $U$ is even weak mixing (see \cite{cornfeld82}). The entropy of $U$ is zero or infinite, depending on whether or not $\g$ is singular with respect to Lebesgue measure (see \cite{delarue95}). 


We recall that $K \subset \mathbb{T}$ is a \textit{Kronecker set} if any continuous function $f: K \rightarrow \mathbb{T}$ is the uniform limit on $K$ of a sequence of functions $f_p: x \mapsto j_p x$, where $j_p$ is an integer.

In the case where $\g$ is non-atomic and concentrated on $K \cup (\frac{1}{2}-K)$, where $K$ is a Kronecker set of $[0,1/2[$, we say that the Gaussian dynamical system of spectral measure $\g$ is a \textit{Gaussian-Kronecker} transformation \cite{foias67}. A Kronecker set always has zero Lebesgue measure, so a Gaussian-Kronecker transformation always has zero entropy.

Gaussian-Kronecker transformations have remarkable properties: they have simple spectrum $L^p$, for any $p \geq 1$ (see \cite{cornfeld82,iwanik91}). Moreover, they satisfy the \textit{Weak Closure Theorem} \cite{thouvenot86}: any $\pr$-preserving transformation $S$commuting with $T$ is a weak limit of powers of $T$. We can find a sequence of integers $j_p$ such that for any $A \in \mathcal{A}$,

\[ \pr(S^{-1}A \Delta T^{-j_p}A) \rightarrow_{p \rightarrow + \infty} 0 \]

Lastly, they satisfy the spectral stability property: any system that is spectrally isomorphic to a Gaussian-Kronecker is actually metrically isomorphic to it.

In \cite{katok03,katokfive}, Katok raised the problem of the construction of a Gaussian-Kronecker transformation that is smooth. In this paper, we give a construction of this kind:

\begin{theorem}
\label{thfonda}
\label{theoremfonda}
There exists $T \in$ Diff$^\infty(\varmathbb{T} \times [0,1]^{\varmathbb{N}},Leb)$ that is Gaussian-Kronecker.
\end{theorem}

\begin{remark}
Moreover, as in \cite{delarue98}, we can get $T$ non-loosely Bernoulli. This gives another example of a smooth non-loosely Bernoulli diffeomorphism, besides the example we constructed in \cite{nlbernoulli}. However, contrary to \cite{nlbernoulli}, our construction here does not easily generalize to an uncountable family of pairwise non-Kakutani equivalent diffeomorphisms.
\end{remark}

To show this theorem, we provide a smooth version of a transformation constructed by Thierry De La Rue \cite{delarue98}. He constructed a Gaussian-Kronecker and non-loosely Bernoulli transformation using a method of approximation by periodic transformations of the trajectory of the planar Brownian motion.

This adaptation to the smooth case is made using the method of approximation by successive conjugacies, introduced by Anosov and Katok in \cite{anosovkatok70}. More precisely, we generalize our construction of a smooth ergodic diffeomorphism of the annulus $\varmathbb{T} \times [0,1]$ that is equal to a rotation $S_\alpha$ on the boundary, but that is metrically isomorphic to a rotation $R_\beta$ of the circle, such that $\alpha \not\eq \pm \beta$. In our generalization, we construct a sequence of smooth periodic diffeomorphisms of $\varmathbb{T} \times [0,1]^{\varmathbb{N}}$, where $[0,1]^{\varmathbb{N}}$ is the Hilbert cube, such that each one is metrically isomorphic to an increasingly large family of mutually independent periodic rotations of the circle with different angles. The limit angles of these rotations will provide the Kronecker spectrum of the limit diffeomorphism.

We need to carry the construction on $\varmathbb{T} \times [0,1]^{\varmathbb{N}}$, instead of $\varmathbb{T} \times [0,1]$, because in order to ensure the smooth convergence of the sequence of smooth periodic diffeomorphisms, we need a space of infinite dimensions (see page \pageref{whyinfinite} for details). This framework is new with respect to known Anosov-Katok constructions.

\subsection{Definitions}


Let Diff$^\infty(M,\mu)$ be the class of smooth diffeomorphisms of $M$ preserving the Lebesgue measure $\mu$.  For $B \in$ Diff$^\infty(M,\mu)$ and $j \in \varmathbb{N}^*$, let $D^jB$ be the $j^{th}$ derivative of $B$ if $j >0$, and the $-j^{th}$ derivative of $B^{-1}$ if $j<0$. For $x \in M$, let $|D^jB(x)|$ be the norm of $D^jB(x)$ at $x$. We denote $\|B\|_k= \max_{0< |j| \leq k} \max_{ x \in M} |D^jB(x)|$.

A \textit{finite measurable partition} $\bar{\xi}$ of a measured manifold $(N, \nu)$ is the equivalence class of a finite set $\xi$ of disjoint measurable subsets of $N$ whose union is $N$, modulo sets of $\nu$-measure zero. In most of this paper, we do not distinguish a partition $\xi$ with its equivalent class $\bar{\xi}$ modulo sets of $\nu$-measure zero. In these cases, both are denoted $\xi$. Moreover, all partitions considered in this paper are representatives of a finite measurable partition. 




A partition $\xi'$ is \textit{subordinate} to a partition $\xi$ if any element of $\xi$ is a union of elements of $\xi'$, modulo sets of $\nu$-measure zero. In this case, if $\mathcal{B}(\xi)$ denotes the completed algebra generated by $\xi$, then $\mathcal{B}(\xi) \subset \mathcal{B}(\xi')$. The inclusion map $i: \mathcal{B}(\xi) \rightarrow \mathcal{B}(\xi')$ will be denoted $\xi \hookrightarrow \xi'$. This notation also means that $\xi'$ is \textit{subordinate} to $\xi$. A sequence of partitions $\xi_n$ is \textit{monotonic} if for any $n$, $ \xi_n \hookrightarrow \xi_{n+1}$. These definitions and properties are independent of the choice of the representatives $\xi$ and $\xi'$ of the equivalence classes $\bar{\xi}$ and $\bar{\xi'}$. 

A measure-preserving bijective bimeasurable map $T: (M_1,\mu_1, \mathcal{B}_1) \rightarrow (M_2,\mu_2, \mathcal{B}_2)$ induces an \textit{isomorphism} of measure algebras, still denoted $T: (\mu_1, \mathcal{B}_1) \rightarrow (\mu_2, \mathcal{B}_2)$. If $\xi_1,\xi_2$ are partitions, and if $\mathcal{B}_1=\mathcal{B}(\xi_1)$ and $\mathcal{B}_2=\mathcal{B}(\xi_2)$, we denote $T: \xi_1 \rightarrow \xi_2$ this induced isomorphism of measure algebras. If $M_1=M_2$, $\mu_1=\mu_2$ and $\mathcal{B}_1= \mathcal{B}_2$, then $T$ is a \textit{measure-preserving transformation}. Its induced isomorphism is an \textit{automorphism} (see \cite[p.43]{halmos56} and \cite{weiss72}).

A \textit{metric isomorphism} $L$ of measure-preserving transformations $T_1:(M_1,\mu_1, \mathcal{B}_1) \rightarrow (M_1,\mu_1, \mathcal{B}_1) $, $T_2:(M_2,\mu_2, \mathcal{B}_2) \rightarrow (M_2,\mu_2, \mathcal{B}_2)$ is a measure-preserving bijective bimeasurable map $L: (M_1,\mu_1, \mathcal{B}_1) \rightarrow (M_2,\mu_2, \mathcal{B}_2)$ such that $L T_1= T_2 L$ a.e. For convenience, when the measure is the Lebesgue measure and the algebra is the Borelian algebra, we omit to mention the measures and algebras, and we simply say that $L:(M_1,T_1) \rightarrow (M_2,T_2)$ is a metric isomorphism.

Let $\bar{\xi}$ be a measurable partition and $\xi$ a representative of this equivalent class modulo sets of $\mu$-measure zero. For $x \in M$, we denote $\xi(x)$ the element of the partition $\xi$ such that $x \in \xi(x)$. A sequences of partitions $\xi_n$ of measurable sets \textit{generates} if there is a set of full measure $F$ such that for any $x \in F$,

\[  \{ x \} = F \bigcap_{n \geq 1} \xi_n(x) \]

This property of generation is independent of the choice of the representatives $\xi_n$ of the equivalent class $\bar{\xi_n}$ and therefore, we will say that the sequence of measurable partitions $\bar{\xi_n}$ generates.
Let $M/\xi$ denote the equivalent class of the algebra generated by $\xi$, modulo sets of $\mu$-measure zero. $M/\xi$ is independent of the choice of the representative $\xi$ of the equivalent class $\bar{\xi}$. If $T: M_1 \rightarrow M_2$ is a measure-preserving map such that $T(\xi_1)=\xi_2$ $\mu$-almost everywhere, we can define a quotient map: $T/\xi_1: M/\xi_1 \rightarrow M/\xi_2$.


Let $M= \varmathbb{T} \times [0,1]^{\varmathbb{N}}$. We consider the periodic flow $S_t$ defined by: 

\begin{eqnarray*}
S_t \colon \varmathbb{T} \times [0,1]^{\varmathbb{N}}  &\to \varmathbb{T} \times [0,1]^{\varmathbb{N}}  \\
(s,x) &\mapsto ( t+s \mod 1,x)
\end{eqnarray*}
For $a,b \in \varmathbb{T}^1$, let $[a,b[$ be the positively oriented circular sector between $a$ and $b$, with $a$ included and $b$ excluded.




For $I \subset \varmathbb{T}$ or $[0,1]$, and $i \in \varmathbb{N}$, we denote $(I)_i= \varmathbb{T} \times [0,1] \times ... \times [0,1] \times I \times [0,1] \times .....$, where $I$ is located at the $i^{th}$ position. ($I \subset \varmathbb{T}$ if $i=0$, $I \subset [0,1]$ if $i \geq 1$).

For $i \not\eq i'$, we denote $ (I)_i \times (I')_{i'}= (I)_i \cap (I')_{i'}$.

By \cite[p. 157]{halmos74}, the infinite Cartesian product of the one-dimensional Lebesgue measure defines a probability measure $\mu$ on $M$, still called Lebesgue measure.

A sequence $T_n$ of $\mu$-preserving maps \textit{weakly converges} to $T$ if, for any measurable set $E$, $\mu(T_n E \Delta E) \rightarrow 0$, where $A \Delta B= (A - B) \cup (B - A)$.

let $E$ be the set of bijections $f$ of $\varmathbb{T} \times [0,1]^{\varmathbb{N}}$, such that $f((x_i)_{i \in \varmathbb{N}})= (f_n(x_1,...,x_n),x_{n+1},..)$ with $f_n$ a smooth diffeomorphism of $\varmathbb{T} \times [0,1]^n$  for some integer $n$. The smooth distance on finite-dimensional smooth diffeomorphisms extends to $E$: $d(f,g)=d(f_{n},g_n)$ (we take $n$ to be the maximum of the integers $n(f)$ and $n(g)$ associated with $f$ and $g$ respectively). The completion of $E$ for $d$ is denoted $\mbox{Diff}^\infty(\varmathbb{T} \times [0,1]^{\varmathbb{N}})$. It corresponds to the set of smooth diffeomorphisms of $\varmathbb{T} \times [0,1]^{\varmathbb{N}}$, and extends the finite-dimensional notion of smooth diffeomorphisms.




\subsection{Transformations of the planar Brownian motion}

We use the representation of a Gaussian dynamical system as a geometric transformation of the trajectory of the planar Brownian motion, developed in \cite{delarue95}. We denote by $(\Omega, \mathcal{A},\pr)$ the canonical space of the planar Brownian motion issued from $0$, on the time interval $[0,1]$:

\begin{enumerate}
\item $\Omega=C_0^0([0,1], \varmathbb{C})$ is the space of continuous maps from $[0,1]$ to $\mathbb{C}$, that cancels in $0$.
\item $\pr$ is the Wiener measure on $\Omega$,
\item $\mathcal{A}$ is the Borelian sigma-algebra, completed for $\pr$.
\end{enumerate}

For $\omega \in \Omega$, and $0 \leq u \leq 1$, we denote by $B_u(\omega)$ the position of the trajectory $\omega$ at time $u$.

If $\sigma$ is a probability measure on $[0,1]$ concentrated in a finite number of points $\alpha_1<...<\alpha_{t}$ of respective weights $m_1,...,m_t$, we define a transformation $U_\sigma$ of $\Omega$, preserving $\pr$, by the following: for any $k=1,...,t$, $u_k=\sum_{j=1}^k m_j$, and $u_0=0$. We cut the trajectory $\omega$ in $t$ pieces, corresponding to time intervals $[u_{k-1},u_k], 1 \leq k \leq t$, then we perform a rotation of angle $\alpha_k$ on the $k^{th}$-piece. $U_\sigma \omega$ is the trajectory obtained by gluing the rotated pieces. Thus, for $u \in [u_j,u_{j+1}]$, 

\[ B_u \circ U_\sigma = \sum_{k=1}^j e^{2i \pi \alpha_k} (B_{u_k}- B_{u_{k-1}})+ e^{2i \pi \alpha_{j+1}} (B_{u}- B_{u_{j}})  \]

Suppose, moreover, that $\alpha_i= \frac{p_n}{q_n} b_{n,i}$, with $p_n$ and $b_{n,i}$ integers relatively prime with the integer $q_n$. Let $U_n=U_\sigma$. Let also

\[  c_{n,i,l}= \left\{ \omega \in \Omega / \; \arg(B_{u_{i+1}}(\omega)-B_{u_{i}}(\omega)) \in \left[ \frac{l}{q_n}, \frac{l+1}{q_n} \right[  \right\}  \]

and $\zeta_{n,i}= \left\{ c_{n,i,l}, l=0,...,q_n-1 \right\}$. $\zeta_{n,i}$ is a partition stable by $U_n$, the $\zeta_{n,i},i=0,...,t_n-1$ are mutually independent, and $(U_{n|\zeta_{n,i}}, \zeta_{n,i})$ is metrically isomorphic to $(R_{\frac{p_n}{q_n}b_{n,i}}, \xi_n)$, where $R_{\frac{p_n}{q_n}b_{n,i}}$ is the circle rotation of angle $\frac{p_n}{q_n}b_{n,i}$, and $\xi_n$ is the partition of the circle $\varmathbb{T}$ defined by $\xi_n=\{ \left[ \frac{l}{q_n}, \frac{l+1}{q_n} \right[ , l=0,...,q_n-1 \}$. This metric isomorphism is the basis of our construction.

Now, if $\sigma$ is a non-atomic probability measure on $[0,1]$, we can define $U_\sigma$ as he limit of a sequence of transformations $U_{\sigma_n}$, where the measures $\sigma_n$ are concentrated on a finite number of points converging sufficiently well towards $\sigma$. For any $u \in [0,1]$, and any $p \in \varmathbb{Z}$, we have:

\[  B_u \circ U_\sigma^p = \int_0^u e^{2i\pi p \psi(s)} \, d B_s \]

where $\psi(s)= \inf \{ x \in [0,1] / \sigma([0,x]) \geq s \}$. 

Moreover, if $\sigma$ is concentrated on $[0,1/2]$, then $(\Omega, \mathcal{A}, \pr,U_\sigma)$ is a Gaussian dynamical system of spectral measure $\g$, where $\g$ is the symmetric probability measure on $[-1/2,1/2]$, defined by

\[ \g(A)= \frac{1}{2} \left( \sigma(A\cap [0,1/2])+ \sigma(-A\cap [0,1/2]) \right)  \]

The underlying real Gaussian process is given by $X_0=\mbox{Re}(B_1)$ \cite{delarue95}. 

As a corollary, if $\sigma$ is non-atomic on $[0,1]$, then $U_\sigma$ is ergodic (and even weak mixing), and if, moreover, $\sigma$ is singular with respect to Lebesgue measure, then $U_\sigma$ has zero entropy.

\subsection{Basic steps of the proof}

The metric isomorphism is obtained as the limit of isomorphisms of finite algebras. We use the lemma \cite[p.18]{anosovkatok70}:

\begin{lemma}
\label{lemmekatokrotisom}
Let $M_1$ and $M_2$ be Lebesgue spaces and let $\xi_n^{(i)}$ ($i=1,2$) be monotonic and generating sequences of finite measurable partitions of $M_i$. Let $T_n^{(i)}$ be automorphisms of $M_i$ such that $T_n^{(i)} \xi_n^{(i)}= \xi_n^{(i)}$ and $T_n^{(i)} \rightarrow T^{(i)}$ in the weak topology. Suppose there are metric isomorphisms $L_n: M_1 / \xi_n^{(1)} \rightarrow M_2 / \xi_n^{(2)} $ such that 

\[ L_n T_n^{(1)} / \xi_n^{(1)} =  T_n^{(2)} / \xi_n^{(2)} L_n \] 
and

\[ L_{n+1} \xi_n^{(1)} = \xi_n^{(2)} \]

then $(M_1,T_1)$ and $(M_2, T_2)$ are metrically isomorphic.

\bigskip

Said otherwise, if we have generating sequences of partitions and sequences of automorphisms $T_n^{(i)}$ weakly converging towards $T^{(i)}$, and if, for any integer $n$, the following diagram commutes:

\[ \xymatrix{
\xi_n^{(1)} \ar@(ul,dl)[]_{T_n^{(1)}}  \ar@{->}[r]^{L_n} \ar@{^{(}->}[d]  & \xi_n^{(2)} \ar@(ur,dr)[]^{T_n^{(2)}} \ar@{^{(}->}[d]  \\
\xi_{n+1}^{(1)}  \ar@{->}[r]^{L_{n+1}} & \xi_{n+1}^{(2)} 
 } \]

then $(M_1,T_1)$ and $(M_2, T_2)$ are metrically isomorphic.

\end{lemma}

The proof of theorem \ref{theoremfonda} is in two steps. In the first step (lemma \ref{lemmefondarotisom}), we determine sufficient conditions such that there exists sequences of finite partitions and automorphisms satisfying the assumptions of lemma \ref{lemmekatokrotisom} with $M_1= \Omega$, $M_2=M=\varmathbb{T} \times [0,1]^{\varmathbb{N}}$, $T_n^{(1)}=U_n$, $T_n^{(2)}=T_n$, where $T_n$ is a smooth diffeomorphism, and such that the limit $T$ in the smooth topology of the sequence $T_n$ is smooth. 

In the second step (lemma \ref{lemme1point3rotisom}), we construct sequences of integers satisfying the conditions of the first step.

\begin{lemma} 
\label{lemmefondarotisom}
\label{lemmefonda}
There exists an explicit family of integers $\rr(n)\geq n \label{r0csterotisom}, \rr \label{rtn}, \rr \label{rqn}$, such that, if there exist increasing sequences of integers $t_n,p_n,q_n, a_{n}(i'/t_n),b_{n}(i'/t_n) \in \varmathbb{N}^*$, $i'=0,...,t_n-1$, and sequences $s_{n}(i'/t_n)  \in \varmathbb{Z}^*$, $i'=0,...,t_n-1$ such that, for any integer $n$, any $i'=0,...,t_n-1$,

\begin{enumerate}
\item  \textit{(temporal monotonicity)} \label{60rotisom} $t_{n+1}=2^{R_{\ref{rtn}}(n,q_n,t_n)} t_n$.
\item \textit{(primality)} \label{61}\label{61rotisom} $a_{n}(i'/t_n)b_{n}(i'/t_n)-s_{n}(i'/t_n)q_n=1$.
\item \textit{(monotonicity)} \label{64}\label{63rotisom}$q_n$ divides $q_{n+1}$.
\item \textit{(isomorphism)} \label{65}\label{65rotisom} $q_n$ divides $a_{n+1}(i/t_{n+1})-a_{n}(i'/t_n)$, for $i=i'\frac{t_{n+1}}{t_n} ,...,  (i'+1)\frac{t_{n+1}}{t_n}-1$.
\item \textit{(convergence of the diffeomorphism, generation, Kronecker)} \label{66rotisom} \[ 0< \left| \frac{p_{n+1}}{q_{n+1}} - \frac{p_n}{q_n} \right| \leq 1/R_{\ref{r0csterotisom}}\left(n,q_n, \prod_{i=0}^{t_{n+1}-1} b_{n+1}(i/t_{n+1}) \right) \]




then there exists a smooth ergodic measure-preserving diffeomorphism $T$ of $M$ that is a Gaussian transformation. Moreover, if 

\item \textit{(Kronecker)} \label{67rotisom} the translation of $\varmathbb{T}^{t_n}$ of vector $\frac{p_n}{q_n} \left( b_n(0/t_n),...,b_n((t_n-1)/t_n)\right)$ has a fundamental domain of diameter smaller than $1/n$,

\end{enumerate}

then $T$ is a Gaussian-Kronecker.

\end{lemma}

\begin{lemma}
\label{lemme1point3rotisom}
For any $n \geq 1$, there exists increasing sequences of integers

$t_n,p_n,q_n, a_{n}(i'/t_n),b_{n}(i'/t_n) \in \varmathbb{N}^*$, $i'=0,...,t_n-1$, and sequences $s_{n}(i'/t_n)  \in \varmathbb{Z}^*$, $i'=0,...,t_n-1$ satisfying the assumptions of lemma \ref{lemmefondarotisom}.

\end{lemma}



We divide the proof of lemma \ref{lemmefondarotisom} in three main parts. In the first part of the proof, we construct a sequence of monotonic and generating partitions of $\Omega$, called $\zeta_n^\infty$, which is stabilized by the transformation $U_n$. To that end, we use assumptions $\ref{60rotisom}$, $\ref{61rotisom}$, $\ref{63rotisom}$, $\ref{65rotisom}$ and $\ref{66rotisom}$. In the second part of the proof, we elaborate sufficient conditions on $B_n \in$ Diff$^\infty(M,\mu)$, so that if $T_n= B_n^{-1} S_{\frac{p_n}{q_n}} B_n$ weakly converges towards an automorphism $T$, then there exists a metric isomorphism between $(\Omega,U,\pr)$ and $(M,T,\mu)$. To that end, we apply lemma \ref{lemmekatokrotisom}: we construct a monotonous and generating sequence of partitions $\xi_n^\infty$ of $M$ and a sequence of isomorphisms $\bar{K}_n^\infty: \Omega/\zeta_n^\infty \rightarrow M/ \xi_n^\infty$, such that $ \bar{K}_n^\infty U_n = T_n\bar{K}_n^\infty $ and $\bar{K}_{n+1|\zeta_n^\infty}^\infty=\bar{K}_n^\infty$. In the construction of this isomorphism, assumption $\ref{65rotisom}$ is important. Moreover, elements of $\xi_n^\infty$ must be chosen in a way that ensures the monotonicity of the sequence $\bar{K}_n^\infty$. This condition of monotonicity induces combinatorial constraints on the elements of the partition $\xi_n^\infty$. 

In the third part of the proof, we construct diffeomorphisms $T_n= B_n^{-1} S_{\frac{p_n}{q_n}}B_n $ on $M$ stabilizing $\xi_n^\infty$, obtained by successive conjugations from the rotation $S_{\frac{p_n}{q_n}}$. In particular, in this  part, we use smooth quasi-permutations, introduced in \cite{nlbernoulli}.

\subsection{Construction of suitable sequences of integers: proof of lemma \ref{lemme1point3rotisom}}

\begin{lemma}
There exist increasing sequences of integers $t_n,p_n,q_n, a_{n}(i'/t_n),b_{n}(i'/t_n) \in \varmathbb{N}^*$, $i'=0,...,t_n-1$, and sequences $s_{n}(i'/t_n)  \in \varmathbb{Z}^*$, $i'=0,...,t_n-1$, satisfying the assumptions of lemma \ref{lemmefondarotisom}.

\end{lemma}

\begin{proof}

We construct these sequences by induction. Let $t_0=p_0=q_0=1$, $b_{0}(0/t_0)=s_{0}(0/t_0)=1$, $a_{0}(0/t_0)=2$. Suppose we have defined $t_k,p_k,q_k,a_k(i'/t_k),b_k(i'/t_k),s_k(i'/t_k)$, satisfying the assumptions of lemma \ref{lemmefondarotisom}, up to the rank $k=n$, and let us define $t_{n+1},p_{n+1},q_{n+1},a_{n+1}(i/t_{n+1}),b_{n+1}(i/t_{n+1}),s_{n+1}(i/t_{n+1})$.

We can define $t_{n+1}=2^{R_{\ref{rtn}}(n,q_n,t_n)} t_n$. Let 

\[  b(n)= ( b_n(0/t_n),...,b_n((t_n-1)/t_n)) \]

\[  \tilde{b}(n)= ( b_n(0/t_n),..., b_n(0/t_n), b_n(1/t_n),..., b_n(1/t_n),...,b_n((t_n-1)/t_n),...,b_n((t_n-1)/t_n)) \]

\[= ( \tilde{b}_n(0/t_{n+1}),...,\tilde{b}_n((t_{n+1}-1)/t_{n+1})) \] 

where each $b_n(i/t_n)$ is repeated $t_{n+1}/t_n$ times.

To get conditions $\ref{66rotisom}$, $\ref{67rotisom}$, as in \cite{anosovkatok70,katokant}, we seek $b(n+1)$ of the form:

\[  b(n+1)= q_n v(n+1)+ (e_{n+1}q_n+1) \tilde{b}(n) \]

where $v(n+1) \in \varmathbb{Z}^{t_{n+1}}$ and $e_{n+1} \in \varmathbb{N}$. 

Let $b_n=\prod_{i'=0}^{t_n-1} b_n(i'/t_n)$. Since $\gcd(\{ \tilde{b}_n(0/t_{n+1}),...,\tilde{b}_n((t_{n+1}-1)/t_{n+1})\}= \gcd(\{  b_n(0/t_n),...,b_n((t_n-1)/t_n) \})=1$, then we can apply a result in \cite[section 1.3.2]{katokant}: there exists $v(n+1) \in \varmathbb{Z}^{t_{n+1}}$ and $e_{n+1} \in \varmathbb{N}$, there exists 
$\rr(n,b(n)) \label{rnbn}$ such that:

\[  \|  v(n+1) \| \leq R_{\ref{rnbn}}(n,b_n) \]

\[   e_{n+1} \leq R_{\ref{rnbn}}(n,b_n) \]

and such that the translation flow of $\varmathbb{T}^{t_{n+1}}$ of vector $b(n+1)$ has a fundamental domain $B(n) \subset \varmathbb{T}^{t_{n+1}-1} \times \{0\}$ of diameter smaller than $1/(2n)$. Therefore, for any $q_{n+1} \geq \rr(n,b(n)) \label{rnbnqnplusun}$ for some $R_{\ref{rnbnqnplusun}}(n,b(n))$, and any $p_{n+1}$ such that $\gcd(p_{n+1},q_{n+1})=1$, the translation of $\varmathbb{T}^{t_{n+1}}$ of vector $\frac{p_{n+1}}{q_{n+1}} b(n+1)$ has a fundamental domain of diameter smaller than $1/n$. Hence condition \ref{67rotisom}.

We write: \[  v(n+1)= ( v_{n+1}(0/t_{n+1}),...,v_{n+1}((t_{n+1}-1)/t_{n+1})) \]

Let \[  s_{n+1}(i/t_{n+1})= s_n(i'/t_n)+ a_n(i'/t_n) [ v_{n+1}(i/t_{n+1}) + e_{n+1}  b_{n}(i'/t_{n}) ]  \]

Let \[  \mu_{n+1}(i/t_{n+1})= b_n(i'/t_n)+ q_n [ v_{n+1}(i/t_{n+1}) + e_{n+1}  b_{n}(i'/t_{n}) ]  \]

Let $d_{n+1} \geq R_{\ref{r0csterotisom}}\left(n,q_n, b_{n+1} \right)$ be an integer, let

\[  c_{n+1}(i/t_{n+1})= d_{n+1}  s_{n+1}(i/t_{n+1}) \prod_{k=0, k \not\eq i }^{t_{n+1}-1} \mu_{n+1}(k/t_{n+1})   \]

and let

\[ q_{n+1}= q_n  \left( 1+ d_{n+1} \prod_{k=0}^{t_{n+1}-1} \mu_{n+1}(k/t_{n+1}) \right) \]

Notice that $q_{n+1}$ is independent of $i$ (it is important to obtain condition $\ref{61rotisom}$). Let

\[  a_{n+1}(i/t_{n+1})= a_n(i'/t_n)+ q_n c_{n+1}(i/t_{n+1})  \]

Thus, assumption \ref{65rotisom} is also satisfied. 

We show assumption \ref{61rotisom} at rank $n+1$. We have:

\[  a_{n+1}(i/t_{n+1})  b_{n+1}(i/t_{n+1})  = \left(a_n(i'/t_n)+ q_n c_{n+1}(i/t_{n+1}) \right) \left( q_n  v_{n+1}(i/t_{n+1}) + (e_{n+1} q_n +1) b_n(i'/t_n)  \right)  \]

\[  = 1+ q_n \left[   s_n(i'/t_n) + e_{n+1} (1+ s_n(i'/t_n) q_n) + a_n(i'/t_n) v_{n+1}(i/t_{n+1}) \right. \] \[ \left. + c_{n+1}(i/t_{n+1}) \left( q_n  v_{n+1}(i/t_{n+1}) + (e_{n+1} q_n+1) b_n(i'/t_n)  \right)    \right] \]

\[  = 1+ q_n \left[   s_n(i'/t_n) +    a_n(i'/t_n) ( e_{n+1} b_n(i'/t_n) +v_{n+1}(i/t_{n+1}) )  \right. \] \[ \left. + d_{n+1} s_{n+1}(i/t_{n+1} )   \prod_{k=0, k \not\eq i }^{t_{n+1}-1} \mu_{n+1}(k/t_{n+1})      \left( q_n  v_{n+1}(i/t_{n+1}) + (e_{n+1} q_n+1) b_n(i'/t_n)  \right)    \right] \]

\[ = 1+ q_n \left[ s_{n+1}(i/t_{n+1}) \left(1+ d_{n+1}   \prod_{k=0, k \not\eq i }^{t_{n+1}-1} \mu_{n+1}(k/t_{n+1})   \right)   \right]  \]

\[ = 1+  s_{n+1}(i/t_{n+1}) q_{n+1}  \]

Hence assumption \ref{61rotisom} at rank $n+1$.
\end{proof}

\subsubsection{Proof that $T$ is Kronecker}

First, we show:

\begin{lemma}
Let \[ L_n\left( \frac{i'}{t_n} \right)= \left[ \frac{p_n}{q_n}  b_n(i'/t_n), \frac{p_n}{q_n} b_n(i'/t_n)+ \frac{1}{nq_n} \right] \]

and \[  L=  \bigcap_{n \geq 1} \bigcup_{i'=0}^{t_n-1}  L_n(i'/t_n) \]

If the translation $\frac{p_n}{q_n} b(n)$ of the $\varmathbb{T}^{t_n}$-torus has a fundamental domain of diameter smaller than $ 1/n$, then $L$ is a Kronecker set.

\end{lemma}

\begin{proof}
We adapt the proof of \cite[p.6]{delarue98}. We must show that for any $f:L \rightarrow \varmathbb{T}$ continuous, for any $\epsilon>0$, there exists $k \in \varmathbb{Z}$ such that

\[  \sup_{x \in L} |f(x)-kx| \leq \epsilon \]

We fix $n \geq 1$ and first, we suppose that $f$ is constant on each interval $L_n\left( \frac{i'}{t_n} \right)$, and we denote $z_n(i'/t_n)= f\left(L_n\left( \frac{i'}{t_n} \right)\right)$. Let $z_n(i'/t_n)=(z_n(0),...,z_n((t_n-1)/t_n))$. Since the translation $\frac{p_n}{q_n} b(n)$ of the $\varmathbb{T}^{t_n}$-torus has a fundamental domain of diameter smaller than $ 1/n$, then there exists $0 \leq k \leq q_n-1$ such that 

\[  |k  \frac{p_n}{q_n} b(n) -u| \leq \frac{1}{n}  \]

Therefore, for any $0 \leq i' \leq q_n-1$,

\[  |k  \frac{p_n}{q_n} b_n(i'/t_n) -z_n(i'/t_n)| \leq \frac{1}{n}  \]

Let $x \in L$ and $0 \leq i' \leq t_n-1$ such that $x \in L_n\left( \frac{i'}{t_n} \right)$. We have:

\[ \left| x- \frac{p_{n}}{q_{n}}   b_{n}(i'/t_{n}) \right| \leq \frac{1}{nq_{n}}    \]

Therefore,


\[ | k x - f(x)|= |kx- z_n(i'/t_n) | \leq \left| k x- k \frac{p_{n}}{q_{n}}   b_{n}(i'/t_{n})  \right| +  \left| k  \frac{p_{n}}{q_{n}}   b_{n}(i'/t_{n}) -  z_n(i'/t_n) \right|  \leq  \frac{k}{nq_{n}} + \frac{1}{n}  \]

\begin{equation}
\label{eqcasconstant}
| k x - f(x)| \leq \frac{2}{n} 
\end{equation}

Now, we suppose $f: L \rightarrow \varmathbb{T}$ is any continuous map. Let $\epsilon>0$. $L$ is compact, so by Heine's theorem, $f$ is uniformly continuous. There exists $\eta>0$ such that for any $x,y \in L$, if $|x-y| \leq \eta$, we have:

\begin{equation}
\label{unifcontinuite}
|f(x)-f(y)| \leq \epsilon/2
\end{equation}

We fix $n \geq  \max \left( 4/\epsilon, 1/\eta \right)$. Let $f_n: L \rightarrow \varmathbb{T}$  such that for any $i'=0,...,t_n-1$, any $y \in  L_{n}(i'/t_{n})$, $f_n(y)=f\left( \frac{p_{n}}{q_{n}}   b_{n}(i'/t_{n})\right)$. By relation (\ref{eqcasconstant}), there exists $0 \leq k \leq q_n$ such that for any $x \in L$, 

\[ | k x - f_n(x)| \leq \frac{2}{n}  \]

Let $x \in L$ and let $i'$ such that $x \in  L_{n}(i'/t_{n})$. Since $f_n$ is constant on $ L_{n}(i'/t_{n})$, we have:

\[  |kx - f(x)| \leq |k x - f_n(x)| + |f_n(x) - f(x) | =   |k x - f_n(x)| + |f_n(x) - f(x) | \]

\[  |kx - f(x)| \leq  \frac{2}{n} + \left|f\left( \frac{p_{n}}{q_{n}}   b_{n}(i'/t_{n}) \right) - f(x) \right|  \]

Since

\[ | x - \frac{p_{n}}{q_{n}}   b_{n}(i'/t_{n} ) | \leq \frac{1}{nq_n} \leq  \frac{1}{n} \leq \eta \]

By estimation (\ref{unifcontinuite}), we conclude:

\[  |kx - f(x)| \leq  \frac{2}{n} +  \epsilon/2 \leq \epsilon \]

\end{proof}

\begin{corollary}
For a suitable choice of $R_{\ref{r0csterotisom}}$ in lemma \ref{lemmefondarotisom}, $T$ is Gaussian-Kronecker.
\end{corollary}

\begin{proof}
We must show that the limit $\sigma$ of $\sigma_n$ in the weak topology is non-atomic, and that $supp(\sigma) \subset L$. Let $\epsilon>0$. Since $t_n \rightarrow + \infty$, then for $n$ sufficiently large, for any $x \in \varmathbb{T}$, $\sup_{x \in \varmathbb{T}} \sigma_n(x) \leq \epsilon$. Therefore, $\sigma$ is non-atomic.n To show that the support of $\sigma$ is included in $L$, let:

\[ f : \begin{array}[t]{lcl} \{0,...,t_{n+1}-1\}  &\rightarrow &   \{ 0,...,t_n-1\}  \\
               i & \mapsto    &  i' \; \mbox{s.t.} \; \frac{i'}{t_n} \leq \frac{i}{t_{n+1}}< \frac{i'+1}{t_n} 
           \end{array}
           \]

Let $n \in \varmathbb{N}$. For any $i \in \{0,...,t_{n+1}-1\}$, we have: 

\[ \frac{p_{n+1}}{q_{n+1}} b_{n+1}(i/t_{n+1}) - \frac{p_{n}}{q_{n}} b_{n}(f(i)/t_n)= \left( \frac{p_{n+1}}{q_{n+1}}  - \frac{p_{n}}{q_{n}} \right) b_{n+1}(i/t_{n+1}) \]

Therefore, for any $m \geq 0$, any $i \in \{0,...,t_{n+m}-1\}$, and for a suitable choice of $R_{\ref{r0csterotisom}}$ in lemma \ref{lemmefondarotisom},

\[ \frac{p_{n+m}}{q_{n+m}} b_{n+m}(i/t_{n+m}) - \frac{p_{n}}{q_{n}} b_{n}(f^m(i)/t_n)= \sum_{k=0}^{m-1} \left( \frac{p_{n+k+1}}{q_{n+k+1}}  - \frac{p_{n+k}}{q_{n+k}} \right) b_{n+k+1}(f^{k}(i)/t_{n+k+1}) \] 

\[ \leq   \sum_{k=n}^{+ \infty} \left( \frac{p_{k+1}}{q_{k+1}}  - \frac{p_{k}}{q_{k}} \right) b_{k+1}(f^{k}(i)/t_{n+1}) \leq \frac{1}{nq_n} \]

Therefore, for any integers $n,m$, $\frac{p_{n+m}}{q_{n+m}} b_{n+m}(i/t_{n+m}) \in L_n=\bigcup_{i'=0}^{t_n-1} L_n(i'/t_n)$. 

Therefore, $supp(\sigma_{n+m}) \subset L_n$. Since $L_n$ is closed, then $supp(\sigma) \subset L_n$ for any $n$, and so $supp(\sigma) \subset L$.

\end{proof}

\section{Partitions of $\Omega= C_0^0([0,1], \varmathbb{C})$   }
\label{partitionomega}

The aim of this section is to show the following proposition:

\begin{proposition}
\label{summarypartitions}
If assumptions \ref{60rotisom}, \ref{61rotisom}, \ref{63rotisom}, \ref{65rotisom}, \ref{66rotisom}, of lemma \ref{lemmefonda} hold, there exists measurable partitions $(\zeta_n^m)_{n\geq 0, n <m}$ of $\Omega$, such that $\zeta_n^m$ is stable by $U_n$, and such that at $m$ fixed, for $n<m$, $\zeta_n^m \hookrightarrow \zeta_{n+1}^m$. Moreover, there exists an isomorphism $Q_n^m: \zeta_n \rightarrow \zeta_n^m$ commuting with $U_n$.

Moreover, at $n$ fixed, $ \zeta_n^m$ converges as $ m \rightarrow + \infty$ towards a partition $\zeta_n^\infty$, stable by $U_n$. Moreover, the sequence $(\zeta_n^\infty)_{n \geq 0}$ is monotonous and generates.

\end{proposition}


A natural partition stabilized by $U_n$ is given by $\zeta_n= \vee_{i=0}^{t_n-1} \zeta_{n,i}$, where:

\[ \zeta_{n,i}= \left\{ c_{n,i,l}, l=0,...,q_n-1 \right\} \]

with

\[  c_{n,i,l}= \left\{ \omega \in \Omega / \; \mbox{arg}(B_{u_{i+1}}(\omega)-B_{u_{i}}(\omega)) \in \left[ \frac{l}{q_n}, \frac{l+1}{q_n} \right[  \right\}  \]

In order to apply lemma \ref{lemmekatokrotisom}, we need a monotonous and generating sequence of partitions. In lemma \ref{zngenerates}, we show that $\zeta_n$ generates. However, $\zeta_n$ is not monotonous, because for any integers $k,i,l,l'$,

\[ \left\{ \omega \in \Omega  \;  / \;  \arg\left(B_{\frac{k}{t_n}+ \frac{i+1}{t_{n+1}} }(\omega) - B_{\frac{k}{t_n}+ \frac{i}{t_{n+1}} }(\omega) \right) \in  I_{\frac{l}{q_{n+1}}} \right\}  \not\subset \left\{ \omega \in \Omega  \;  / \; \arg\left(B_{\frac{k+1}{t_n} }(\omega) - B_{\frac{k}{t_n} }(\omega) \right) \in I_{\frac{l'}{q_n}}  \right\}  \]

Therefore, as in \cite{anosovkatok70,katokant}, we "monotonize" $\zeta_n$. This "monotonization" is performed in two steps: in the first step, we construct a partition $\zeta_n^{n+1}$ stable by $U_n$ and such that $\zeta_n^{n+1} \hookrightarrow \zeta_{n+1}$ (lemma \ref{lemmezetannplusun}). We also need that most elements of $\zeta_n^{n+1}$ have a size controlled independently of $q_{n+1}$, to ensure the smooth convergence of $T_n$ towards $T$.

In the second step, we iterate this procedure, so as to obtain a partition $\zeta_n^m$, such that $\zeta_n^m \hookrightarrow \zeta_m$ . At $m$ fixed, for $n<m$, $\zeta_n^m \hookrightarrow \zeta_{n+1}^m $. We take the limit $m \rightarrow + \infty$. This gives a monotonous partition $\zeta_n^\infty$ stable by $U_n$ (corollary \ref{corollaryzetanm}). This procedure is the same as in \cite{anosovkatok70,katokant}. Since $\zeta_n$ generates (lemma \ref{zngenerates}), then $\zeta_n^\infty$ too (corollary \ref{zninftygenerates}).

\bigskip

To show proposition \ref{summarypartitions}, we need some notations and auxiliary transformations. We slightly modify those introduced by De La Rue \cite[p. 395]{delarue96}.

Let $\theta=\arg B_1$, and let $G,D$ transformations of $\Omega$ defined by

\[ B_u \circ G= \sqrt{2} B_{\frac{u}{2}}   \]

\[ B_u \circ D= \sqrt{2} (B_{\frac{u+1}{2}} - B_{\frac{1}{2}})  \]

Since, by scaling, $B_u \circ G$ and $B_u \circ D$ are Brownian motions, then $G$ and $D$ are $\pr$-preserving. Moreover, $G^{-1}(\mathcal{A})$ and $D^{-1}(\mathcal{A})$ are independent. 

Let $\Gamma={G,D}$, and for any $n \in \varmathbb{N}$, $\Gamma^n$ is the set of words of length $n$ on the alphabet $\Gamma$. Let $\theta_G=\theta \circ G$, $\theta_D=\theta \circ D$, $B_G= B_1 \circ G$, $B_D= B_1 \circ D$, and having defined the random variables $\theta_W$ and $B_W$ for $W \in \Gamma^n$, we let $\theta_{GW}=\theta_W \circ G$, $\theta_{DW}=\theta_W \circ D$, $B_{GW}=B_W \circ G$, $B_{DW}=B_W \circ D$. Thus, we obtain families of random variables $\theta_W$ and $B_W$ defined by induction.

For $n \in \varmathbb{N}$, we denote $\mathcal{A}_n = \mathcal{B}(\theta_W, W \in \Gamma^n)$ and $\mathcal{B}_n= \vee_{k=0}^n \mathcal{A}_k$, $\mathcal{B}_\infty= \vee_{k=0}^{+\infty} \mathcal{A}_k$. For any finite measurable partition $\mathcal{N}=\{N_1,...,N_q\}$, $q \in \varmathbb{N}^*$, let  

 \[ c_n\left(\frac{i}{2^n},N_l \right)= \left\{ \omega \in \Omega / \arg\left( B_{\frac{i+1}{2^n}}(\omega)-B_{\frac{i}{2^n}}(\omega)\right)  \in N_l \right\}  \]

and let $\mathcal{A}_n(\mathcal{N})= \{ c_n\left(\frac{i}{2^n},N_l \right), i=0,...,2^n-1, l=0,...,q-1 \}$. In particular, if $N_l=[l/q,(l+1)/q[$, let \[ c_n\left(\frac{i}{2^n},\frac{l}{q}\right)= \left\{ \omega \in \Omega / \arg\left( B_{\frac{i+1}{2^n}}(\omega)-B_{\frac{i}{2^n}}(\omega)\right)  \in \left[ \frac{l}{q} , \frac{l+1}{q} \right[ \right\}  \]

and let $\mathcal{A}_n(q)= \{ c_n\left(\frac{i}{2^n},\frac{l}{q}\right), i=0,...,2^n-1, l=0,...,q-1 \}$, $\mathcal{B}_n(q)= \vee_{k=0}^n \mathcal{A}_k(q)$.

Let $t'_n$ such that $t_n=2^{t'_n}$. Let also

\[ \phi : \begin{array}[t]{lcl} \Gamma^{t'n}  &\rightarrow &   \left\{ \frac{k}{2^{t'n}}  , k=0,...,2^{t'_n}-1 \right\}  \\
               a_1...a_{t'_n}  & \mapsto    &  \sum_{i=1}^{t'_n} \frac{1_{a_i=D}}{2^{t'_n-(i+1)}} 
           \end{array}
           \]
Let

\[  \mathcal{A}_{W,p}(q) = \mathcal{B}\left(\arg\left(B_{\phi(W)+\frac{j+1}{2^{p+|W|}}}-B_{\phi(W)+\frac{j}{2^{p+|W|}}}\right) \in \left[ \frac{l}{q} , \frac{l+1}{q} \right[  ,l=0,...,q-1,  j=0,...,2^p-1\right)  \]

\[ \mathcal{B}_{W,p} (q)= \vee_{k=0}^n \mathcal{A}_{W,k}(q) \]

In particular, we have:

\[   \mathcal{A}_{p+1}(q)=  \mathcal{A}_{G,p}(q) \vee \mathcal{A}_{D,p}(q) \]

The main step in the proof is the following proposition:

\begin{proposition}
\label{lemmezetannplusun}
If assumptions $\ref{60rotisom}$, $\ref{61rotisom}$, $\ref{63rotisom}$, $\ref{65rotisom}$, $\ref{66rotisom}$, of lemma \ref{lemmefonda} hold, then for any $n\geq 0$, there exists a measurable partition $\zeta_n^{n+1} \hookrightarrow \zeta_{n+1}$ of $\Omega$, stable by $U_n$, there exists $Q_n^{n+1}: \zeta_n \rightarrow  \zeta_n^{n+1}$ isomorphism such that:

\begin{enumerate}
\item For any $c \in \zeta_n$, $\pr(c \Delta Q_n^{n+1} c )\leq \frac{1}{2^n q_n^{t_n} }$.
\item $Q_n^{n+1}U_n=U_n Q_n^{n+1}$. 
\item Elements of $\zeta_n^{n+1}$ are given by relation (\ref{defcnnplusun}).
\end{enumerate}

\end{proposition}

\begin{proof}[Proof of proposition \ref{lemmezetannplusun}.]

First, we need the lemma:
\begin{lemma}
\label{espanbn}
For any $p \geq 0$,$q>0$, $\pr$-almost surely,

\[ \esp\left[B_W |  \mathcal{A}_{W,p}(q) \right] = \esp\left[B_W |  \mathcal{B}_{W,p}(q) \right] \]

\end{lemma}

\begin{proof}

The word $W$ is fixed. For $u \in [0,1]$, let $B'_u=\sqrt{2^{|W|}} (B_{\frac{u}{ 2^{|W|} }+\phi(W)}  - B_{\phi(W)}  ) $. By scaling, $B'_u$ is another Brownian motion issued from $0$, and we have: $B'_1= \sqrt{2^{|W|}} B_W$. Let $\theta'= \arg B'_1$, and for $n \in \varmathbb{N}$, we denote $\mathcal{A'}_n(q) = \mathcal{B}\left(\theta'_{W'} \in [l/q,(l+1)/q[,l=0,...,q-1, W' \in \Gamma^n\right)$ and $\mathcal{B'}_n(q)= \vee_{k=0}^n \mathcal{A'}_k(q)$. We have $\mathcal{A}_{W,p} = \mathcal{A'}_p$, $\mathcal{B}_{W,p} = \mathcal{B'}_p$. Therefore, it suffices to show, $\pr$-almost surely:

\[ \esp\left[B'_1 |  \mathcal{A'}_{p}(q) \right] = \esp\left[B'_1 |  \mathcal{B'}_{p}(q) \right] \]

We show that $ \esp\left[B_1 |  \mathcal{A}_{p}(q) \right] = \esp\left[B_1 |  \mathcal{B}_{p}(q) \right] $ $\pr$-almost surely. Let 

\[ u_p= \esp\left[  \left| \esp[B_1| \mathcal{B}_{p}(q)] - \esp[B_1| \mathcal{A}_{p}(q)]  \right|^2 \right]  \]

We show that for any integer $p \geq 0$, $u_p=0$. Since $=\mathcal{A}_{0}(q)=\mathcal{B}_{0}(q)$, then $u_0=0$. It suffices to show that $u_{p+1}=u_p$. We follow the method of the proof of relation $(15)$ in \cite[p.397]{delarue96}. We have:

\[ u_{p+1}= \esp\left[ \left|  \esp\left[ B_{\frac{1}{2}} | \mathcal{B}_{p+1}(q) \right] -  \esp\left[ B_{\frac{1}{2}} | \mathcal{A}_{p+1}(q) \right] +  \esp\left[ B_1- B_{\frac{1}{2}} | \mathcal{B}_{p+1}(q) \right] -  \esp\left[ B_1 - B_{\frac{1}{2}} | \mathcal{A}_{p+1}(q) \right] \right|^2 \right]  \]

\[ u_{p+1}= \esp\left[  \frac{1}{2}  \left|  \esp\left[ B_{1} \circ G | \mathcal{B}_{p+1}(q) \right] -  \esp\left[   B_{1} \circ G | \mathcal{A}_{p+1}(q) \right] +  \esp\left[ B_{1} \circ D  | \mathcal{B}_{p+1}(q) \right] -  \esp\left[ B_{1} \circ D | \mathcal{A}_{p+1}(q) \right] \right|^2 \right]  \]

For any complex numbers $a,b$, $|a+b|^2= |a|^2+b^2+a\bar{b}+\bar{a}b$, where $\bar{z}$ denotes the conjugate of the complex number $z$. Therefore, 

\[ u_{p+1}= \frac{1}{2} \esp\left[   \left|  \esp\left[ B_{1} \circ G | \mathcal{B}_{p+1}(q) \right] -  \esp\left[ B_{1} \circ G | \mathcal{A}_{p+1}(q) \right] \right|^2  + \left| \esp\left[ B_{1} \circ D  | \mathcal{B}_{p+1}(q) \right] -  \esp\left[ B_{1} \circ D | \mathcal{A}_{p+1}(q) \right] \right|^2 \right]  \]

\[   + \frac{1}{2}  \esp\left[ \left(   \esp\left[ B_{1} \circ G | \mathcal{B}_{p+1}(q) \right] -  \esp\left[B_{1} \circ G | \mathcal{A}_{p+1}(q) \right] \right) \left( \overline{\esp\left[ B_{1} \circ D  | \mathcal{B}_{p+1}(q) \right]} - \overline{  \esp\left[ B_{1} \circ D | \mathcal{A}_{p+1}(q) \right]}  \right) \right]    \]

\[   + \frac{1}{2}  \esp\left[ \left(  \overline{\esp\left[ B_{1} \circ G | \mathcal{B}_{p+1}(q) \right]} -  \overline{\esp\left[B_{1} \circ G | \mathcal{A}_{p+1}(q) \right] } \right) \left( \esp\left[ B_{1} \circ D  | \mathcal{B}_{p+1}(q) \right] -  \esp\left[ B_{1} \circ D | \mathcal{A}_{p+1}(q) \right]  \right) \right]    \]

Now, we apply the claim:

\begin{claim}
\label{condindtribus}
Let $X$ a random variable on $(\Omega,\mathcal{A},\pr)$, and $\mathcal{G},\mathcal{H}$ sigma-algebras such that $\mathcal{B}(X,\mathcal{H})$ and $\mathcal{G}$ are independent. We have, a.e.: 

\[  \esp\left[X| \mathcal{G} \vee \mathcal{H} \right]= \esp\left[X| \mathcal{H} \right]   \]

\end{claim}

\begin{proof}

We reproduce the proof of \cite{hansen12}. We need to show that
$$
\esp[X\mid \mathcal{B}(\mathcal{G}\cup\mathcal{H})]=\esp[X\mid\mathcal{H}],
$$
that is, we need to show that $\esp[X\mid\mathcal{H}]$ can serve as the conditional expectation of $X$ given $\mathcal{B}(\mathcal{G}\cup\mathcal{H})$, i.e. show that

\begin{itemize}
\item $\esp[X\mid\mathcal{H}]$ is $\mathcal{B}(\mathcal{G}\cup\mathcal{H})$-measurable,
\item $\esp[X\mid\mathcal{H}]$ is integrable,
\item $\int_A \esp[X\mid\mathcal{H}]\,\mathrm d\pr=\int_A X\,\mathrm d\pr$ for all $A\in\mathcal{B}(\mathcal{G}\cup\mathcal{H})$.
\end{itemize}

The first two are obvious. For the third, let us note that (note that by linearity, we can assume that $X$ is non-negative)
$$
\mathcal{B}(\mathcal{G}\cup\mathcal{H})\ni A\mapsto \int_A\esp[X\mid\mathcal{H}]\,\mathrm d\pr
$$
and
$$
\mathcal{B}(\mathcal{G}\cup\mathcal{H})\ni A\mapsto \int_A X\,\mathrm d\pr
$$
are two measures defined on $\mathcal{B}(\mathcal{G}\cup\mathcal{H})$ with equal total mass being $\esp[X]$. Hence, it is enough to show that the two measures are identical on some $\cap$-stable generator of $\mathcal{B}(\mathcal{G}\cup\mathcal{H})$. Here, we use that
$$
\{A\cap B\mid A\in\mathcal{G},\,B\in\mathcal{H}\}
$$
is indeed a $\cap$-stable generator of $\mathcal{B}(\mathcal{G}\cup\mathcal{H})$. Therefore, it suffices to show that
$$
\int_{A\cap B} \esp[X\mid \mathcal{H}]\,\mathrm d\pr=\int_{A\cap B} X\,\mathrm d\pr,\quad A\in\mathcal{G},\,B\in\mathcal{H}
$$

This is true, because since $\mathcal{G}$ and $\mathcal{H}$ are independent,

\[ \esp\left[ 1_{A \cap B}   \esp[X\mid \mathcal{H}]  \right] = \esp\left[ 1_A 1_B   \esp[X\mid \mathcal{H}]  \right] = \esp[1_A] \esp[1_B \esp[X\mid \mathcal{H}]  ] \]

By the defining property of conditional expectation, $\esp[1_B \esp[X\mid \mathcal{H}]  ] = \esp[1_B X ]$. Moreover, $1_B X \in \mathcal{B}(X, \mathcal{H})$, which is independent of $\mathcal{G}$ by assumption. Therefore,

\[ \esp\left[ 1_{A \cap B}   \esp[X\mid \mathcal{H}]  \right] = \esp[1_A] \esp[1_B X]= \esp[1_A 1_B X] = \int_{A\cap B} X\,\mathrm d\pr \]

\end{proof}  

We have that $ \mathcal{A}_{p+1}(q)  = \mathcal{A}_{G,p}(q) \vee \mathcal{A}_{D,p}(q)$ and $ \mathcal{B}_{p+1}(q)  = \mathcal{B}_{G,p}(q) \vee \mathcal{B}_{D,p}(q)$. Moreover, the sigma-algebra $\mathcal{B}(B_G,\mathcal{B}_{G,p}(q))$ is independent of $\mathcal{B}_{D,p}(q)$. 

Indeed, $\mathcal{B}(B_G,\mathcal{B}_{G,p}(q)) \subset \mathcal{B}( B_{\frac{i+1}{2^p}} -B_{\frac{i}{2^p}}, i=0,...,2^{p-1}-1  )$ and 

$\mathcal{B}_{D,p}(q) \subset \mathcal{B}( B_{\frac{i+1}{2^p}} -B_{\frac{i}{2^p}}, i=2^{p-1},...,2^{p}-1  ) $, and these two sigma-algebras are independent, because the Brownian motion has mutually independent increments. Likewise, $\mathcal{B}(B_D,\mathcal{B}_{D,p}(q))$ is independent of $\mathcal{B}_{G,p}(q)$. Therefore, by claim \ref{condindtribus}, we get:

\[ u_{p+1}= \frac{1}{2} \esp\left[   \left|  \esp\left[ B_{1} \circ G | \mathcal{B}_{G,p}(q) \right] -  \esp\left[ B_{1} \circ G | \mathcal{A}_{G,p}(q) \right] \right|^2 \right. \] \[ \left. + \left| \esp\left[ B_{1} \circ D  | \mathcal{B}_{D,p}(q) \right] -  \esp\left[ B_{1} \circ D | \mathcal{A}_{D,p}(q) \right] \right|^2 \right]  \]

\[   + \frac{1}{2}  \esp\left[ \left(   \esp\left[ B_{1} \circ G | \mathcal{B}_{G,p}(q) \right] -  \esp\left[B_{1} \circ G | \mathcal{A}_{G,p}(q) \right] \right) \left( \overline{\esp\left[ B_{1} \circ D  | \mathcal{B}_{D,p}(q) \right]} - \overline{  \esp\left[ B_{1} \circ D | \mathcal{A}_{D,p}(q) \right]}  \right) \right]    \]

\[   + \frac{1}{2}  \esp\left[ \left(  \overline{\esp\left[ B_{1} \circ G | \mathcal{B}_{G,p}(q) \right]} -  \overline{\esp\left[B_{1} \circ G | \mathcal{A}_{G,p}(q) \right] } \right) \left( \esp\left[ B_{1} \circ D  | \mathcal{B}_{G,p}(q) \right] -  \esp\left[ B_{1} \circ D | \mathcal{A}_{D,p}(q) \right]  \right) \right]    \]

Since $\mathcal{B}_{G,p}(q) $ and $\mathcal{B}_{D,p}(q) $ are independent, we get:

\[ u_{p+1}= \frac{1}{2} \esp\left[   \left|  \esp\left[ B_{1} \circ G | \mathcal{B}_{G,p}(q) \right] -  \esp\left[ B_{1} \circ G | \mathcal{A}_{G,p}(q) \right] \right|^2 \right. \] \[ \left. + \left| \esp\left[ B_{1} \circ D  | \mathcal{B}_{D,p}(q) \right] -  \esp\left[ B_{1} \circ D | \mathcal{A}_{D,p}(q) \right] \right|^2 \right]  \]

Since $\mathcal{A}_{G,p}(q)= G^{-1}(\mathcal{A}_{p}(q))$ and  $\mathcal{B}_{G,p}= G^{-1}(\mathcal{B}_{p}(q))$, and since $G$ is measure-preserving, we obtain:

\[ \esp\left[   \left|  \esp\left[ B_{1} \circ G | \mathcal{B}_{G,p}(q) \right] -  \esp\left[ B_{1} \circ G | \mathcal{A}_{G,p}(q) \right] \right|^2 \right]  \] \[=  \esp\left[   \left|  \esp\left[ B_{1} \circ G | G^{-1} (\mathcal{B}_{p}(q)) \right] -  \esp\left[ B_{1} \circ G | G^{-1}(\mathcal{A}_{p}(q)) \right] \right|^2  \right]  \]

\[ = \esp\left[   \left|  \esp\left[ B_{1}  | \mathcal{B}_{p}(q) \right] -  \esp\left[ B_{1} | \mathcal{A}_{p}(q) \right] \right|^2 \circ G \right] = \esp\left[   \left|  \esp\left[ B_{1}  | \mathcal{B}_{p}(q) \right] -  \esp\left[ B_{1} | \mathcal{A}_{p}(q) \right] \right|^2 \right] =u_p  \]

Since $D$ is also measure-preserving, we can do the same for the right-hand side of the equation and therefore, $u_{p+1}= u_p$.

\end{proof}

\begin{lemma}
\label{lemmegeneratesbinfinity}
Let $j_k \rightarrow + \infty$ and $r_k \rightarrow + \infty$ two sequences of positive integers. We have:
\[  \bigvee_{k=0}^{+\infty} \mathcal{B}_{j_k}(r_k)=\mathcal{B}_\infty  \]

\end{lemma}

\begin{proof}

Since $r_k \rightarrow + \infty$, then for any $t \geq 0$, $u \geq 0$, $\mathcal{A}_t= \bigvee_{k=u}^{+ \infty} \mathcal{A}_t (r_k)$.

Since $j_k \rightarrow + \infty$, there exists $k_0(t)$ such that for any $k \geq k_0$, $j_k \geq t$. Moreover, for any $t\leq j_k$, $\mathcal{A}_t(r_k)\subset \mathcal{B}_{j_k}(r_k)$. Therefore,

\[ \mathcal{A}_t= \bigvee_{k=k_0}^{+ \infty} \mathcal{A}_t (r_k) \subset \bigvee_{k=k_0}^{+ \infty} \mathcal{B}_{j_k}(r_k) \subset \bigvee_{k=0}^{+ \infty} \mathcal{B}_{j_k}(r_k) \]

Therefore, 

\[ \mathcal{B}_\infty =  \bigvee_{t=0}^{+ \infty} \mathcal{A}_t \subset \bigvee_{k=0}^{+ \infty} \mathcal{B}_{j_k}(r_k) \]

\end{proof}

\begin{lemma}
\label{lemmekqk}
Almost surely, we have:
\[ \esp\left[ B_W| \mathcal{A}_{W,k}(q_n^k) \right]  \rightarrow_{k \rightarrow + \infty} B_W   \] 
\end{lemma}

\begin{proof}
By scaling, as in lemma \ref{espanbn}, it suffices to show that a.s.,

\[ \esp\left[ B_1| \mathcal{A}_{k}(q_n^k) \right]  \rightarrow_{k \rightarrow + \infty} B_1   \]

Theorem 3.2 in \cite[p.395]{delarue96} implies that $|B_1|$ is $\bigvee_{k=1}^{+\infty} \mathcal{A}_k$-measurable. Therefore, $B_1=|B_1|e^{i \arg B_1}$ is $\mathcal{B}_\infty$-measurable (indeed, \cite{delarue96} introduces $\tilde{\theta}=\arg\left( \frac{B_{\frac{1}{2}} - B_1}{B_{\frac{1}{2}}} \right)$ and $\tilde{\mathcal{A}}_n= \mathcal{B}(\tilde{\theta}_W, W \in \Gamma^n)$, $\tilde{\mathcal{B}}_n= \bigvee_{k=0}^n \tilde{\mathcal{A}}_k$.

Since $\tilde{\theta}= \theta_D+ \pi - \theta_G$, then $\tilde{\theta}$ is $\mathcal{A}_1$-measurable, and we have $\tilde{\mathcal{B}}_n \subset \bigvee_{k=1}^{n+1} \mathcal{A}_k$).

Therefore, \[ \esp[B_1 | \mathcal{B}_\infty]= B_1 \]

Let $\mathcal{B}'_k=\mathcal{B}_k(q_n^k)$. By lemma \ref{lemmegeneratesbinfinity}, $\bigvee_{k=0}^{+ \infty} \mathcal{B}'_k= \mathcal{B}_\infty$. Moreover, since $q_n^{k+1}$ is divided by $q_n^k$, then $\mathcal{B}'_{k+1} \subset \mathcal{B}'_k$, and $\mathcal{B}'_k$ is a filtration. Therefore, $\esp[B_1 | \mathcal{B}'_k]$ is a closed martingale, and almost surely,

\[  \esp[B_1 | \mathcal{B}'_k] \rightarrow_{k \rightarrow + \infty} \esp[B_1 | \mathcal{B}_\infty]  \]

Since, by lemma \ref{espanbn}, $ \esp[B_1 | \mathcal{B}'_k]= \esp\left[ B_1| \mathcal{A}_{k}(q_n^k) \right]$, and since $\esp[B_1 | \mathcal{B}_\infty]= B_1 $, we conclude that, almost surely

\[  \esp\left[ B_1| \mathcal{A}_{k}(q_n^k) \right]  \rightarrow_{k \rightarrow + \infty}   B_1   \]

\end{proof}

Let $\mathcal{Q}, \mathcal{N}$ be two measurable partitions. The \textit{safe zone} of $\mathcal{N}$ with respect to $\mathcal{Q}$ is defined by:

\[  S(\mathcal{Q},\mathcal{N})= \{ P \in \mathcal{N}/ \exists Q \in \mathcal{Q}, P \subset Q \} \]

For $k \geq 0$ integer, let \[ q^k= \{ [j/q_n^k,(j+1)/q_n^k[,j=0,...,q_n^k-1\}  \]

Let \[ S_k(q^k,\mathcal{N}) = \bigcap_{t=0}^{2^k-1} \left\{  \arg \left( B_{\frac{t+1}{2^k}}-B_{\frac{t}{2^k}} \right) \in S(q^k,\mathcal{N}) \right\}   \]

Lemma \ref{lemmekqk} gives the following corollary:

\begin{corollary}
\label{corebn}
For any $\epsilon>0$, any $\eta>0$, there exists $k_0(\epsilon,\eta)$ such that for any $k \geq k_0$, for any finite measurable partition $\mathcal{N}$ of $\varmathbb{T}$ such that $Leb(S(q^k,\mathcal{N})) \geq 1- \epsilon/2^{k}$, we have:

\[ \pr\left( \left|   \esp \left[ B_1 | \mathcal{A}_k(\mathcal{N}) \right]- B_1 \right| \leq \eta  \right) \geq 1- \epsilon  \]

\end{corollary}

\begin{proof}
By lemma \ref{lemmekqk}, since convergence almost sure implies convergence in probability, then for any $\epsilon>0$, any $\eta>0$, there exists $k_0(\epsilon,\eta)$ such that for any $k \geq k_0$,

\[ \pr\left( \left|   \esp \left[ B_1 | \mathcal{A}_k(q^k) \right]- B_1 \right| \leq \eta  \right) \geq 1- \epsilon  \]

Moreover, if  \[ \left|   \esp \left[ B_1 | \mathcal{A}_k(q^k) \right]- B_1 \right| \leq \eta \] then 

\[  \left|  \esp\left[    \esp \left[ B_1 | \mathcal{A}_k(q^k) \right]- B_1  | \mathcal{A}_k(\mathcal{N}) \right] \right|  \leq  \esp\left[ \left|   \esp \left[ B_1 | \mathcal{A}_k(q^k) \right]- B_1 \right| | \mathcal{A}_k(\mathcal{N}) \right]   \leq \eta \]

Moreover, if 

\[ \esp\left[    \esp \left[ B_1 | \mathcal{A}_k(q^k) \right]  | \mathcal{A}_k(\mathcal{N}) \right] =   \esp \left[ B_1 | \mathcal{A}_k(q^k) \right]    \]

then 

\[  \left|   \esp \left[ B_1 | \mathcal{A}_k(\mathcal{N}) \right]- B_1 \right| \leq  \left| \esp \left[ B_1 | \mathcal{A}_k(\mathcal{N}) \right] -       \esp\left[    \esp \left[ B_1 | \mathcal{A}_k(q^k) \right]  | \mathcal{A}_k(\mathcal{N}) \right]  \right| + \left| \esp \left[ B_1 | \mathcal{A}_k(q^k) \right] -B_1   \right|     \leq  2 \eta  \]

Moreover, \[ S(q^k,\mathcal{N}) \subset \left\{  \esp\left[    \esp \left[ B_1 | \mathcal{A}_k(q^k) \right]  | \mathcal{A}_k(\mathcal{N}) \right] =   \esp \left[ B_1 | \mathcal{A}_k(q^k) \right]    \right\} \]

Indeed, let

\[ \nu : \begin{array}[t]{lcl} \mathcal{Q}  &\rightarrow &  S(q^k,\mathcal{N}) \\
               Q & \mapsto    &  \{ P \in \mathcal{N}/ P \subset Q \}
           \end{array}
           \]

The map $\nu$ is surjective. Let $\mathcal{C}$ be the partition generating $\mathcal{A}_k(q^k)$. Elements $Q \in \mathcal{C}$ are of the form:

\[ Q= \bigcap_{j=0}^{2^k-1}   \left\{ \arg \left( B_{\frac{j+1}{2^k}}- B_{\frac{j}{2^k}}   \right) \in Q_j  \right\}  \]

with $Q_j \in q^k$. We have:

\[  1_{S_k(q^k,\mathcal{N})} \esp \left[ B_1 | \mathcal{A}_k(q^k) \right]= \sum_{Q \in \mathcal{A}_k(q^k)} \frac{\esp[B_1 1_Q]}{\pr(Q)} 1_{Q \cap  S_k(q^k,\mathcal{N})}  \]

Moreover, 

\[ Q \cap  S_k(q^k,\mathcal{N}) = \bigcap_{j=0}^{2^k-1}   \left\{ \arg \left( B_{\frac{j+1}{2^k}}- B_{\frac{j}{2^k}}   \right) \in \nu(Q_j)  \right\}  \]

which is $\mathcal{A}_k(\mathcal{N})$-measurable. Therefore, 

 \[ S(q^k,\mathcal{N}) \subset \left\{  \esp\left[    \esp \left[ B_1 | \mathcal{A}_k(q^k) \right]  | \mathcal{A}_k(\mathcal{N}) \right] =   \esp \left[ B_1 | \mathcal{A}_k(q^k) \right]    \right\} \]

Moreover, 

\[  \pr\left( S_k(q^k,\mathcal{N}) \right) \geq \epsilon  \]

Therefore, 

\[   \pr\left(  \left| \esp \left[ B_1 | \mathcal{A}_k(\mathcal{N}) \right]- B_1 \right| \leq 2 \eta \right) \]  \[\geq   \pr\left( \left\{ \left|   \esp \left[ B_1 | \mathcal{A}_k(q^k) \right]- B_1 \right| \leq \eta \right\} \cap   \left\{  \esp\left[    \esp \left[ B_1 | \mathcal{A}_k(q^k) \right]  | \mathcal{A}_k(\mathcal{N}) \right] =   \esp \left[ B_1 | \mathcal{A}_k(q^k) \right]  \right\} \right)   \]

\[ \geq  \pr\left( \left\{ \left|   \esp \left[ B_1 | \mathcal{A}_k(q^k) \right]- B_1 \right| \leq \eta \right\} \cap   S_k(q^k,\mathcal{N})  \right)   \]

Therefore,

\[    \pr\left( \left| \esp \left[ B_1 | \mathcal{A}_k(\mathcal{N}) \right]- B_1 \right| \leq 2 \eta \right) \geq 1-2\epsilon \]

\end{proof}

\begin{corollary}
\label{corarg}
For any $\epsilon>0$, any $\pi/2>\eta>0$, there exists $k_0(\epsilon,\eta)$ such that for any $k \geq k_0$, any finite partition $\mathcal{N}$ of $\varmathbb{T}$ such that $Leb \left(S(q^k,\mathcal{N})\right) \geq 1-\epsilon/2^k$,

\[ \pr\left( \left|  \arg\left( \esp \left[ B_1 | \mathcal{A}_k(\mathcal{N}) \right]\right)  - \arg(B_1) \right| \leq \eta  \right) \geq 1- \epsilon  \]

\end{corollary}

\begin{proof}
We need the claim:

\begin{claim}
\label{estz}
For any $z \in \varmathbb{C}$ such that $|z-1|<1$, we have:
\[ |z-1| \geq |\sin (\arg z)| \geq \frac{2}{\pi}|\arg z |  \]
\end{claim}

\begin{proof}
If $|z-1|<1$ then $|\arg z | \leq \pi/2$ and so $|\sin (\arg z)| \geq \frac{2}{\pi}|\arg z | $. For the other estimate, we write $z=re^{i\theta}$. Since $(r-\cos \theta)^2 \geq 0$, then 

\[ r^2+1-\sin^2 \theta-2r \cos \theta \geq 0 \]

\[  r^2 \cos^2 \theta +1 -2r \cos \theta + r^2 \sin^2 \theta \geq \sin^2 \theta \]

\[ (r\cos \theta -1 )^2+ (r \sin \theta)^2 \geq (\sin \theta )^2  \]

\[  |z-1 | \geq |\sin (\arg z)| \]

\end{proof}

Let $\epsilon>0$. There exists $g(\epsilon)>0$ such that 

\[ \pr\left(|B_1|\geq g(\epsilon) \right) \geq 1- \epsilon  \]

Let $\pi/2 >\eta >0$ and $k_0=k_0\left(\epsilon,\frac{2g(\epsilon)\eta}{\pi}\right)$ in corollary \ref{corebn}. If $|B_1|\geq g(\epsilon)$ and if 

\[  \left| \esp \left[ B_1 | \mathcal{A}_k(\mathcal{N}) \right]- B_1 \right| \leq \frac{2g(\epsilon)\eta }{\pi} \]
 
then

\[  \left| \frac{\esp \left[ B_1 | \mathcal{A}_k(\mathcal{N}) \right]}{B_1}-1 \right| \leq \frac{2g(\epsilon)\eta }{\pi|B_1|} \leq \frac{2\eta}{\pi}<1  \]
 
By claim \ref{estz}, we get:

\[  \frac{2}{\pi} \left| \arg \left( \frac{\esp \left[ B_1 | \mathcal{A}_k(\mathcal{N}) \right]}{B_1}  \right)  \right| \leq \frac{2}{\pi} \eta \]

Therefore, 

\[  \left| \arg \left( \esp \left[ B_1 | \mathcal{A}_k(\mathcal{N}) \right]  \right) - \arg(B_1) \right| \leq  \eta \]

Therefore,
\[ \pr\left( \left|  \arg\left( \esp \left[ B_1 | \mathcal{A}_k(\mathcal{N}) \right]\right)  - \arg(B_1) \right| \leq \eta  \right) \geq  \pr\left( \{|B_1| \geq g(\epsilon) \} \cap \{  \left| \esp \left[ B_1 | \mathcal{A}_k(\mathcal{N}) \right]- B_1 \right| \leq  \frac{2g(\epsilon)\eta }{\pi}  \} \right) \]

By corollary \ref{corebn}, we get:

\[ \pr\left( \left|  \arg\left( \esp \left[ B_1 | \mathcal{A}_k(\mathcal{N}) \right]\right)  - \arg(B_1) \right| \leq \eta  \right) \geq  1-\epsilon-\epsilon=1-2\epsilon \]

\end{proof}

\begin{corollary}
\label{cordelta}
For any $0< \epsilon< 1/(2q_n)$, there exists $k(\epsilon,t_n,q_n)$ such that for any $l=0,...,q_n-1$, for any $W \in \Gamma^{t_n}$, for any finite measurable partition $\mathcal{N}$ such that

\[ Leb\left( S(q^{k_0},\mathcal{N})\right) \geq 1-\frac{\epsilon}{2^{k}}   \]

\[ \pr\left( \left( \arg B_W  \in \left[\frac{l}{q_n}, \frac{l+1}{q_n} \right[ \right) \Delta \left( \arg\esp\left[ B_W| \mathcal{A}_{W,k}(\mathcal{N}) \right] \in \left[\frac{l}{q_n}, \frac{l+1}{q_n} \right[ \right) \right) \leq \epsilon \]

\end{corollary}

\begin{proof}

More generally, corollary \ref{corarg} holds for any $B_W$. Let $k(\epsilon,t_n,q_n)= \max_{W \in \Gamma^{t'_n}} k_0(\epsilon,\epsilon,W)$. 

We have:

\[ \pr\left[ \left( \arg B_W  \in \left[\frac{l}{q_n}, \frac{l+1}{q_n} \right[ \right) \Delta \left( \arg\esp\left[ B_W| \mathcal{A}_{W,k}(\mathcal{N}) \right] \in \left[\frac{l}{q_n}, \frac{l+1}{q_n} \right[ \right) \right) \] \[= \pr\left(  \left( \arg B_W  \in \left[\frac{l}{q_n}, \frac{l+1}{q_n} \right[ \right)^c \cap \left( \arg\esp\left[ B_W| \mathcal{A}_{W,k}(\mathcal{N}) \right] \in \left[\frac{l}{q_n}, \frac{l+1}{q_n} \right[ \right) \right] +  \] 
 
\[  \pr\left[  \left( \arg B_W  \in \left[\frac{l}{q_n}, \frac{l+1}{q_n} \right[ \right) \cap \left( \arg\esp\left[ B_W| \mathcal{A}_{W,k}(\mathcal{N}) \right] \in \left[\frac{l}{q_n}, \frac{l+1}{q_n} \right[ \right)^c \right] \]

 \[ \leq  \pr\left[  \left( \arg B_W  \in \left[\frac{l}{q_n}, \frac{l+1}{q_n} \right[ \right)^c \cap \left( \arg\esp\left[ B_W| \mathcal{A}_{W,k}(\mathcal{N}) \right] \in \left[\frac{l}{q_n}, \frac{l+1}{q_n} \right[ \right) \right. \] \[ \left.  \cap  \left(  \left| \arg B_W - \arg\esp\left[ B_W| \mathcal{A}_{W,k}(\mathcal{N}) \right] \right|  <\epsilon \right)  \right]   \] 
 
\[+  \pr\left[  \left( \arg B_W  \in \left[\frac{l}{q_n}, \frac{l+1}{q_n} \right[ \right) \cap \left( \arg\esp\left[ B_W| \mathcal{A}_{W,k}(\mathcal{N}) \right] \in \left[\frac{l}{q_n}, \frac{l+1}{q_n} \right[ \right)^c    \right. \] \[ \left. \cap \left(  \left| \arg B_W - \arg\esp\left[ B_W| \mathcal{A}_{W,k}(\mathcal{N}) \right] \right|  <\epsilon \right)  \right] \] 
 
\[+ 2  \pr\left[  \left| \arg B_W - \arg\esp\left[ B_W| \mathcal{A}_{W,k}(\mathcal{N}) \right] \right|  \geq \epsilon \right]       \] 
 
Moreover, since $\epsilon< 1/(2q_n)$ then:

\[   \left( \arg B_W  \in \left[\frac{l}{q_n}, \frac{l+1}{q_n} \right[ \right)^c \cap \left( \arg\esp\left[ B_W| \mathcal{A}_{W,k}(\mathcal{N}) \right] \in \left[\frac{l}{q_n}, \frac{l+1}{q_n} \right[ \right)    \cap  \left(  \left| \arg B_W - \arg\esp\left[ B_W| \mathcal{A}_{W,k}(\mathcal{N}) \right] \right|  <\epsilon \right) \] \[ \subset \left( \arg B_W \in  \left[-\epsilon + \frac{l}{q_n}, \frac{l}{q_n} \right[ \cup \left[ \frac{l+1}{q_n}, \frac{l+1}{q_n} +\epsilon \right[  \right) \]
 
Therefore,

\[   \pr\left(  \left( \arg B_W  \in \left[\frac{l}{q_n}, \frac{l+1}{q_n} \right[ \right)^c \cap \left( \arg\esp\left[ B_W| \mathcal{A}_{W,k}(\mathcal{N}) \right] \in \left[\frac{l}{q_n}, \frac{l+1}{q_n} \right[ \right)    \cap  \left(  \left| \arg B_W - \arg\esp\left[ B_W| \mathcal{A}_{W,k}(\mathcal{N}) \right] \right|  <\epsilon \right)  \right) \] \[ \leq 2 \epsilon \]

Likewise,

\[   \pr\left(  \left( \arg B_W  \in \left[\frac{l}{q_n}, \frac{l+1}{q_n} \right[ \right) \cap \left( \arg\esp\left[ B_W| \mathcal{A}_{W,k}(\mathcal{N}) \right] \in \left[\frac{l}{q_n}, \frac{l+1}{q_n} \right[ \right)^c    \cap  \left(  \left| \arg B_W - \arg\esp\left[ B_W| \mathcal{A}_{W,k}(\mathcal{N}) \right] \right|  <\epsilon \right)  \right) \] \[ \leq 2 \epsilon \]

Therefore,

\[ \pr\left( \left( \arg B_W  \in \left[\frac{l}{q_n}, \frac{l+1}{q_n} \right[ \right) \Delta \left( \arg\esp\left[ B_W| \mathcal{A}_{W,k}(\mathcal{N}) \right] \in \left[\frac{l}{q_n}, \frac{l+1}{q_n} \right[ \right) \right) \leq 6 \epsilon  \]

\end{proof}

Let $\epsilon= \frac{1}{2^n t_n q_n^{t_n}} $. By corollary \ref{cordelta}, we can fix $k_0(n,q_n,t_n)$ such that for any $l=0,...,q_n-1$, $W \in \Gamma^{t'_n}$, for any finite measurable partition $\mathcal{N}$ such that $Leb\left( S(q^{k_0},\mathcal{N})\right) \geq 1-\frac{\epsilon}{2^{k_0}}$, we have:

\begin{equation}
\label{estcncnplusun}
\pr\left( \left( \arg B_W  \in \left[\frac{l}{q_n}, \frac{l+1}{q_n} \right[ \right) \Delta  \left( \arg\esp\left[ B_W| \mathcal{A}_{W,k_0}(\mathcal{N}) \right] \in \left[\frac{l}{q_n}, \frac{l+1}{q_n} \right[ \right) \right)  \leq \frac{1}{2^n t_n q_n^{t_n}}   
\end{equation}

Let \[  v_{n+1} = \left\lfloor  \frac{q_{n+1}}{2^{k_0} 2^{2n+1} t_n q_n^{t_n} q_n^{k_0} }  \right\rfloor   \]

For $q_{n+1}$ sufficiently large, $v_{n+1}\geq 1$. By Euclidean division, we can write:

\begin{equation}
\label{defynrhon}
\frac{q_{n+1}}{q_n}=  v_{n+1}  y_{n+1} +  \rho_{n+1} 
\end{equation} 

with $0 \leq \rho_{n+1}<  v_{n+1}$.

Let \[  \mathcal{N}_n^{n+1}= \left\{ \frac{\gamma}{q_n} + j  \frac{v_{n+1}}{ q_{n+1}} + N_j,j=0,..., y_{n+1}, \gamma=0,...,q_n-1  \right\} \]

with $N_j=[0, v_{n+1}/ q_{n+1}[$, for $j=0,...,y_{n+1}-1$, and $N_{y_{n+1}}=[0, \rho_{n+1}/ q_{n+1}[$.

We have: $ Leb\left( S(q^{k_0},\mathcal{N}_n^{n+1})\right) \geq 1-\frac{\epsilon}{2^{k_0}}$ and therefore, estimate (\ref{estcncnplusun}) holds for $\mathcal{N}_n^{n+1}$.

When $q_{n+1}$ varies, the partition $\mathcal{N}_n^{n+1}$ varies, but $k_0$ must remain fixed independently of $q_{n+1}$, and this is why we need a uniform estimate on all partitions $\mathcal{N}$ sufficiently refined. If we could take $q_n^{k_0}$ dividing $q_{n+1}$, a non-uniform estimate in corollary \ref{cordelta} would be enough.

For any $l=0,...,q_n-1$, $W \in \Gamma^{t'_n}$, let

\begin{equation}
\label{defcnnplusun}
c_n^{n+1} \left( \phi(W), \frac{l}{q_n}\right) =  \left\{ \omega \in \Omega / \arg\esp\left[ B_W| \mathcal{A}_{W,k_0}(\mathcal{N}_n^{n+1}) \right](\omega) \in \left[\frac{l}{q_n}, \frac{l+1}{q_n} \right[ \right\}  
\end{equation} 

Let 

\[ \zeta_{n,\phi(W)}^{n+1}= \left\{ c_n^{n+1}\left( \phi(W), \frac{l}{q_n}\right), l=0,...,q_n-1  \right\}  \]

The set $\zeta_{n,\phi(W)}^{n+1}$ is a partition of $\Omega$, because it is the pre-image of a partition of $\varmathbb{T}$. For any $i=0,...,t_{n+1}-1$, $j_i=0,...,y_{n+1}$, $\gamma_i=0,...,q_n-1$, let 

\[  \tilde{e}\left(\frac{i}{t_{n+1}}, \frac{\gamma_i}{q_n}+ j_i \frac{v_{n+1}}{q_{n+1}} \right)= \left\{ \omega \in \Omega/ \arg\left(B_{\frac{i+1}{t_{n+1}}}(\omega)-B_{\frac{i}{t_{n+1}}}(\omega)\right) \in \frac{\gamma_i}{q_n} + j_i  \frac{v_{n+1}}{ q_{n+1}} + N_{j_i} \right\} \]

The sets $\tilde{e}\left(\frac{i}{t_{n+1}}, \frac{\gamma_i}{q_n}+ j_i \frac{v_{n+1}}{q_{n+1}} \right)$, $i=0,...,t_{n+1}-1$, are mutually independent. For any $i'=0,...,t_n-1$, $ \gamma= (\gamma_{i'\frac{t_{n+1}}{t_n}}, \gamma_{i'\frac{t_{n+1}}{t_n}+1},...,\gamma_{(i'+1)\frac{t_{n+1}}{t_n}- \frac{1}{t_n}})$, $j=(j_{i'\frac{t_{n+1}}{t_n}}, j_{i'\frac{t_{n+1}}{t_n}+1},...,j_{(i'+1)\frac{t_{n+1}}{t_n}- \frac{1}{t_n}})$, let

\[ e\left(\frac{i'}{t_{n}}, \frac{\gamma}{q_n}+ j \frac{v_{n+1}}{q_{n+1}} \right) = \bigcap_{\frac{i}{t_{n+1}}= \frac{i'}{t_n} }^{ \frac{i'+1}{t_n}- \frac{1}{t_{n+1}}}  \tilde{e}\left(\frac{i}{t_{n+1}}, \frac{\gamma_i}{q_n}+ j_i \frac{v_{n+1}}{q_{n+1}}  \right)    \]
Let 

\[  \mathcal{P}(i')= \left\{   e\left(\frac{i'}{t_{n}},\frac{\gamma}{q_n}+ j \frac{v_{n+1}}{q_{n+1}} \right),\gamma \in \{ 0,...,q_n-1\}^{\frac{t_{n+1}}{t_n}}, j \in  \{0,..., y_{n+1}  \}^{\frac{t_{n+1}}{t_n}} \right\} \]

The partition $\mathcal{P}(i')$ is finite and $\mathcal{B}(\mathcal{P}(i'))= \mathcal{A}_{W,k_0}(\mathcal{N})$. Therefore,

\[  \arg\esp\left[ B_W| \mathcal{A}_{W,k_0}(\mathcal{N}_n^{n+1})\right]= \arg\left( \sum_{e \in  \mathcal{P}(i') } \frac{\esp[B_W 1_e]}{\pr(e)}1_e \right)=  \sum_{e \in  \mathcal{P}(i') }\left(  \arg\esp[B_W 1_e] \right) 1_e   \]

(we can check that the right-hand term satisfies the characteristic property of conditional expectation)

\bigskip

For $i'=0,...,t_n-1$, $ l=0,...,q_n-1$, let 

\[
E\left(  \frac{i'}{t_n}, \frac{l}{q_n} \right)= \left\{ (\gamma,j) \in \{ 0,...,q_n-1\}^{\frac{t_{n+1}}{t_n}} \times  \{0,..., y_{n+1}  \}^{\frac{t_{n+1}}{t_n}}  / \right. \] 
\begin{equation}
\label{defdee}
\left. \arg\left(  \esp\left[ \left( B_{\frac{i'+1}{t_{n}}}-B_{\frac{i'}{t_{n}}} \right) 1_{e\left(\frac{i'}{t_{n}},  \frac{\gamma}{q_n}+ j \frac{v_{n+1}}{q_{n+1}}  \right) } \right] \right) \in \left[ \frac{l}{q_n}, \frac{l+1}{q_n} \right[  \right\} 
\end{equation}

Therefore,

\begin{equation}
\label{defbonneunion}
c_{n}^{n+1}(\phi(W), \frac{l}{q_n}) = \bigcup_{ (\gamma,j) \in E\left(  \phi(W), \frac{l}{q_n} \right) } e\left(\phi(W), \frac{\gamma}{q_n} + j  \frac{v_{n+1}}{ q_{n+1}} \right)
\end{equation}

Moreover, since $\mathcal{N}_n^{n+1} \hookrightarrow \{ [k/q_{n+1},(k+1)/q_{n+1}[,k=0,...,q_{n+1}-1\} $, 

then $\mathcal{A}_{W,k_0}(\mathcal{N}_n^{n+1}) \subset \mathcal{B}(\zeta_{n+1})$, and therefore, $\zeta_{n,\phi(W)}^{n+1} \hookrightarrow \zeta_{n+1} $.

\begin{lemma}
\label{qnnplusuni}
Let 

\[ Q_{n,\phi(W)}^{n+1} : \begin{array}[t]{lcl} \zeta_{n,\phi(W)}  &\rightarrow &   \zeta_{n,\phi(W)}^{n+1}  \\
               c_{n}(\phi(W),\frac{l}{q_n}) & \mapsto    & c_{n}^{n+1}(\phi(W), \frac{l}{q_n})
           \end{array}
           \]
We have: $Q_{n,\phi(W)}^{n+1} U_{n|\zeta_{n,\phi(W)}}= U_{n|\zeta_{n,\phi(W)}} Q_{n,\phi(W)}^{n+1} $. Moreover, $Q_{n,\phi(W)}^{n+1} $ is measure-preserving.
\end{lemma}

\begin{proof}
We have:

\[  U_n Q_{n,\phi(W)}^{n+1} (c_{n}(\phi(W), \frac{l}{q_n})) = \left\{ U_n(\omega) / \arg\esp\left[ B_W| \mathcal{A}_{W,k_0}(\mathcal{N}_n^{n+1}) \right](\omega) \in \left[\frac{l}{q_n}, \frac{l+1}{q_n} \right[ \right\}  \]

\[  =  \left\{ \omega \in \Omega / \arg\esp\left[ B_W| \mathcal{A}_{W,k_0}(\mathcal{N}_n^{n+1}) \right] \circ U_n^{-1}(\omega) \in \left[\frac{l}{q_n}, \frac{l+1}{q_n} \right[ \right\}  \]

\[  =  \left\{ \omega \in \Omega / \arg\esp\left[ B_W\circ U_n^{-1} | U_n(\mathcal{A}_{W,k_0}(\mathcal{N}_n^{n+1})) \right](\omega) \in \left[\frac{l}{q_n}, \frac{l+1}{q_n} \right[ \right\}  \]

Moreover, since $R_{\frac{1}{q_n}}(\mathcal{N}_n^{n+1})= \mathcal{N}_n^{n+1}$, then $U_n(\mathcal{A}_{W,k_0}(\mathcal{N}_n^{n+1}))=\mathcal{A}_{W,k_0}(\mathcal{N}_n^{n+1})$. On the other hand,

\[B_W\circ U_n^{-1} =  e^{-i \frac{p_n}{q_n} b_{n}(\phi(W)) } B_W \]

Therefore,

\[  U_n Q_{n,\phi(W)}^{n+1} \left(c_{n}\left(\phi(W), \frac{l}{q_n}\right)\right) \]

\[ = \left\{ \omega \in \Omega / \arg\esp\left[ B_W| \mathcal{A}_{W,k_0}(\mathcal{N}_n^{n+1}) \right](\omega) \in \left[  \frac{p_n b_{n}(\phi(W))}{q_n} + \frac{l}{q_n},  \frac{p_n b_{n}(\phi(W))}{q_n} + \frac{l+1}{q_n} \right[ \right\}  \] \[  = c_{n}\left(\phi(W), \frac{p_n b_{n}(\phi(W))}{q_n}+\frac{l}{q_n}\right)  \]

On the other hand,

\[   Q_{n,\phi(W)}^{n+1} U_n\left(c_{n}\left( \phi(W), \frac{l}{q_n}\right)\right) =  Q_{n,\phi(W)}^{n+1} \left(c_{n}\left( \phi(W), \frac{p_n b_{n}(\phi(W))}{q_n}+ \frac{l}{q_n}\right)\right) \]

\[= c_{n}^{n+1} \left( \phi(W), \frac{p_n b_{n}(\phi(W))}{q_n}+ \frac{l}{q_n}\right) = U_n Q_{n,\phi(W)}^{n+1} \left(c_{n}\left( \phi(W), \frac{l}{q_n}\right)\right)  \]

Finally, $Q_{n,\phi(W)}^{n+1} $ is measure-preserving: indeed, let $\tilde{U}_n$ be the transformation defined like $U_n$, but with $p_n=b_{n}(\phi(W))=1$ (we use $\tilde{U}_n$ instead of $U_n$ in order not to use the assumption $\gcd(p_nb_{n}(\phi(W)), q_n) =1$). We have:  

\[ \zeta_{n,\phi(W)}^{n+1} = \left\{ \tilde{U}_n^l\left( c_{n}^{n+1}\left( \phi(W),0\right)\right), l=0,...,q_n-1  \right\}      \]

Since, like $U_n$, $\tilde{U}_n$ is measure-preserving, then all the elements of $\zeta_{n,\phi(W)}^{n+1} $ have the same measure. Therefore,

\[  \pr\left(c_{n}^{n+1}\left(\phi(W), \frac{l}{q_n}\right)\right) = \frac{1}{q_n}= \pr\left(c_{n}\left( \phi(W), \frac{l}{q_n}\right)\right)   \]

Therefore, $Q_{n,\phi(W)}^{n+1} $ is measure-preserving.
\end{proof}

To conclude, we let $\zeta_n^{n+1}= \bigvee_{i=0}^{t_n-1} \zeta_{n,\frac{i}{t_n}}^{n+1}$, and

\[ Q_{n}^{n+1} : \begin{array}[t]{lcl} \zeta_{n}  &\rightarrow &   \zeta_n^{n+1}  \\
               c_{n}\left(0,\frac{l_0}{q_n}\right) \cap...\cap c_{n}\left(\frac{t_n-1}{t_n},\frac{l_{t_n-1}}{q_n}\right) & \mapsto    & 
               Q_{n,0}^{n+1} \left(c_{n}\left(0,\frac{l_0}{q_n}\right)\right) \cap ...\cap  Q_{n,\frac{t_n-1}{t_n}}^{n+1} \left(c_{n}\left(\frac{t_n-1}{t_n},\frac{l_{t_n-1}}{q_n}\right)\right)
           \end{array}
           \]

By independence of the $\zeta_{n,\frac{i}{t_n}}^{n+1}$ and by lemma \ref{qnnplusuni}, $\zeta_n^{n+1}$, stable by $U_n$, and $Q_n^{n+1}: \zeta_n \rightarrow  \zeta_n^{n+1}$ is a measure-preserving isomorphism such that $Q_n^{n+1}U_n=U_n Q_n^{n+1}$.

Moreover, since, by estimation (\ref{estcncnplusun}) and definition (\ref{defcnnplusun}), we have, for any $i=0,...,t_n-1$, $l=0,...,q_n-1$,

\[ \pr\left(c_{n}\left( \frac{i}{t_n}, \frac{l}{q_n}\right) \Delta  c_{n}^{n+1}\left( \frac{i}{t_n}, \frac{l}{q_n}\right) \right) \leq \frac{1}{2^nt_n q_n^{t_n}}  \]

and since, for any $A,A',B,B' \in \Omega$, \[ \pr((A \cap A') \Delta (B \cap B')) \leq \pr(A \Delta B)+ \pr(A'\Delta B') \]

then we have, for any $c \in \zeta_n$, \[ \pr(c \Delta Q_n^{n+1} c )\leq \frac{1}{2^n q_n^{t_n} } \]


\end{proof}

From here, the rest of the proof is analogous to \cite[section 2]{katokant}, except lemma \ref{zngenerates}.

\begin{corollary}
\label{corollaryzetanm}
If assumptions $\ref{60rotisom}$, $\ref{61rotisom}$, $\ref{63rotisom}$, $\ref{65rotisom}$ and $\ref{66rotisom}$ of lemma \ref{lemmefonda} hold, at $n$ fixed, there exists measurable partitions $(\zeta_n^m)_{n\geq 0, n <m}$ of $\Omega$, such that $\zeta_n^m$ is stable by $U_n$, and such that at $m$ fixed, for $n<m$, $\zeta_{n+1}^m \hookrightarrow  \zeta_n^m$. Moreover, at $n$ fixed, $\zeta_n^m$ converges to a partition $\zeta_n^\infty$ as $m \rightarrow + \infty$, such that $\zeta_n^\infty$ is monotonous and stable by $U_n$.
\end{corollary}

\begin{proof}

For $m>n$, let

\begin{equation}
\label{defqmn}
Q^m_n= Q^m_{m-1}...Q^{n+1}_n
\end{equation}

and let $\zeta_n^m= Q^m_n (\zeta_n)$. For $n \leq m-1$, $\zeta_n^m \hookrightarrow  \zeta_{n+1}^m$. Moreover, $\zeta_n^m$ is stable by $U_n$: we recall that $b(n)=(b_n(0/t_n),...,b_n((t_n-1)/t_n)$. Let $\bar{U}_n$ be defined as $U_n$, but by taking $p_n=1$. Since $b(p+1)/q_p=b(p)/q_p$ mod $\varmathbb{Z}$ for any $p \geq n$, and since $q_n$ divides $q_p$, then $\frac{q_p}{q_n} b(p)/q_p=b(p)/q_n= b(n)/q_n$ mod $\varmathbb{Z}$ by assumption $\ref{65rotisom}$. Since $\bar{U}_p$ commutes with $Q^{p+1}_p$, then $\bar{U}_n= \frac{q_p}{q_n} \bar{U}_p$ also commutes with $Q^{p+1}_p$, and therefore, $\bar{U}_n$ also commutes with  $Q_n^m=Q_n^{n+1}Q^{n+2}_{n+1}...Q^m_{m-1}$. Therefore, $U_n=\bar{U}_n^{p_n}$ commutes with $Q_n^m$, and stabilizes $\zeta_n^m$.

We show that $(\zeta_n^m)_{m \geq n} $ is a Cauchy sequence for the metric on measurable partitions. For any fixed $n$, $n<m$, we have:

\[ d(\zeta_n^m,\zeta_n^{m+1})= \sum_{ c^m_{n} \in \zeta_n^m} \pr \left( c^m_{n}  \Delta  Q^{m+1}_m (c^m_{n})   \right) =   \sum_{ c \in \zeta_n} \pr \left( Q_n^m (c)  \Delta  Q^{m+1}_m Q_n^m (c)   \right)   \]

For $c \in \zeta_n$, $Q_n^m (c) \subset \zeta_m$, so we can write:

\[  Q_n^m (c)= \bigcup_{i \in I(c)} c_i \]

with $c_i \in \zeta_m$. By volume conservation of the map $Q_n^m$, $|I(c)|=q_m^{t_m}/q_n^{t_n}$. Therefore, by applying proposition \ref{lemmezetannplusun},

\[    \pr \left( Q_n^m (c)  \Delta  Q^{m+1}_m Q_n^m (c)   \right)  =     \pr \left( \bigcup_{i \in I(c)} c_i \Delta  Q^{m+1}_m \left( \bigcup_{i \in I(c)} c_i \right)  \right) =     \pr \left( \bigcup_{i \in I(c)} c_i \Delta \left( \bigcup_{i \in I(c)} Q^{m+1}_m  c_i \right)  \right) \]

\[   \leq  \sum_{i \in I(c)}  \pr \left( c_i \Delta Q^{m+1}_m  c_i \right) \leq  \frac{q_m^{t_m}}{q_n^{t_n}} \frac{1}{2^m q_m^{t_m}} \]

Therefore,

\[ d(\zeta_n^m,\zeta_n^{m+1}) =   \sum_{ c \in \zeta_n} \pr \left( Q_n^m (c)  \Delta  Q^{m+1}_m Q_n^m (c)   \right) \leq \frac{1}{2^m}  \]

\begin{equation}
\label{summn}
\sum_{m \geq n } d(\zeta_n^m,\zeta_n^{m+1}) \leq \frac{1}{2^n}
\end{equation}

Therefore, $(\zeta_n^m)_{m \geq n} $ is a Cauchy sequence. Let $\zeta_n^\infty$ its limit. Let $Q_n^\infty$ the limit of $Q_n^m$.

Now, we show that $\zeta_n^\infty$ is monotonous. Let $n\geq 0$ and $\epsilon>0$. Let $m>n$ such that $d(\zeta_n^m, \zeta_n^\infty) <\epsilon/2$ and $d(\zeta_{n+1}^m, \zeta_{n+1}^\infty) <\epsilon/2$. Let $c_{n}\left(k\right) \in \zeta_n$, and $c_{n}^\infty(k)=Q_n^\infty (c_{n}(k)) \in \zeta_n^\infty$, where $k=(\frac{i}{t_n},\frac{l}{q_n}),i=0,...,t_n-1$ and $l=0,...,q_n-1$. Let $m \geq n$ such that $\pr\left(c_{n}^m(k) \Delta c_{n}^\infty(k)\right) \leq \epsilon/2$. 

Since $\zeta_n^m \hookrightarrow \zeta_{n+1}^m$, we can write: 

\[ c^m_{n}(k) = \bigcup_{l \in L(k)} c^m_{n+1}(l) \]


Since $d(\zeta_{n+1}^m, \zeta_{n+1}^\infty) <\epsilon/2$, we have:

\[   \pr\left(\bigcup_{l \in L(k)} c^m_{n+1}(l) \Delta \bigcup_{l \in L} c^\infty_{n+1}(l)  \right) \leq \sum_{l \in L}  \pr\left( c^m_{n+1}(l)  \Delta c^\infty_{n+1}(l)   \right) \leq \epsilon/2   \]

Therefore,

\[  \pr\left( c^\infty_{n}(k)   \Delta \bigcup_{l \in L} c^\infty_{n+1}(l)   \right) \leq \pr\left( c^\infty_{n}(k) \Delta c^m_{n}(k) \right) + \pr\left(c^m_{n}(k) \Delta \bigcup_{l \in L}  c^m_{n+1}(l) \right) + \pr\left(\bigcup_{l \in L}  c^m_{n+1}(l) \Delta \bigcup_{l \in L}  c^\infty_{n+1}(l) \right)  \]

\[  \pr\left( c^\infty_{n}(k)   \Delta \bigcup_{l \in L} c^\infty_{n+1}(l)   \right) \leq \epsilon/2+0+\epsilon/2=\epsilon \]

Since this estimate holds for any $\epsilon>0$, we conclude that: \[  \pr\left( c^\infty_{n}(k)   \Delta \bigcup_{l \in L} c^\infty_{n+1}(l)   \right) =0 \]

Therefore, $\zeta_n^\infty \hookrightarrow \zeta_{n+1}^\infty$. The proof that $U_n$ stabilizes $\zeta_n^\infty$ is analogous. 

Finally, let us show that $(\zeta_n^\infty)_{n \geq 0}$ generates:

\begin{lemma}
\label{zngenerates}
$(\zeta_n)_{n \geq 0}$ generates.
\end{lemma}

\begin{proof}

We need the lemma:

\begin{lemma}
\label{cgcebdeuxp}
For any $p \geq 0$, $k=0,...,2^p-1$, we have, almost surely:
\[  \esp[B_{\frac{k}{2^p}}|\zeta_n] \rightarrow_{n \rightarrow + \infty} B_{\frac{k}{2^p}}  \]

\end{lemma}

\begin{proof}

By proceeding as in lemma \ref{lemmekqk}, since $q_n$ divides $q_{n+1}$, then for any $W \in \Gamma^{t'_p}$,

\[  \esp\left[ B_W| \mathcal{A}_{W,t'_n- t'_p}(q_n)\right] \rightarrow_{n \rightarrow + \infty } B_W \]

Moreover, $\mathcal{B}(\zeta_n)= \mathcal{A}_{t'_n}(q_n)$, and by applying claim \ref{condindtribus}, $\esp\left[ B_W| \mathcal{A}_{t'_n}(q_n)\right] =  \esp\left[ B_W| \mathcal{A}_{W,t'_n- t'_p}(q_n)\right]$. Therefore, we get:

\[  \esp\left[ B_W|  \zeta_n \right]=\esp\left[ B_W| \mathcal{A}_{t'_n}(q_n)\right] =  \esp\left[ B_W| \mathcal{A}_{W,t'_n- t'_p}(q_n)\right] \rightarrow_{n \rightarrow + \infty } B_W  \]

Since $B_{\phi(W)}= \sum_{W' \in  \Gamma^{t'_p} / \phi(W') \leq \phi(W)} B_{W'}$, then we get lemma \ref{cgcebdeuxp}.

\end{proof}

Let $\omega,\omega'\in \Omega$ such that $c_n(\omega)=c_n(\omega')$, where $c_n(\omega)$ is the element of the partition $\zeta_n$ to which $\omega$ belongs. Then almost surely, for any $p \geq 0$, $k=0,...,2^p-1$, we have:

\[ \esp[B_{\frac{k}{2^p}}|\zeta_n](\omega)= \esp[B_{\frac{k}{2^p}}|\zeta_n](\omega') \]

By taking $n \rightarrow + \infty$, and applying lemma \ref{cgcebdeuxp}, we get:

\[ \omega\left(\frac{k}{2^p}\right)= B_{\frac{k}{2^p}}(\omega)=B_{\frac{k}{2^p}}(\omega')= \omega'\left(\frac{k}{2^p}\right) \]

By continuity of $\omega$ and $\omega'$, we conclude that $\omega=\omega'$.

\end{proof}

\begin{corollary}
\label{zninftygenerates}
$(\zeta_n^\infty)_{n \geq 0}$ generates.
\end{corollary}

\begin{proof}[Proof of corollary \ref{zninftygenerates}.]






Let $G$ be a measurable set and let $\epsilon>0$. There exists $n_0 \geq 0$ such that for any $n \geq n_0$, there is a $\zeta_n$-measurable set $G_n$ such that $\pr \left( G \Delta G_n \right) \leq \epsilon$. Let $I_n$ the (finite) set of indices such that \[ G_n= \bigcup_{i_n \in I_n} c_n(i_n) \] 
Let

\[  Q_n^\infty G_n =   \bigcup_{i_n \in I_n} Q_n^\infty c_n(i_n)  \]

By (\ref{summn}), $\sum_{m \geq n}  d(\zeta_n^m,\zeta_n^{m+1}) \rightarrow_{n \rightarrow + \infty} 0$. Therefore, there is an integer $n_1 > n_0$ such that for any $n \geq n_1$: 

\[ \sum_{m \geq n} \sum_{c \in \zeta_n} \mu_h \left(  Q_n^{m+1} c \Delta Q_n^{m} c \right) \leq \epsilon \] 

Since for any $c \in \zeta_n$, $Q_n^n c=c$, then 

\begin{eqnarray*}
\pr\left( Q_n^\infty G_n \Delta G_n \right) = \pr\left( Q_n^\infty G_n \Delta Q_n^n G_n \right) \leq \sum_{m \geq n} \pr\left( Q_n^{m+1} G_n \Delta Q_n^m G_n \right)  \\
=  \sum_{m \geq n}  \sum_{i_n \in I_n} \pr\left( Q_n^{m+1}c_n(i_n)  \Delta Q_n^m  c_n(i_n)  \right)   \leq  \sum_{m \geq n}   \sum_{c \in \zeta_n} \pr \left(  Q_n^{m+1} c \Delta Q_n^m c \right)  \leq \epsilon 
\end{eqnarray*}

Therefore, for any $n \geq n_1$, $\mu_h\left( Q_n^\infty G_n \Delta G \right) \leq 2 \epsilon$. Hence the generation of $\zeta_n^\infty$.

\end{proof}

\end{proof}



\section{The metric isomorphism}

This section is similar to \cite[section 2]{rotisom}, although the framework is more general. Our aim is to elaborate sufficient conditions on $B_n \in$ Diff$^\infty(M,\mu)$, so that if $T_n= B_n^{-1} S_{\frac{p_n}{q_n}} B_n$ weakly converges towards an automorphism $T$, then there exists a metric isomorphism between $(\Omega,U,\pr)$ and $(M,T,\mu)$.


To that end, we use lemma \ref{lemmekatokrotisom}: we construct a monotonous and generating sequence of partitions $\xi_n^\infty$ of $M$ and a sequence of isomorphisms $\bar{K}_n^\infty: \varmathbb{T}^1/\zeta_n^\infty \rightarrow M/ \xi_n^\infty$, such that $ \bar{K}_n^\infty U_n = T_n\bar{K}_n^\infty $ and $\bar{K}_{n+1|\zeta_n^\infty}^\infty=\bar{K}_n^\infty$. 


For $l=0,...,q_n-1$, let \[ \Delta_n\left(0, \frac{l}{q_n}\right)= \left[\frac{l}{q_n},\frac{l+1}{q_n} \right[_0 \]

\[ \eta_{n,0}= \left\{   \Delta_n\left(0, \frac{l}{q_n}\right),l=0,...,q_n-1 \right\} \]

For $i=1,...,t_n-1$, $l=0,...,q_n-1$, let \[ \Delta_n\left(\frac{i}{t_n}, \frac{l}{q_n}\right)= \bigcup_{j=0}^{q_n-1} \left( \frac{l}{q_n} + \left[\frac{j}{q_n},\frac{j+1}{q_n} \right[ \right)_0  \times \left[\frac{j}{q_n},\frac{j+1}{q_n} \right[_i \]

\[ \eta_{n,\frac{i}{t_n}}= \left\{   \Delta_n\left( \frac{i}{t_n}, \frac{l}{q_n}\right),l=0,...,q_n-1 \right\} \]

For any $i=0,...,t_n-1$, $ \eta_{n,\frac{i}{t_n}}$ is a partition of $M$ stable by $S_{\frac{p_n}{q_n}}$. Moreover, the family $ \left(\eta_{n,\frac{i}{t_n}}\right)_{i=0,...,t_n-1}$ is mutually independent. Let $\eta_n= \vee_{i=0}^{t_n-1} \eta_{n,\frac{i}{t_n}}$. For $n \geq 1$, $j=0,...,t_{n-1}$, $k=0,...,t_n/t_{n-1}-1$, let \[ h_n\left( \frac{t_n}{t_{n-1}}j + k \right)= j+ t_{n-1}k \]

The map $h_n:\{0,...,t_n-1\} \rightarrow \{0,...,t_n-1\}$ is bijective. It permutes coordinates $x_i$, $i \geq 1$, and "distributes" them in order to make possible the construction of a diffeomorphism $A_{n+1}$ that simultaneously sends $t_n$ independent partitions $\eta_{n,i}$ on $t_n$ other independent partitions $\eta_{n,i}^{n+1}$, modulo a set of small measure. See the beginning of section \ref{constrconjug} for more explanations.

The following lemma gives the isomorphism from which we start:

\begin{lemma}
\label{lempermutation}
For $i'=0,...,t_n-1$, let $a_n(i'/t_n)$ be relatively prime with $q_n$, and let

\[ K_{n,\frac{i'}{t_n}} : \begin{array}[t]{lcl} \zeta_{n,\frac{i'}{t_n}}  &\rightarrow &   \eta_{n,\frac{h_n(i')}{t_n}}  \\
               c_n\left(\frac{i'}{t_n},\frac{l}{q_n}\right)  & \mapsto    & \Delta_n\left(\frac{h_n(i')}{t_n},\frac{l a_n(i'/t_n)}{q_n} \right)
           \end{array}
           \]

$K_{n,\frac{i}{t_n}}$ is a metric isomorphism such that $K_{n,\frac{i}{t_n}}U_{n|\zeta_{n,\frac{i}{t_n}}} = S_{\frac{p_n}{q_n}} K_{n,\frac{i}{t_n}}$.

\end{lemma}

\begin{corollary}
\label{coravecind}
Let

\[ K_{n} : \begin{array}[t]{lcl} \zeta_{n}  &\rightarrow &   \eta_{n}  \\
             c_n\left(0,\frac{l_0}{q_n}\right) \cap ...\cap  c_n\left(\frac{t_n-1}{t_n},\frac{l_{t_n-1}}{q_n}\right)  & \mapsto    & K_{n,0} \left( c_n\left(0,\frac{l_0}{q_n}\right) \right)\cap ...\cap  K_{n,\frac{t_n-1}{t_n}} \left( c_n\left(\frac{t_n-1}{t_n},\frac{l_{t_n-1}}{q_n}\right) \right) 
           \end{array}
           \]

$K_{n}$ is a metric isomorphism such that $K_{n}U_{n} = S_{\frac{p_n}{q_n}} K_{n}$. 

In other words, the following diagram commutes:

\[ \xymatrix{
 \zeta_n \ar@(ul,dl)[]_{U_n}  \ar@{->}[r]^{K_n} & \eta_n \ar@(ur,dr)[]^{S_{\frac{p_n}{q_n}}} 
 } \]

\end{corollary}

\begin{proof}[Proof of corollary \ref{coravecind}.]

Since each $K_{n,\frac{i}{t_n}}$ is surjective, then $K_n$ is also surjective. Because of the mutual independence of partitions, $|\zeta_n|=|\eta_n|=q_n^{t_n}$. Therefore, $K_n$ is also injective, and it is an isomorphism.

\end{proof}


\bigskip

The following lemma combines corollary \ref{coravecind} with the facts that $\zeta_n  \hookrightarrow  \zeta_{n+1}$ and 

$\eta_n \hookrightarrow  \eta_{n+1}$:

\begin{lemma}
\label{diagcomutrotisom}
For any integers $0 \leq i' \leq t_n-1$ and $i' t_{n+1}/t_n \leq i < (i'+1) t_{n+1}/t_n$, let $a_n(i'/t_n)$,$a_{n+1}(i/t_{n+1})$, $q_n,q_{n+1} \in \varmathbb{N}$ such that $\gcd(a_n(i'/t_n),q_n)=\gcd(a_{n+1}(i/t_{n+1}),q_{n+1})=1$, such that $q_n$ divides $q_{n+1}$ and such that $q_n$ divides $a_{n+1}(i/t_{n+1})-a_n(i'/t_n)$.

There exists a partition $\eta_n^{n+1} \hookrightarrow  \eta_{n+1}$ of $M$ stable by $S_{\frac{p_n}{q_n}}$, and there exists a metric isomorphism $K_n^{n+1}: \zeta_n^{n+1} \rightarrow \eta_n^{n+1}$ such that $K_n^{n+1}=K_{n+1|\zeta_n}$ and such that $K_n^{n+1} U_{n} =S_{\frac{p_n}{q_n}} K_n^{n+1}$. There exists also a metric isomorphism $C_n^{n+1}: \eta_n \rightarrow \eta_n^{n+1}$ such that $C_n^{n+1} S_{\frac{p_n}{q_n}} =S_{\frac{p_n}{q_n}} C_n^{n+1}$ and $K_n^{n+1}Q_n^{n+1}=C_n^{n+1}K_n$. Said otherwise, we have the following commutative diagram:

\[ \xymatrix{
  \zeta_n \ar@(ul,dl)[]_{U_{n} }  \ar@{->}[r]^{K_n} \ar@{->}[d]_{Q_n^{n+1}} & \eta_n \ar@(ur,dr)[]^{S_{\frac{p_n}{q_n}}} \ar@{->}[d]^{C_n^{n+1}} \\
 \zeta_n \ar@(ul,dl)[]_{U_{n}}  \ar@{->}[r]^{K_n^{n+1}} \ar@{^{(}->}[d]  & \eta_n^{n+1} \ar@(ur,dr)[]^{S_{\frac{p_n}{q_n}}} \ar@{^{(}->}[d]  \\
 \zeta_{n+1}  \ar@{->}[r]^{K_{n+1}}  & \eta_{n+1} 
 } \]

\end{lemma}

\begin{proof}

We still give a proof, because the similarity with the corresponding lemma in \cite{rotisom} may not be obvious (especially the fact that there are multiple coefficients $b_{n+1}(i/t_{n+1})$, and because it allows to introduce some notations and objects that are useful in the next section.
Since $\gcd(a_{n+1}(i/t_{n+1}),q_{n+1})=1$, then by corollary \ref{coravecind}, $K_{n+1}$ is an isomorphism. Since $\zeta_n^{n+1} \hookrightarrow  \zeta_{n+1}$, we can define the isomorphism $K_n^{n+1}= K_{n+1|\zeta_n^{n+1}}$. Let $\eta_n^{n+1}=K_n^{n+1}(\zeta_n)$. We have $\eta_n^{n+1} \hookrightarrow  \eta_{n+1}$.

It remains to show that $K_n^{n+1}U_n =S_{\frac{p_n}{q_n}} K_n^{n+1}$ (it automatically implies that $\eta_n^{n+1}$ is stable by $S_{\frac{p_n}{q_n}}$, and that there is $C_n^{n+1}: \eta_n \rightarrow \eta_n^{n+1}$ such that $C_n^{n+1} S_{\frac{p_n}{q_n}} =S_{\frac{p_n}{q_n}} C_n^{n+1}$).

Let $K_{n,\frac{i'}{t_n}}^{n+1}= K_{n+1|\zeta_{n,\frac{i}{t_n}}^{n+1}}$ and let $0 \leq l \leq q_n-1$.

Let $k_0(n,q_n,t_n)$ be the integer defined by estimation (\ref{estcncnplusun}). For $0 \leq i \leq t_{n+1}-1$, $j_i=0,...,y_{n+1}$, $\gamma_i=0,...,q_n-1$, let 


\begin{equation}
\label{defdeetilde}
\tilde{e}\left( \frac{i}{t_{n+1}},\frac{\gamma_i}{q_n}+ j_i \frac{v_{n+1}}{q_{n+1}} \right)= \left\{  \omega \in \Omega / \arg\left(B_{\frac{i+1}{t_{n+1}}}(\omega)-B_{\frac{i}{t_{n+1}}}(\omega) \right)\in \frac{\gamma_i}{q_n}+ j_i \frac{v_{n+1}}{q_{n+1}} + N_j \right\}
\end{equation}

\[  \zeta_n^{n+1,n}= \left\{  \tilde{e}\left( \frac{i}{t_{n+1}}, \frac{\gamma_i}{q_n}+ j_i \frac{v_{n+1}}{q_{n+1}} \right),  0 \leq i \leq t_{n+1}-1, j_i=0,...,y_{n+1}, \gamma_i=0,...,q_n-1  \right\} \]

The set $\zeta_n^{n+1,n}$ is not a partition, but by (\ref{defbonneunion}), we have: 

\begin{equation}
\label{doubleinclusion}
\mathcal{B}(\zeta_n^{n+1}) \subset \mathcal{B}(\zeta_n^{n+1,n}) \subset \mathcal{B}(\zeta_{n+1} )
\end{equation}

Moreover, for any $W \in \Gamma^{t'_n}$ and $t_{n+1}=2^{k_0} t_n$, we have:

\[  \mathcal{A}_{W,k_0}(\mathcal{N}_n^{n+1})=  \mathcal{B}\left( \left\{ \tilde{e}\left( \frac{i}{t_{n+1}}, \frac{\gamma_i}{q_n}+ j_i \frac{v_{n+1}}{q_{n+1}}  \right), \phi(W) \leq \frac{i}{t_{n+1}} < \phi(W)+ \frac{1}{t_n}, 0 \leq \frac{\gamma_i}{q_n}+ j_i \frac{v_{n+1}}{q_{n+1}}  < 1 \right\}  \right)  \]

On the one hand, for $\frac{i'}{t_n} \leq \frac{i}{t_{n+1}} < \frac{i'+1}{t_n}$, $i'=0,...,t_n-1$, we have:

\[ U_n   \left(  \tilde{e}\left(\frac{i}{t_{n+1}}, \frac{\gamma_i}{q_n}+ j_i \frac{v_{n+1}}{q_{n+1}}  \right) \right)= \tilde{e}\left(\frac{i}{t_{n+1}},  \frac{\gamma_i}{q_n}+ j_i \frac{v_{n+1}}{q_{n+1}}   + \frac{p_n}{q_n} b_n\left(\frac{i'}{t_n} \right)\right)  \] 

On the other hand, since

\[   \tilde{e}\left(\frac{i}{t_{n+1}}, \frac{\gamma_i}{q_n}+ j_i \frac{v_{n+1}}{q_{n+1}}  \right)= \bigcup_{ \frac{j'}{q_{n+1}} \in \frac{\gamma_i}{q_n}+ j_i \frac{v_{n+1}}{q_{n+1}} +N_{j_i} } c_{n+1}\left(\frac{i}{t_{n+1}}, \frac{j'}{q_{n+1}} \right) \]

then

\[ K_{n+1}  \left( \tilde{e}\left(\frac{i}{t_{n+1}}, \frac{\gamma_i}{q_n}+ j_i \frac{v_{n+1}}{q_{n+1}}  \right) \right)= \bigcup_{  \frac{j'}{q_{n+1}} \in \frac{\gamma_i}{q_n}+ j_i \frac{v_{n+1}}{q_{n+1}} +N_{j_i}  } \Delta_{n+1}\left(\frac{h_{n+1}(i)}{t_{n+1}}, \frac{j'a_{n+1}\left( \frac{i}{t_{n+1}} \right)}{q_{n+1}} \right) \]

Therefore, 

\[ K_{n+1} U_n \left( \tilde{e}\left(\frac{i}{t_{n+1}},\frac{\gamma_i}{q_n}+ j_i \frac{v_{n+1}}{q_{n+1}} \right) \right)= \bigcup_{ \frac{j'}{q_{n+1}} \in \frac{\gamma_i}{q_n}+ j_i \frac{v_{n+1}}{q_{n+1}} +N_{j_i}  }    \Delta_{n+1}\left(\frac{h_{n+1}(i)}{t_{n+1}}, \frac{j'a_{n+1}\left( \frac{i}{t_{n+1}} \right)}{q_{n+1}} +  \frac{p_n}{q_n} b_n\left(\frac{i'}{t_n}\right) a_{n+1}\left( \frac{i}{t_{n+1}}  \right) \right) \]

Since we assumed that \[ \frac{1}{q_n} a_{n+1}\left( \frac{i}{t_{n+1}} \right)=  \frac{1}{q_n} a_{n}\left( \frac{i'}{t_{n}} \right) \; \mod \, 1 \]

and

\[  \frac{1}{q_n} a_{n}\left( \frac{i'}{t_{n}} \right)  b_{n}\left( \frac{i'}{t_{n}} \right) = \frac{1}{q_n} \; \mod \, 1 \]

then we have:

\[ K_{n+1} U_n \left( \tilde{e}\left(\frac{i}{t_{n+1}},  \frac{\gamma_i}{q_n}+ j_i \frac{v_{n+1}}{q_{n+1}}  \right) \right)= \bigcup_{ \frac{j'}{q_{n+1}} \in \frac{\gamma_i}{q_n}+ j_i \frac{v_{n+1}}{q_{n+1}} +N_{j_i}  }    \Delta_{n+1}\left(\frac{h_{n+1}(i)}{t_{n+1}}, \frac{j'a_{n+1}\left( \frac{i}{t_{n+1}} \right)}{q_{n+1}} +  \frac{p_n}{q_n}  \right) \]

\[ = S_{\frac{p_n}{q_n}}  K_{n+1} \left( \tilde{e}\left(\frac{i}{t_{n+1}},  \frac{\gamma_i}{q_n}+ j_i \frac{v_{n+1}}{q_{n+1}}  \right) \right)     \]

Therefore, by (\ref{doubleinclusion}), we conclude:

\[ K_{n+1} U_n \left( c_n^{n+1}\left(\frac{i'}{t_{n}}, \frac{l}{q_n} \right) \right)=
 S_{\frac{p_n}{q_n}}  K_{n+1} \left( c_n^{n+1}\left(\frac{i'}{t_{n}}, \frac{l}{q_n} \right)  \right)     \]

\end{proof}

We also let:

\begin{equation}
\label{defgammatilde}
\tilde{\Gamma}\left(\frac{i}{t_{n+1}},\frac{\gamma_i}{q_n}+ j_i \frac{v_{n+1}}{q_{n+1}} \right)= K_{n+1} \left( \tilde{e}\left(\frac{i}{t_{n+1}}, \frac{\gamma_i}{q_n}+ j_i \frac{v_{n+1}}{q_{n+1}}   \right) \right)
\end{equation}

where $ \tilde{e} \left(\frac{i}{t_{n+1}}, \frac{\gamma_i}{q_n}+ j_i \frac{v_{n+1}}{q_{n+1}} \right)$ was defined in (\ref{defdeetilde}).

\bigskip

By iterating lemma \ref{diagcomutrotisom}, we get a corollary that is important for the construction of the isomorphism:

\begin{corollary}
\label{diagitererotisom}
For any $m>n$, there are partitions $\eta_n^{m} \hookrightarrow  \eta_{n+1}^m$ of $M$ such that $\eta_n^{m}$ is stable by $S_{\frac{p_n}{q_n}}$ and there exists an isomorphism $K_n^{m}: \zeta_n \rightarrow \eta_n^{m}$ such that $K_n^{m} U_n =S_{\frac{p_n}{q_n}} K_n^{m}$ and $K_n^m=K_{n+1|\eta_n^m}^m$.

Said otherwise, the following diagram commutes:

\[ \xymatrix{
\zeta_n^m \ar@(ul,dl)[]_{U_n}  \ar@{->}[r]^{K_n^{m}} \ar@{^{(}->}[d]  & \eta_n^{m} \ar@(ur,dr)[]^{S_{\frac{p_n}{q_n}}} \ar@{^{(}->}[d]  \\
\zeta_{n+1}   \ar@{->}[r]^{K_{n+1}^{m}} & \eta_{n+1}^{m}  
 } \]

\end{corollary}

\begin{proof} The proof is similar to the one found in \cite{katokant}.
\end{proof}

For any $n$ fixed, the sequence of partitions $(\eta_n^{m})_{m\geq n}$ must converge when $m \rightarrow + \infty$, in order to obtain a full sequence of monotonic partitions. Moreover, the possible limit sequence (i.e. a possible $\eta_n^\infty$) must generate. Indeed, these assumptions are required to apply lemma \ref{lemmekatokrotisom}. However, none of these assumptions are satisfied, in general. Therefore, to obtain these assumptions, we pull back the partition $\eta_n^{m}$ by a suitable smooth measure-preserving diffeomorphism $B_m$. The following lemma, already proved in \cite{katokant}, gives the conditions that $B_m$ must satisfy:

\begin{lemma}
\label{conditionsbnrotisom}
Let $B_m \in$ Diff$^\infty(M,\mu)$. Let $A_{m+1}= B_{m+1} B_m^{-1}$.

\begin{enumerate}
\item If $A_{m+1} S_{\frac{1}{q_m}}=  S_{\frac{1}{q_m} }A_{m+1}$ and
if \[ \sum_{m\geq 0} \sum_{c \in \eta_m} \mu \left( A_{m+1}(c) \Delta C_m^{m+1}(c)  \right) < + \infty \]
then for any fixed $n$, when $m \rightarrow + \infty$, the sequence of partitions $\xi_n^m= B_m^{-1} \eta_n^m$ converges. We denote $\xi_n^\infty$ the limit. The sequence $\xi_n^\infty$ is monotonous and $T_n= B_n^{-1} S_{\frac{p_n}{q_n}} B_n$ stabilizes each $\xi_n^\infty$.
\item If, moreover, the sequence $\xi_n= B_n^{-1} (\eta_n)$ generates, then so does $\xi_n^\infty$.
\end{enumerate}

\end{lemma}

By adding to lemma \ref{conditionsbnrotisom} the convergence of the sequence $T_n$, we obtain the required isomorphism:

\begin{corollary}
\label{corolisomrotisom}
If both conditions 1. and 2. of lemma \ref{conditionsbnrotisom} hold, and if $T_n= B_n^{-1} S_{\frac{p_n}{q_n}} B_n$ weakly converges towards an automorphism $T$, then $(\Omega, U, \pr)$ and $(M,T, \mu)$ are metrically isomorphic.
\end{corollary}



\begin{proof}
The proof is the same as in \cite{katokant}.

\end{proof}

\section{The sequence of conjugacies}

In order to construct a suitable smooth approximation $A_{n+1}$ of $C_n^{n+1}$, we re-write and approximate the partition $\eta_n^{n+1}$, so that most elements of the approximated partition $\eta_n^{'n+1}$ consist of unions of medium-sized "cubes" with suitable properties. 

These cubes need to be small enough in order to have a good approximation of $\eta_n^{n+1}$ by $\eta_n^{'n+1}$, but they must be large enough to suitably control the norm of $A_{n+1}$. Likewise, elements of $\eta_n$ are mostly decomposed into cubes of the same medium-size. 

Thus, transforming $\eta_n$ into $\eta_n^{'n+1}$ consists of permuting these cubes. We construct a smooth approximation of this permutation. For the vertical permutation along the $z$ coordinate, we apply a transformation developed in \cite{anosovkatok70}, \cite{katokant}, \cite{rotisom}, based on fibred rotations of the flow along $z$.

For horizontals permutations, along coordinates $x_i, i \geq 1$, we apply and generalize the method of "quasi-permutations" that we introduced in \cite{nlbernoulli}. There is no rotation flow along $x_i$, so we displace cubes one by one.

The re-writing and approximation of $\eta_n^{n+1}$ has 3 steps: first (lemma \ref{estdern}), we re-write elements of $\eta_n^{n+1}$ using the "stacking phenomenon" presented in \cite{rotisom}. A priori, $\Gamma\left(\frac{i}{t_{n+1}}, \frac{j}{q_n^{k_0}} \right)$ consists of $v$  "slices" of width $1/q_{n+1}$, with $v=v_{n+1}$ or $\rho_{n+1}$. However, this fact does not ensure the convergence of $T_n$, because it only implies that $\|B_{n+1} \|_{j} \leq F(q_{n+1})$ for some fixed function $F$. In order to apply the reasoning above successfully, we need a better estimate. Lemma \ref{estdern} shows that "slices" of $\Gamma\left(\frac{i}{t_{n+1}}, 0 \right)$ of width $1/q_{n+1}$ stack on each other, which gives $b_{n+1}(i/t_{n+1})$ connected components to $\Gamma\left(\frac{i}{t_{n+1}}, 0 \right)$, each having a width of order $v/(q_{n+1} b_{n+1}(i/t_{n+1}))$. If we only consider slices for $v=v_{n+1}$, this will allow an estimate of the form $\|B_{n+1} \|_{j} \leq F(q_n, \max_{i} b_{n+1}(i/t_{n+1}))$, for some fixed function $F$, which will allow the convergence of $T_n$, by taking $b_{n+1}(i/t_{n+1})$ small.


Second (lemma \ref{approxhorizontale}), we approximate $\eta_n^{n+1}$ horizontally: along the coordinate $x_i$, elements of $\eta_n^{n+1}$ are piecewise constant, with a thickness of $1/q_{n+1}$. This would lead to a too large conjugacy, of order $q_{n+1}$. Mostly, we make an approximation of thickness $w_{n+1}/q_{n+1}$, with $w_{n+1}$ large.

\bigskip
\label{whyinfinite}
The construction is carried on the infinite-dimensional Hilbert cube, instead of the finite-dimensional annulus, as in previous constructions \cite{anosovkatok70,windsor07,katokant,rotisom,nlbernoulli}, because the stacking phenomenon only allows to stack together $q_{n+1}/f(q_n)$ elements.

Therefore, if elements of the partition $\eta_{n+1}$ have a length negligible with respect to $1/q_{n+1}$, we cannot stack them to get a partition with elements of length negligible with respect to $q_{n+1}$, which is a  necessary condition for the smooth convergence with our method.

In our case, the volume of an element of the intersection of $t_{n+1}$ mutually independent partitions of size $q_{n+1}$ each has a volume of $1/q_{n+1}^t$. Therefore, in dimension $d$, the maximum of the minimal length of an element of this partition is $1/q_{n+1}^{t-d}$. Since $t_{n} \rightarrow + \infty$, then at the limit, the dimension of the ambient space must be infinite. For example, in dimension 2, for $t \geq 1$ integer, let $\eta_{n,t}$ be the partition defined by 

\[ \Delta_n(t,l/q_n)=   \bigcup_{k=0}^{q_n^{t-1}-1} \bigcup_{j=0}^{q_n-1}    \left( \frac{l}{q_n} + \left[ \frac{j}{q_n}, \frac{j+1}{q_n} \right[ \right) \times \left(\frac{k}{q_n^{t-1}}+  \left[ \frac{j}{q_n^t}, \frac{j+1}{q_n^t} \right[ \right)  \]

The partitions $\eta_{n,t}$, $t \geq 1$ are mutually independent and stable by $S_{1/q_n}$. However, elements of $\eta_{n+1,t}$ are $1/q_{n+1}^{t-1}$-periodic along the horizontal coordinate, so for $t \geq 2$, we cannot move them into elements of $\eta_n^{n+1}$ with a diffeomorphism having a norm smaller than $q_{n+1}$.

However, if, for example, we had to consider $t$ pairwise-independent partitions of size $q_{n+1}$, instead of mutually independent partitions, then we would be able to carry the construction on the 2-dimensional annulus.

The broad idea is that the Anosov-Katok method does not allow the manipulation of excessively "complicated" partitions, which have smaller components, because of the constraint of smoothness. To manipulate simpler partitions, we increase the number of dimensions of the partitions. We met the same problem in \cite{nlbernoulli}, where we considered two-dimensional partitions, instead of the "one-dimensional" partitions from \cite{anosovkatok70} (i.e. the periodic transformation was metrically isomorphic to a cyclic permutation) in order to avoid elements of the partition to be intertwined.

\subsection{Re-writing and approximation of the partition $\eta_n^{n+1}$ by $\eta_n^{'n+1}$}


For $i=0,...,t_{n}-1$, let 

\[  J_{n+1}(h_n(i))= \left\{ h_{n+1}\left( \frac{t_{n+1}}{t_n} i + k \right)= h_n(i)+t_n k, k=0,...,t_{n+1}/t_n-1   \right\} \]

The set $\{ J_{n+1}(h_n(i)),i=0,...,t_n-1 \}$ is a partition of $\{ 0,...,t_{n+1}-1\}$.

The following lemma, already proved in \cite{rotisom}, shows how slices forming $\eta_{n}^{n+1}$ stack on each other.

\begin{lemma}
\label{estdern}

For $v=v_{n+1}$ or $\rho_{n+1}$, let 

\[  f_{n+1}(i/t_{n+1},v)  =  \left\lfloor \frac{v -1}{b_{n+1}(i/t_{n+1})} \right\rfloor \]

\[ m_{n+1}(i/t_{n+1},v)= v -1 - b_{n+1}(i/t_{n+1}) f_{n+1}(i/t_{n+1},v) \]

and for $0 \leq l \leq b_{n+1}(i/t_{n+1})-1$, let

\[ k_{n+1}(i/t_{n+1},l,v)= \left\lfloor l a_{n+1}(i/t_{n+1}) v \right\rfloor \]

\[ r_{n+1}(i/t_{n+1},l,v)= la_{n+1}(i/t_{n+1})- v k_{n+1}(i/t_{n+1},l,v) \]

We have: 

\[ \tilde{\Gamma}\left(\frac{h_{n+1}(i)}{t_{n+1}}, \frac{\gamma_{h_{n+1}(i)}}{q_n}+ j_{h_{n+1}(i)} \frac{v_{n+1}}{q_{n+1}}  \right) = \bigcup_{l=0}^{b_{n+1}\left(\frac{i}{t_{n+1}} \right)-1}  \left(\frac{\gamma_{h_{n+1}(i)}}{q_n}+ j_{h_{n+1}(i)} \right)_0+ \tilde{\Gamma}_{j_{h_{n+1}(i)}}\left(\frac{h_{n+1}(i)}{t_{n+1}}, l \right) \]

If $j_{h_{n+1}(i)}=0,...,y_{n+1}-1$, then we let $v=v_{n+1}$. If $j_{h_{n+1}(i)}=y_{n+1}$, we let $v=\rho_{n+1}$. If $0 \leq l \leq m_{n+1}(i/t_{n+1},v)$, we have:

\[  \tilde{\Gamma}_{j_{h_{n+1}(i)}}\left(\frac{h_{n+1}(i)}{t_{n+1}}, l \right) = \left( \frac{k_{n+1}(i/t_{n+1},l,v)v}{q_{n+1}}+ \frac{r_{n+1}(i/t_{n+1},l,v)}{q_{n+1}} \right)_0 + \bigcup_{j'=0}^{f_{n+1}(i/t_{n+1},v)+1 } \Delta_{n+1} \left(\frac{h_{n+1}(i)}{t_{n+1}}, \frac{j'}{q_{n+1}}   \right) \]

and if $m_{n+1}(i/t_{n+1},v) +1 \leq l \leq b_{n+1}(i/t_{n+1}) -1$:

\[  \tilde{\Gamma}_{j_{h_{n+1}(i)}}\left(\frac{h_{n+1}(i)}{t_{n+1}}, l \right) = \left( \frac{k_{n+1}(i/t_{n+1},l,v)v}{q_{n+1}}+ \frac{r_{n+1}(i/t_{n+1},l,v)}{q_{n+1}} \right)_0 + \bigcup_{j'=0}^{ f_{n+1}(i/t_{n+1},v) } \Delta_{n+1} \left(\frac{h_{n+1}(i)}{t_{n+1}}, \frac{j'}{q_{n+1}}   \right) \]

\end{lemma}

\begin{proof}

The proof is the same as the corresponding lemma in \cite[section 2]{rotisom}.

\end{proof}

We fix $0 \leq i \leq t_{n+1}-1$. For $n \geq 1$, let 

\[ R^{(n)}(h_{n+1}(i)/t_{n+1},v)= \bigcup_{l=0}^{b_{n+1}(i/t_{n+1})-1}  R^{(n),l}(h_{n+1}(i)/t_{n+1},v) \]

with, if $0 \leq l \leq m_{n+1}(i/t_{n+1},v)$:

\[  R^{(n),l}(h_{n+1}(i)/t_{n+1},v) = \left( \frac{k_{n+1}(i/t_{n+1},l,v)v}{q_{n+1}}+ \frac{r_{n+1}(i/t_{n+1},l,v)}{q_{n+1}} + \left[ 0, \frac{\left\lfloor \frac{v -1}{b_{n+1}(i/t_{n+1})} \right\rfloor +1}{q_{n+1}} \right[ \right) \]

and if $m_{n+1}(i/t_{n+1},v)+1 \leq l \leq b_{n+1}(i/t_{n+1})-1$:

\[  R^{(n),l}(h_{n+1}(i)/t_{n+1},v) =   \left( \frac{k_{n+1}(i/t_{n+1},l,v)v}{q_{n+1}}+ \frac{r_{n+1}(i/t_{n+1},l,v)}{q_{n+1}} + \left[ 0, \frac{\left\lfloor \frac{v -1}{b_{n+1}(i/t_{n+1})} \right\rfloor }{q_{n+1}} \right[ \right)  \]

By abuse of notation, if $j=0,...,y_{n+1}-1$, we write $R^{(n),l}(h_{n+1}(i)/t_{n+1},j)$ to denote $R^{(n),l}(h_{n+1}(i)/t_{n+1},v_{n+1})$, and if $j=y_{n+1}$, we write $R^{(n),l}(h_{n+1}(i)/t_{n+1},j)$ to denote $R^{(n),l}(h_{n+1}(i)/t_{n+1},\rho_{n+1})$.

\bigskip

Let $\epsilon'_1 >0$, let $b_{n+1}=\max_{0 \leq i \leq t_{n+1}-1} b_{n+1}(i/t_{n+1})$ and let: 

\begin{equation}
\label{defwnplusun}
w_{n+1}= \left\lfloor \frac{q_{n+1}\epsilon'_1 }{2^{3n} q_n^{t_n} \frac{t_{n+1}}{t_n} (y_{n+1}q_n)^{ \frac{t_{n+1}}{t_n}} b_{n+1}  }  \right\rfloor 
\end{equation} 

where $y_{n+1}$ was defined in relation (\ref{defynrhon}). For $q_{n+1}$ sufficiently large, by Euclidean division, we can write: \[ q_{n+1}= w_{n+1} u_{n+1} + \lambda_{n+1} \] with $0 \leq \lambda_{n+1} < w_{n+1}$ (in our construction, we can even take $\lambda_{n+1}=q_n$). Moreover, $u_{n+1} \leq q_{n+1}/ w_{n+1}$. We show the lemma:

\begin{lemma}
\label{defpbarrezero}
For any $\epsilon'_0>0$, there exists an integer $w'_{n+1} \geq 0$ such that

\[ \frac{w'_{n+1}}{q_{n+1}} \geq 1/\rr(n, t_{n+1}, b_{n+1},q_n,\epsilon'_0, \epsilon'_1 ) \label{rwprime}  \]

and there exists a family of pairwise disjoint measurable sets $\bar{\mathcal{P}}_0 \subset \varmathbb{T}$, $R_{1/q_n}$-invariant, such that:

\[ \mu\left( \bar{\mathcal{P}}_0  \right) \geq 1- \epsilon'_0  \]

such that each set $P_x$ of $\bar{\mathcal{P}}_0$ is of the form 
 $x+[0,w'_{n+1}/q_{n+1}[$, and such 
 that for any $0 \leq i \leq t_{n+1}-1$, 
 any $0 \leq j \leq u_{n+1}$, there
  exists $0 \leq \gamma \leq q_n-1$, $0 \leq k \leq y_{n+1}$,
   $0 \leq l \leq b_{n+1}(i/t_{n+1})-1$, such that

\[ x+ \left[ 0,\frac{w'_{n+1}}{q_{n+1}} \right[ \subset  \left(  \frac{jw_{n+1}}{q_{n+1}} +  \frac{\gamma}{q_{n}}+ \frac{kv_{n+1}}{q_{n+1}}   +  R^{(n),l}(i/t_{n+1},k) \right)  \]

\end{lemma}

\begin{corollary}
\label{pzeropartition}
Let $ \mathcal{P}_0 =  \bar{\mathcal{P}}_0  \cup \{ [k/q_{n+1},(k+1)/q_{n+1}[, [k/q_{n+1},(k+1)/q_{n+1}[ \not\subset \bar{\mathcal{P}}_0    \}$. 
$ \mathcal{P}_0$ is a partition of $\varmathbb{T}$ such that for 
any $c \in \mathcal{P}_0$, for any $0 \leq i \leq t_{n+1}-1$, any $0 \leq j \leq u_{n+1}$, there exists $0 \leq \gamma \leq q_n-1$, $0 \leq k \leq y_{n+1}$, $0 \leq l \leq b_{n+1}(i/t_{n+1})-1$, such that

\[  c \subset  \left(  \frac{jw_{n+1}}{q_{n+1}} +\frac{\gamma}{q_{n}}+ \frac{kv_{n+1}}{q_{n+1}} +  R^{(n),l}(i/t_{n+1},k) \right)  \]
\end{corollary}

Corollary \ref{pzeropartition} is immediate from lemma \ref{defpbarrezero}.

\begin{proof}[Proof of lemma \ref{defpbarrezero}.]

For $i=0,...,t_{n+1}-1$, $j=0,...,u_{n+1}$, let 

\[ \mathcal{P}\left(\frac{i}{t_{n+1}},\frac{jw_{n+1}}{q_{n+1}} \right)= \left\{  \frac{jw_{n+1}}{q_{n+1}}+ \frac{\gamma}{q_{n}}+\frac{kv_{n+1}}{q_{n+1}} + R^{(n),l}(i/t_{n+1},k) , \right. \] \[ \left.  l=0,...,b_{n+1}(i/t_{n+1})-1,\gamma=0,...,q_n-1,k=0,...,y_{n+1} \right\}  \]

$\mathcal{P}\left(\frac{i}{t_{n+1}},\frac{jw_{n+1}}{q_{n+1}} \right)$ is a partition of $\varmathbb{T}$. In particular, 

\[ Leb\left( \mathcal{P}\left(\frac{i}{t_{n+1}},\frac{jw_{n+1}}{q_{n+1}} \right) \right)=1 \]


On the other hand, let \[ w'_{n+1}= \left\lfloor \frac{q_{n+1} \epsilon'_0}{3u_{n+1} t_{n+1} q_{n}b_{n+1}q_{n+1}}   \right\rfloor \]

On the other hand, by Euclidean division, we can write \[ \frac{q_{n+1}}{q_n}= u'_{n+1} w'_{n+1} + \lambda'_{n+1} \]

with $\lambda'_{n+1} \leq u'_{n+1}$, and $u'_{n+1}\geq 1/\epsilon'_0$. Let

\[ \mathcal{P}'(w')= \left\{ \frac{\gamma}{q_n}+ \frac{j'w'_{n+1}}{q_{n+1}} + G_{j'} , \gamma=0,...,q_{n}-1, j'=0,..., u'_{n+1}  \right\} \] 

with $G_{j'}=[0,w'_{n+1}/q_{n+1}[$ if $j'=0,...,u'_{n+1}-1$, and $G_{u'_{n+1}}=[0,\lambda'_{n+1}/q_{n+1}[$. For $i=0,...,t_{n+1}-1$, $j=0,...,u_{n+1}-1$, let

\[ \phi(i,j) : \begin{array}[t]{lcl} \mathcal{P}\left(\frac{i}{t_{n+1}},\frac{jw_{n+1}}{q_{n+1}} \right)  &\rightarrow &   \mathcal{P}'(w')  \\
               A & \mapsto    &  \bigcup \{c \in   \mathcal{P}'(w') / c \subset A\}
           \end{array}
           \]

$\phi(i,j)(A)$ is the "projection" of $A$ on the partition $\mathcal{P}'(w')$. Note that in general, $\phi(i,j) (A \cup B) \not\subset \phi(i,j) (A) \cup \phi(i,j) (B)$. Moreover, it is possible to have $\phi(i,j) (A)=\emptyset$ for some $A$, if $w'_{n+1}$ does not divide $f_{n+1}(i/t_{n+1},\rho_{n+1})$. Let 

\[ \mathcal{P}''\left(\frac{i}{t_{n+1}},\frac{jw_{n+1}}{q_{n+1}} \right)=  \bigcup_{A \in \mathcal{P}\left(\frac{i}{t_{n+1}},\frac{jw_{n+1}}{q_{n+1}} \right) } \phi(i,j) (A)  \]

and let 
\[ \mathcal{P}''= \bigcap_{j=0}^{u_{n+1}} \bigcap_{i=0}^{t_{n+1}-1}  \mathcal{P}''\left(\frac{i}{t_{n+1}},\frac{jw_{n+1}}{q_{n+1}} \right)   \]

For $A \in  \mathcal{P}\left(\frac{i}{t_{n+1}},\frac{jw_{n+1}}{q_{n+1}} \right) $, let 

\[  \psi(i,j)(A)= \bigcup \{c \in   \mathcal{P}'(w') / c \cap A \not\eq \emptyset \} \]

We have: $ \phi(i,j)(A) \subset A \subset \psi(i,j)(A)$. At most, $\psi(i,j)(A)- \phi(i,j)(A)$ correspond to 2 elements of $ \mathcal{P}'(w')$, because $A$ is an interval. Therefore, \[Leb( \psi(i,j)(A)- \phi(i,j)(A)) \leq 2 w'_{n+1}/q_{n+1}  \] Therefore,

\[ Leb( \phi(i,j)(A)) \geq \mu(A) -2 w'_{n+1}/q_{n+1}  \]

Note that this lower bound can be negative for some $A$.

\[ Leb\left( \mathcal{P}''\left(\frac{i}{t_{n+1}},\frac{jw_{n+1}}{q_{n+1}} \right)  \right) = \sum_{ A \in \mathcal{P}\left(\frac{i}{t_{n+1}},\frac{jw_{n+1}}{q_{n+1}} \right) } Leb\left( \phi(i,j)(A)  \right)  \geq \sum_{ A \in \mathcal{P}\left(\frac{i}{t_{n+1}},\frac{jw_{n+1}}{q_{n+1}} \right) } Leb(A)-2 w'_{n+1}/q_{n+1}   \]

\[ Leb\left( \mathcal{P}''\left(\frac{i}{t_{n+1}},\frac{jw_{n+1}}{q_{n+1}} \right)  \right) \geq 1- 2 q_n b_{n+1} y_{n+1}w'_{n+1}/q_{n+1}  \]

\[ Leb\left( \mathcal{P}'' \right) \geq 1- 2u_{n+1}t_{n+1}  q_n b_{n+1} y_{n+1}w'_{n+1}/q_{n+1}   \]

Let \[  \bar{\mathcal{P}}'(w')= \left\{ \frac{\gamma}{q_n}+ \frac{j'w'_{n+1}}{q_{n+1}} + G_{j'} , \gamma=0,...,q_{n}-1, j'=0,..., u'_{n+1}-1  \right\} \] 

and let $\bar{\mathcal{P}}_0= \mathcal{P}'' \cap  \bar{\mathcal{P}}'(w')$.

Using our definition of $ w'_{n+1}$, we have:

\[ Leb\left( \bar{\mathcal{P}}_0 \right) \geq 1- 3u_{n+1}t_{n+1}  q_n b_{n+1} y_{n+1}w'_{n+1}/q_{n+1} \geq 1-\epsilon'_0 \]

\end{proof}

Let $u(0)= \{0,...,|\mathcal{P}_0|-1\}$. We write:

\[ \mathcal{P}_0=\{ P_0(j), j \in u(0) \} \]

\[ \left( \mathcal{P}_0 \right)_0= \{ C(0),...,C( |\mathcal{P}_0|-1) \}  \]

and \[  \bar{u}(0)= \{ j \in u(0) / P_0(j) \in \bar{\mathcal{P}}_0 \} \]

For $i \geq 1$, let $\bar{\mathcal{P}}_i= \{ P_i(j) , j=0,...,u_{n+1}-1 \} $, where $P_i(j)=[jw_{n+1}/q_{n+1},(j+1)w_{n+1}/q_{n+1}[$, and $\mathcal{P}_i= \bar{\mathcal{P}}_i \cup \{ P_i(u_{n+1}) \} $ where $P_i(u_{n+1})=[u_{n+1} w_{n+1}/q_{n+1},1]$. $\mathcal{P}_i$ is a partition of $[0,1]$. We write:

\[ \left( \mathcal{P}_i \right)_i= \{ C(0),...,C( |\mathcal{P}_i|-1) \}  \]

We define $u(i)= \{0,...,|\mathcal{P}_i|-1\}$ and $\bar{u}(i)= \{ j \in u(i) / C(j) \in \bar{\mathcal{P}}_i \}$.

For $J \subset \varmathbb{N}$, let \[ u(J) = \times_{j \in J} u (j) \]

\[ \bar{u}(J) = \times_{j \in J} \bar{u}(j) \]

and for $m \in u(J)$, $m=(m_1,...,m_{|J|})$, let

\[ C(m)= \times_{j \in J} C(m_j)  \]

The following proposition gives horizontal approximation:

\begin{proposition}
\label{propapprox}
\label{approxhorizontale}
There exists a partition $\eta_n^{'n+1} \subset \mathcal{B}\left( C(m), m \in u\left(\{0,...,t_{n+1}-1\} \right) \right)$ such that $\eta_n^{'n+1}$ is stable by $S_{\frac{p_n}{q_n}}$ and 

\[ d\left( \eta_n^{n+1} , \eta_n^{'n+1} \right) \leq \frac{1}{2^n}  \]

Moreover, $\eta_n^{'n+1} = \vee_{i=0}^{t_n-1} \eta_{n,i}^{'n+1}$, with \[ \eta_{n,i}^{'n+1} = \left\{ \Delta_n^{'n+1}\left(\frac{i}{t_n}, \frac{l}{q_n} \right), l=0,...,q_n-1  \right\}  \]

and we have:

\[  \Delta_n^{'n+1}\left(\frac{0}{t_n}, \frac{l}{q_n} \right) = \bigcup_{m \in u\left( J_{n+1}(0)-\{0\}\right)  }   \left( \frac{l}{q_n} + R^{(m)}\right)_0 \times C(m) \]

where, for any $m$, $R^{(m)}$ is a fundamental domain of the circle rotation $R_{\frac{1}{q_n}}$. For $1 \leq i \leq t_n-1$,

\[  \Delta_n^{'n+1}\left(\frac{i}{t_n}, \frac{l}{q_n} \right) =  \bigcup_{i_0 \in u(0) } \bigcup_{m \in E(i_0,i,l)  }  C( R_{\frac{l}{q_n}} (i_0),m) \]

where $E(i_0,i,l) \subset u\left( i \right) $ and $|E(i_0,i,l)|$ is independent of $i_0$.

\end{proposition}

\begin{proof}

By (\ref{defgammatilde}) and lemma \ref{estdern}, for any $0 \leq j_0 \leq y_{n+1}$, $0 \leq \gamma_0 \leq q_n-1 $, we have:

\[ \tilde{\Gamma} \left( \frac{0}{t_{n+1}}, \frac{j_0v_{n+1}}{q_{n+1}} + \frac{\gamma_0}{q_n} \right)= \left(  \frac{j_0v_{n+1}}{q_{n+1}} + \frac{\gamma_0}{q_n}+ R^{(n)}(0/t_{n+1},j_0) \right)_0  \]

and for $1 \leq i \leq t_{n+1}-1$, $0 \leq j_i \leq y_{n+1}$, $0 \leq \gamma_i \leq q_n-1 $, we have:

 \[ \tilde{\Gamma} \left( \frac{i}{t_{n+1}},\frac{j_iv_{n+1}}{q_{n+1}} + \frac{\gamma_i}{q_n} \right)=  \bigcup_{j'=0}^{q_{n+1}-1} \left(  \frac{j'}{q_{n+1}} + \frac{j_iv_{n+1}}{q_{n+1}} + \frac{\gamma_i}{q_n} + R^{(n)}(i/t_{n+1},j_i) \right)_0 \times \left( \left[ \frac{j'}{q_{n+1}}, \frac{j'+1}{q_{n+1}} \right[ \right)_{i}   \]

For $0 \leq j_0 \leq y_{n+1}$, $0 \leq \gamma_0 \leq q_n-1 $, let:

\[ \Gamma'\left( \frac{0}{t_{n+1}},  \frac{j_0v_{n+1}}{q_{n+1}} + \frac{\gamma_0}{q_n}  \right)= \tilde{\Gamma} \left( \frac{0}{t_{n+1}},  \frac{j_0v_{n+1}}{q_{n+1}} + \frac{\gamma_0}{q_n}  \right)  \]

and for $1 \leq i \leq t_{n+1}-1$, $0 \leq j_i \leq y_{n+1}$, $0 \leq \gamma_i \leq q_n-1 $, let:

\[ \Gamma'\left( \frac{i}{t_{n+1}}, \frac{j_iv_{n+1}}{q_{n+1}} + \frac{\gamma_i}{q_n}  \right)=  \bigcup_{j' \in u(i)} \left( \frac{j'w_{n+1}}{q_{n+1}}+ \frac{j_iv_{n+1}}{q_{n+1}} + \frac{\gamma_i}{q_n}+ R^{'(n)}(i/t_{n+1},j_i)  \right)_0 \times \left( P(j') \right)_{i}   \]

$\Gamma'$ is a horizontal approximation of $\tilde{\Gamma}$. The sets $\tilde{\Gamma} \left( \frac{i}{t_{n+1}}, \frac{j_iv_{n+1}}{q_{n+1}} + \frac{\gamma_i}{q_n},l \right)$ and $\Gamma'\left( \frac{i}{t_{n+1}}, \frac{j_iv_{n+1}}{q_{n+1}} + \frac{\gamma_i}{q_n},l\right)$
 coincide, except on elements of the form $ \left( \left[ \frac{j'}{q_{n+1}}, \frac{j'+1}{q_{n+1}} \right[ \right)_0 \times \left( \left[ \frac{j'}{q_{n+1}}, \frac{j'+1}{q_{n+1}} \right[ \right)_{i}$ located on their boundaries. Therefore,

\[ \mu\left( \tilde{\Gamma} \left( \frac{i}{t_{n+1}}, \frac{j_iv_{n+1}}{q_{n+1}} + \frac{\gamma_i}{q_n} \right) \Delta \Gamma' \left( \frac{i}{t_{n+1}}, \frac{j_iv_{n+1}}{q_{n+1}} + \frac{\gamma_i}{q_n},l\right) \right) \leq \left( \frac{1}{q_{n+1}} + \frac{w_{n+1}}{q_{n+1}} \right)b_{n+1}(i/t_{n+1})  \]

Let \[ \Gamma' \left( \frac{h_{n}(i')}{t_{n}}, \frac{jv_{n+1}}{q_{n+1}} + \frac{\gamma}{q_n}  \right) =  \bigcap_{\frac{i}{t_{n+1}}= \frac{i'}{t_n}}^{\frac{i'+1}{t_n}- \frac{1}{t_{n+1}} } \Gamma' \left( \frac{h_{n+1}(i)}{t_{n+1}}, \frac{j_{h_{n+1}(i)}v_{n+1}}{q_{n+1}} + \frac{\gamma_{h_{n+1}(i)}}{q_n} \right)   \]

and

\[  \Delta_n^{'n+1}\left( \frac{h_{n}(i')}{t_n}, \frac{l}{q_n}  \right)= \bigcup_{j \in E\left(\frac{h_{n}(i')}{t_n}, \frac{l}{q_n} \right) } \Gamma' \left( \frac{h_{n}(i')}{t_{n}}, \frac{j}{q_n^{k_0}}\right)    \]

where $E\left(\frac{h_{n}(i')}{t_n}, \frac{l}{q_n} \right)$ was defined in (\ref{defdee}).

By our choice of $w_{n+1}$ in (\ref{defwnplusun}), we have:

\begin{equation}
\label{equidefinitunplusun}
\mu\left( \Delta_n^{'n+1}\left( \frac{h_{n}(i')}{t_n}, \frac{l}{q_n}  \right) \Delta \Delta_n^{n+1}\left( \frac{h_{n}(i')}{t_n}, \frac{l}{q_n}  \right)  \right) \leq  |\mathcal{N}_n^{n+1}|^{t_{n+1}/t_n} \frac{t_{n+1}}{t_n}  b_{n+1} \left( \frac{1}{q_{n+1}} + \frac{w_{n+1}}{q_{n+1}} \right) \leq \frac{1}{2^n q_n^{t_n}}
\end{equation}

Moreover, $\Delta_n^{'n+1}\left( \frac{i'}{t_n}, \frac{l}{q_n}  \right)$ is a fundamental domain of $S_{1/q_n}$,  because $\Delta_n^{n+1}\left( \frac{i'}{t_n}, \frac{l}{q_n}  \right)$ is a fundamental domain of $S_{1/q_n}$.

By corollary \ref{pzeropartition}, $\Delta_n^{'n+1}\left( \frac{i'}{t_n}, \frac{l}{q_n}  \right) \subset  \mathcal{B} \left\{ C(m), m \in u(J_{n+1}(i')) \right\}$. Therefore,
for any $m \in u(J_{n+1}(0) - \{0\} )$, there exists $R^{(m)} \subset \mathcal{B}\left( \mathcal{P}_0 \right)$ such that $R^{(m)}$ is a fundamental domain of the circle rotation $R_{1/q_n}$ and such that:

\[ \Delta_n^{'n+1}\left( \frac{0}{t_n}, \frac{l}{q_n}  \right)  = \bigcup_{m \in E_0(l)} \left( R^{(m)} \right)_0 \times C(m) \]
 
where $ E_0(l) \subset u(J_{n+1}(0) - \{0\} )$.

\bigskip

For $t_n-1 \geq i' \geq 1$, we have:

\[  \Delta_n^{'n+1}\left( \frac{h_n(i')}{t_n}, \frac{l}{q_n}  \right)= \bigcup_{(\gamma,j) \in E\left(\frac{h_n(i')}{t_n}, \frac{l}{q_n} \right) }   \bigcap_{ \frac{i'}{t_{n}} \leq \frac{i}{t_{n+1}} \leq \frac{i'+1}{t_{n}} - \frac{1}{t_{n+1}} }  \bigcup_{l'=0}^{ b_{n+1}(i/t_{n+1})-1  }   \bigcup_{j'=0}^{u_{n+1}} 
\]

\[
\left( \frac{j'w_{n+1}}{q_{n+1}}+ \frac{j_iv_{n+1}}{q_{n+1}} + \frac{\gamma_i}{q_n} + R^{(n,l')}(h_{n+1}(i)/t_{n+1},j_{h_{n+1}(i)}) \right)_0 \times \left( P(j') \right)_{h_{n+1}(i)}   \]

\[ =  \bigcup_{j \in E\left(\frac{h_n(i')}{t_n}, \frac{l}{q_n} \right) }   \bigcap_{ \frac{h_n(i')}{t_{n}} \leq \frac{i}{t_{n+1}} \leq \frac{i'+1}{t_{n}} - \frac{1}{t_{n+1}} }  \bigcup_{l'=0}^{ b_{n+1}(i/t_{n+1})-1  }   \bigcup_{j'=0}^{u_{n+1}} \bigcup_{ m \in E(h_{n+1}(i),j_{h_{n+1}(i)},l',j')  }  C(j',m)  \]

such that $E(i,j_i,l',j') \subset u(i)$, and $|E(i,j_i,l',j')|$ is independent of $j'$. Therefore,

\[  \Delta_n^{'n+1}\left( \frac{h_n(i')}{t_n}, \frac{l}{q_n}  \right)= \bigcup_{j' \in u(0)} \bigcup_{ m \in E(j', h_n(i'),l')} C(j', m)  \]

where $|E(j', i',l')|$ is independent of $j'$.

This completes the proof of proposition \ref{propapprox}.

\end{proof}

\bigskip

The rest of the paper is dedicated to the construction of the sequence of diffeomorphisms $B_n$ satisfying the conditions of lemma \ref{conditionsbnrotisom}. First, we introduce definitions and basic properties. Second, we define the diffeomorphism $A_{n+1}$ and we show that it satisfies conditions of lemma \ref{conditionsbnrotisom}.

\subsection{Definitions and basic properties}

\label{defprops}

For these definitions, we rely and elaborate on \cite[p.155]{halmos74}. Let $J \subset \varmathbb{N}$. For any $X \subset M$, we let $X_J= \{ (x_j)_{j\in J} / x \in X \}$.

A set $E \subset M$ is a $J$-\textit{cylinder} if, for any $x,y \in M$ such that for any $j \in J$, $x_j=y_j$, we do not have: ($x \in E$ and $y \not\in E$) or ($x \not\in E$ and $y \in E$).

In other words, $E$ is a $J$-cylinder if changing the coordinates of an index not in $J$ of a point $x \in E$ cannot remove $x$ from $E$, nor insert it into $E$. The terminology $(z,(x_j)_{j \in J})$-cylinder is synonymous of $\{0,J\}$-cylinder.

For example, $C(m)$ is a $J$-cylinder.
$\Delta_n(0/t_n,l/q_n)$ is a $z$-cylinder, and for $i \geq 1$, $\Delta_n(i/t_n,l/q_n)$ is a $(z,x_i)$-cylinder.

For any $A \subset M$, $A_0 \times ...\times A_n \times M_{n+1} \times ...$ is a $\{0,...,n\}$-cylinder.



We use the following claims:

\begin{claim}
\label{unionintersectcylyndre}
If $E$ is a $J$-cylinder and $E'$ is a $J'$-cylinder, then $E \cup E'$ and $E \cap E'$ are $J \cup J'$-cylinders.
\end{claim}



For example, $\Delta_n(i/t_n,l/q_n) \cap \Delta_n(i'/t_n,l'/q_n)$ is a $(z,x_i,x_{i'})$-cylinder.

For example, by proposition \ref{propapprox}, for any $l=0,...,q_n-1$, $ \Delta_n^{'n+1}\left(\frac{0}{t_n}, \frac{l}{q_n} \right)$ is a $J(0)$-cylinder, and for $i \geq 1$, $\Delta_n^{'n+1}\left(\frac{i}{t_n}, \frac{l}{q_n} \right)$ is a $\{0\} \cup J(i)$-cylinder.

\bigskip

A transformation $F$ of $M$ is a $J$-\textit{transformation} if for any $x \not\in J$, and any $x \in M$, $(F(x))_j=x_j$. It means that at most, $F$ transforms coordinates in $J$.


\begin{claim}
\label{transfocylindreintervide}
If $F$ is a $J$-transformation and if $c$ is a $J'$-cylinder such that $J \cap J'= \emptyset$ , then $F(c)=c$.
\end{claim}

\begin{proof}
First, we show that $F(c) \subset c$. Let $x \in c$ and $y=F(x)$. Since $F$ is a $J$-transformation, then $y_j=(F(x))_j=x_j$ for any $j \in J^c$, where $J^c$ be the complementary of $J$ in $\varmathbb{N}$. Since $J \cap J'= \emptyset$ then $J'\subset J^c$. Therefore, for any $j \in J'$, $y_j=x_j$. Since $c$ is a $J'$-cylinder, then $y \in c$. Therefore, $F(c) \subset c$.

To show that $c \subset F(c)$, we observe that $F^{-1}$ is also a $J$-transformation: if $y_j=x_j$, then $y_j=(F^{-1}(y))_j$ for any $j \in J^c$. Therefore, $F^{-1}(c) \subset c$ and so $c \subset F(c)$.
  
\end{proof}

Let $(J')^c$ be the complementary of $J'$ in $\varmathbb{N}$. A transformation $F$ is $J'$-\textit{dependent} if there exists $\bar{J} \subset J'$ such that $F$ is a $\bar{J}$-transformation, and such that there exists $\tilde{F}:M_{J'} \rightarrow M_{\bar{J}}$ such that for any $x \in M$, $x=(x_j,x_{j'},x_{j''})$ with $x_j \in M_{\bar{J}}$, $x_{j'} \in M_{J'-\bar{J}}$, $x_{j''} \in M_{(J')^c}$, we have:

\[ F(x_j,x_{j'},x_{j''})=(\tilde{F}(x_j,x_{j'}),x_{j'},x_{j''})  \]

\begin{claim}
\label{dependenceavecintersection}
If $c$ is a $J$-cylinder and $F$ is $J'$-dependent, then $F(c)$ is a $J \cup J'$-cylinder.
\end{claim}

\begin{proof}
First, observe that since $F$ is $J'$-dependent, then $F$ is $J \cup J'$-dependent. Let $x \in M$. We write $x=(x_j,x_{j'},x_{j''})$ with $x_j \in M_{\bar{J}}$, $x_{j'} \in M_{J'-\bar{J}}$, $x_{j''} \in M_{(J \cup J')^c}$. Since $F$ is $J \cup J'$-dependent, then there exists $\tilde{F}: M_{J \cup J'} \rightarrow M_{J \cup J'}$ such that:

\[ F(x_j,x_{j'},x_{j''})=(\tilde{F}(x_j,x_{j'}),x_{j''})  \]

Let $y \in M$ such that $(F(x))_l=(F(y))_l$ for any $l \in J \cup J'$. Then we also have:

\[ \tilde{F}(x_j,x_{j'})= \tilde{F}(y_j,y_{j'})  \]

 Let $z= (y_j,y_{j'},x_{j''})$. We have:

 \[  F(z)= (\tilde{F}(y_j,y_{j'}),x_{j''})= (\tilde{F}(x_j,x_{j'}),x_{j''})=F(x)  \]
 
Since $F$ is injective ($F$ is a transformation), then $z=x$. In particular, $x_j=y_j$.

If, moreover, $F(x) \in F(c)$, then $x \in c$, because $F$ is bijective. Since $c$ is a $J$-cylinder, then $y \in c$. Therefore, $F(y) \in F(c)$.

Likewise, if  $F(x) \not\in F(c)$, then $F(y) \not\in F(c)$. Therefore, $F(c)$ is a $J \cup J'$-cylinder.
\end{proof}






\subsection{Construction of the sequence of conjugacies}

\label{constrconjug}

In this section, we construct the sequence of diffeomorphisms $A_{n+1}$ of $M$. We write $A_{n+1}= A_{n+1,t_n} \circ ...\circ A_{n+1,0}$. To simplify notations, we denote $A$ instead of $A_{n+1}$, and for $i=0,...,t_n-1$, we denote $A_i$ instead of $A_{n+1,0}$, and $J(i)$ instead of $J_{n+1}(i)$. First, we construct $A_0$ (lemma \ref{defazero}). Second, we construct $A_i$, for $i \geq 1$ (lemma \ref{defai}). The construction of $A_0$ is different from the others $A_i$ because the partition $\eta_{n,0}$ is different from the others $\eta_{n,i}$. 

The reason why we introduced the permutation of coordinates $h_n(i)$ appears here: for example, if we took $h_n(i)=i$, then for $i \geq 1$, it would be impossible to apply $A_i$ without completely modifying the image of $\eta_{n,0}$ by $A_{n+1,0}$. If we took $h_n(i)$ such that $h_{n+1}(\{ 1,...,t_n-1\}) \cap \{ 1,...,t_n-1\} = \emptyset$, then we cannot obtain generation.

On the other hand, with our choice of $h_n(i)$, we can compose the maps $A_i$ with no problems, by using the notions and properties introduced in \ref{defprops}.

The aim of this subsection is to show the following proposition:

\begin{proposition}
\label{defdea}
For any $\frac{w'_{n+1}}{2q_{n+1}} > \epsilon >0$, there exists $A: M \rightarrow M$ smooth measure-preserving diffeomorphism such that:

\begin{enumerate}
\item $A S_{\frac{1}{q_n}}= S_{\frac{1}{q_n}} A$.
\item There is a fixed function $\rrr(n,t_{n+1},b_{n+1},q_n,\epsilon) \label{rdea} \in \varmathbb{N}$ such that \[ \|A \|_{n+1} \leq R_{\ref{rdea}}(n,t_{n+1},b_{n+1},q_n,\epsilon) \]
\item There exists $E \subset M$ such that: \[ \mu\left( M - E \right) \leq \epsilon \] and for any $i=0,...,t_n-1$,

\[ A\left(E \cap \Delta_n(i/t_n,0/q_n)  \right)= A(E) \cap \Delta_n^{'n+1}(i/t_n,0/q_n)  \]
\item $A$ is a $(z,x_1,...,x_{t_{n+1}-1})$-transformation, $(z,x_1,...,x_{t_{n+1}-1})$-dependent.
\end{enumerate}

\end{proposition}

The construction of $A_0$ is given by the following lemma:

\begin{lemma}
\label{defazero}
For any $\frac{w'_{n+1}}{2q_{n+1}} > \epsilon_0 >0$, there exists $A_0: M \rightarrow M$ smooth measure-preserving diffeomorphism such that:

\begin{enumerate}
\item $A_0S_{\frac{1}{q_n}}= S_{\frac{1}{q_n}} A_0$.
\item There is a fixed function $\rrr(n,t_{n+1},b_{n+1},q_n,\epsilon_0) \label{rdeazero} \in \varmathbb{N}$ such that \[ \|A_{0} \|_{n+1} \leq R_{\ref{rdeazero}}(n,t_{n+1},b_{n+1},q_n,\epsilon_0) \]
\item There exists $E_0 \subset M$ such that: \[ \mu\left( M - E_0 \right) \leq \epsilon_0 \] and:

\[ A_0\left(E_0 \cap \Delta_n(0/t_n,0/q_n)  \right)= \Delta_n^{'n+1}(0/t_n,0/q_n) \cap A_{0}( E_0) \]
\item $A_0$ is a $(z,t_{n})$-transformation, $J(0)$-dependent.
\end{enumerate}

\end{lemma}

To show lemma \ref{defazero}, we use the decomposition of $\Delta_n^{'n+1}(0/t_n,0/q_n) $ given in proposition \ref{propapprox}:

\begin{lemma}
\label{defazerom}
Let $\frac{1}{2u_{n+1}}\epsilon_0>0$. For any $m \in \bar{u}(J(0))- \{0\}$, there exists $A_0(m) \in$ Diff$^\infty(M,\mu)$ such that:

\begin{enumerate}

\item We have: $ A_0(m) S_{\frac{1}{q_n}}=  S_{\frac{1}{q_n} }A_0(m)$.
\item There is a fixed function $\rrr(n,t_{n+1},b_{n+1},q_n,\epsilon_0) \label{rdeazerom} \in \varmathbb{N}$ such that \[ \|A_{0}(m) \|_{n+1} \leq R_{\ref{rdeazerom}}(n,t_{n+1},b_{n+1},q_n,\epsilon_0) \]
\item There exists $E_0(m) \subset [0,1/q_n[_0 \times C(m)$ such that: \[ \mu\left(  [0,1/q_n[_0 \times C(m) - E_0(m)  \right) \leq \epsilon_0 + \epsilon'_0 \] and:

\[ A_0(m)\left(E_0(m) \cap \left([0,1/q_n[_0 \times C(m) \right)  \right)= \left((R^{(m)})_0 \times C(m)\right) \cap A_{0}(m)( E_0(m)) \]
\item $A_0(m)=Id$ on a $\epsilon_0/2$-neighbourhood of $(C(m))^c$.
\item $A_0(m)$ is a $(z,t_{n})$-transformation, $J(0)$-dependent.
\end{enumerate}

\end{lemma}

\begin{proof}[Proof of lemma \ref{defazero}.]
We define $A_0$ by $A_0= \circ_{m \in \bar{u}(J(0)-\{0\})} A_0(m)$, with $A_0(m)$ having disjoint supports. $A_0$ has the required properties.

\end{proof}

\begin{proof}[Proof of lemma \ref{defazerom}.]
By lemmas 3.3, 3.4 and by remark 3.6 in \cite{rotisom}, which elaborates on a construction found in \cite{anosovkatok70} and  \cite{katokant}, for any $ j \leq u_{n+1}-1$, $m \in \bar{u}(J(0)-\{0\})$, there exists $\bar{A}_0(m) \in$ Diff$^\infty(\varmathbb{T} \times [jw_{n+1}/q_{n+1}, (j+1)w_{n+1}/q_{n+1}],Leb_2)$, there exists $\bar{E}_0(m) \subset \bar{\mathcal{P}}_0 \times [jw_{n+1}/q_{n+1}, (j+1)w_{n+1}/q_{n+1}]$ such that \[Leb_2\left( \varmathbb{T} \times [jw_{n+1}/q_{n+1}, (j+1)w_{n+1}/q_{n+1}] - \bar{E}_0(m)\right) \leq \epsilon_0 + \epsilon'_0 \] and such that

\[  \bar{A}_0(m)\left( \left(  [0,1/q_n[ \times [jw_{n+1}/q_{n+1}, (j+1)w_{n+1}/q_{n+1}] \right) \cap \bar{E}_0(m) \right)  \] \[=  \left( R^{(m)} \times [jw_{n+1}/q_{n+1}, (j+1)w_{n+1}/q_{n+1}] \right) \cap   \bar{A}_0(m)( \bar{E}_0(m)) \]

such that $\bar{A}_0(m)= Id$ on a $\epsilon_0/2$-neighbourhood of the boundary of $\varmathbb{T} \times  [jw_{n+1}/q_{n+1}, (j+1)w_{n+1}/q_{n+1}]$, such that $\bar{S}_{1/q_n} \bar{A}_0(m)= \bar{A}_0(m) \bar{S}_{1/q_n}$, where $\bar{S}_t(z,x)=(z+t,x)$ for $(z,x) \in \varmathbb{T} \times  [jw_{n+1}/q_{n+1}, (j+1)w_{n+1}/q_{n+1}]$, and such that

\[   \|\bar{A}_{0}(m) \|_{n+1} \leq \rrr (u_{n+1},\epsilon_0) \]

Moreover, there exists a permutation $\sigma(m)$ of $\bar{\mathcal{P}}_0$ such that

\begin{equation}
\label{defdelapermutationsigma}
\sigma(m)\left([0,1/q_n[_0\right)= R^{(m)}
\end{equation}  

The permutation $\sigma(m)$ stabilizes $\bar{\mathcal{P}}_0$, and for any $k \in \bar{u}(0)$,

\[
\bar{A}_0(m)\left( \left(  P_0(k) \times [jw_{n+1}/q_{n+1}, (j+1)w_{n+1}/q_{n+1}] \right) \cap \bar{E}_0(m) \right) \]

\begin{equation}
\label{propdelapermutationsigma}
=  \left(  P_0(\sigma(k)) \times [jw_{n+1}/q_{n+1}, (j+1)w_{n+1}/q_{n+1}]  \right) \cap   \bar{A}_0(m)( \bar{E}_0(m))
\end{equation}

We can extend $\bar{A}_{0}(m)$ to a smooth measure-preserving diffeomorphism of $\varmathbb{T} \times [0,1]$ equal to identity out of $\varmathbb{T} \times [jw_{n+1}/q_{n+1}, (j+1)w_{n+1}/q_{n+1}] $. We write:

\[   \bar{A}_{0}(m)(z,x_{t_n})=(\bar{A}_{0}(m)_0(z,x_{t_n}), \bar{A}_{0}(m)_{t_n}(z,x_{t_n}))  \] 

Let $\phi: M_{J(0)-\{0\}} \rightarrow [0,1]$ a smooth map such that:

\begin{equation}
\label{estphi}
\| \phi \|_{n+1} \leq \rrr(\epsilon_0,u_{n+1}) 
\end{equation}  

such that $\phi=1$ on a $\epsilon_0/2$-neighbourhood of the boundary of $C(m)$, inside $C(m)$, such that $\phi=0$ on a $\epsilon_0/2$-neighbourhood of $(C(m))^c$. We choose the integer $j$ such that $C(m)= ([j/u_{n+1},(j+1)/u_{n+1}[)_{t_n} \times C(m')$ for some $m'$. For $x=(z,x_{t_n},x') \in M_{J(0)}$, let

\[  A_{0}(m)(z,x_{t_n},x')=(\bar{A}_{0}(m)_0(z,x_{t_n})\phi(x)+ (1-\phi(x))z,\bar{A}_{0}(m)_{t_n}(z,x_{t_n})\phi(x)+ (1-\phi(x)) x_{t_n},x')   \]

Properties 1,2,3,4 hold for $A_{0}(m)$, because similar properties hold for $\bar{A}_{0}(m)$, and because $\phi$ is independent of $z$, and bounded in estimation (\ref{estphi}).

Property 5 is obtained by construction, the $J(0)$-dependence is due to $\phi$.

\end{proof}

For $t_n-1 \geq i \geq 1$, the definition of $A_i$ is given by:

\begin{lemma}
\label{defai}
For $t_n-1 \geq i \geq 1$, for any $\frac{1}{2u_{n+1}} > \epsilon_i >0$, there exists $A_i: M \rightarrow M$ smooth measure-preserving diffeomorphism such that:

\begin{enumerate}
\item $A_iS_{\frac{1}{q_n}}= S_{\frac{1}{q_n}} A_i$.
\item There exists a fixed function $\rrr(u_{n+1},\epsilon_0) \label{rdeai} \in \varmathbb{N}$ such that \[ \|A_{i} \|_{n+1} \leq R_{\ref{rdeai}}(u_{n+1},\epsilon_i) \]
\item There exists $E_i \subset M$ such that: \[ \mu\left( M - E_i \right) \leq \epsilon_i +\epsilon'_1 \] and:

\[ A_i\left(E_i \cap  C_{n,0}^{'n+1}( \Delta_n(i/t_n,0/q_n) ) \right)= \Delta_n^{'n+1}(i/t_n,0/q_n) \cap A_{i}( E_i) \]
where $C_{n,0}^{'n+1}$ is defined in lemma \ref{defdeczeroprime} ($A_0$ is a smooth approximation of $C_{n,0}^{'n+1}$).
\item $A_i$ is a $J(i)$-transformation, $J(0) \cup J(i)$-dependent.
\end{enumerate}

\end{lemma}

The proof of lemma \ref{defai} is based on the following lemma:

\begin{lemma}
\label{defdeczeroprime}
There exists a $z$-transformation $C_{n,0}^{'n+1}$, $J(0)$-dependent, measure-preserving, $S_{\frac{1}{q_n}}$-equivariant, such that:

\[  C_{n,0}^{'n+1}\left(  \Delta_n\left(0/t_n,l/q_n\right) \right)= \Delta_n^{'n+1}\left(0/t_n,l/q_n\right) \]

and \[  C_{n,0|E_0}^{'n+1}= A_{0|E_0}  \]

\end{lemma}

\begin{proof}

Let $m \in u\left(J(0)-\{0\}\right)$ and $\sigma(m)$ be the permutation of $\bar{u}(0)$, given in (\ref{defdelapermutationsigma}) and (\ref{propdelapermutationsigma}), extended to identity on the rest of $u(0)$. Let $C_{n,0}^{'n+1}$ be the piecewise-linear, $S_{1/q_n}$-equivariant, measure-preserving $z$-transformation, $J(0)$-dependent, such that for any $P_0(j) \in \mathcal{P}_0$:

\[  C_{n,0}^{'n+1}: \left( P_0(j) \right)_0 \times C(m) \mapsto  \left( P_0(\sigma(m)(j)) \right)_0 \times C(m)  \]

We have:

\[  C_{n,0|E_0}^{'n+1}= A_{0|E_0}  \]

\end{proof}

\begin{lemma}
\label{defdetau}
For any $i=1,...,t_n-1$, there exists a $J(i)$-transformation $\tau(i,m)$, $J(0) \cup J(i)$-dependent, measure-preserving, $S_{\frac{1}{q_n}}$-equivariant, such that, for any $l=0,...,q_n-1$:

\[ \Delta_n^{'n+1}\left(i/t_n,l/q_n\right)= \tau(i,m) \left( C_{n,0}^{'n+1}\left(  \Delta_n\left(0/t_n,l/q_n\right) \right) \right) \]

Moreover, the $J(0)$-dependence of $\tau(i,m)$ only depends on $u(J(0))$, and for any $(i_0,m) \in u(J(0))$, $\tau(i)(i_0,m)$ is a permutation of $u(J(i))$.

\end{lemma}

\begin{proof}
By the proof of lemma \ref{defdeczeroprime}, for any $(i_0,m) \in u(J(0))$, we have:

\[  C_{n,0}^{'n+1} (C(i_0,m))=  C(\sigma(m)(i_0),m)  \]

On the other hand, 

\[ \Delta_n\left(i/t_n,l/q_n\right)=  \bigcup_{i_0 \in u(0)} \bigcup_{m' \in E(i,i_0,l)} C(i_0,m')  \]

with $E(i,i_0,l) \subset u\left(\{i \}\right)$ and $|E(i,i_0,l)|$ independent of $i_0$.

Since $C(i_0,m')= \cup_{ m \in u(J(0) - \{0\})} C(i_0,m',m)$ then

\[ \Delta_n\left(i/t_n,l/q_n\right)=  \bigcup_{i_0 \in u(0)} \bigcup_{m' \in E(i,i_0,l)}  \bigcup_{m \in u\left(J(0)- \{0\} \right) } C(i_0,m',m)  \]

Therefore,

\[ C_{n,0}^{'n+1} \left(\Delta_n\left(i/t_n,l/q_n\right) \right)=  \bigcup_{i_0 \in u(0)} \bigcup_{m' \in E(i,i_0,l)}  \bigcup_{m \in u\left(J(0)- \{0\} \right) } C( \sigma(m)(i_0),m',m)  \]

\[  = \bigcup_{i_0 \in u(0)}  \bigcup_{m \in u\left(J(0)- \{0\} \right) } \bigcup_{m' \in E(i,l,i_0,m)} C( \sigma(m)(i_0),m',m)  \]

with $|E(i,l,i_0,m)|$ independent of $i_0,m$. Let $i'_0=\sigma(m)(i_0)$. We get:

\[  C_{n,0}^{'n+1} \left(\Delta_n\left(i/t_n,l/q_n\right) \right) = \bigcup_{i_0'=0}^{u_{n+1}-1}  \bigcup_{m \in u\left(J(0)- \{0\} \right) } \bigcup_{m' \in E(i,l,\sigma(m)^{-1}(i_0'),m)} C( i_0',m',m)  \]

with $E(i,l,\sigma(m)^{-1}(i_0'),m) \subset u(\{i\})$. Since for any $m' \in u\left(\{i \}\right)$, \[C(i_0,m')= \cup_{ m \in u(J(i) - \{i\})} C(i_0,m',m) \] then

\[  C_{n,0}^{'n+1} \left(\Delta_n\left(i/t_n,l/q_n\right) \right) = \bigcup_{i_0 \in u(0)}  \bigcup_{m \in u\left(J(0)- \{0\} \right) } \bigcup_{m' \in E'(i,l,i_0,m)} C( i_0,m',m)  \]

with $ E'(i,l,i_0,m) \subset u(J(i))$ and $ |E'(i,l,i_0,m)|$ independent of $i_0,m$.

On the other hand, by proposition \ref{propapprox},

\[  \Delta_n^{'n+1}\left(i/t_n,l/q_n\right)  = \bigcup_{i_0 \in u(0)}  \bigcup_{m' \in F(i,l,i_0)} C( i_0,m')  \]
 
where $F(i,l,i_0) \subset u(J(i))$ and $|F(i,l,i_0)|$ is independent of $i_0$.

Therefore,

\[  \Delta_n^{'n+1}\left(i/t_n,l/q_n\right)  = \bigcup_{i_0 \in u(0)}  \bigcup_{m' \in F(i,l,i_0)} \bigcup_{m \in u\left(J(0)- \{0\} \right) }    C( i_0,m',m)  \]

\[ = \bigcup_{i_0 \in u(0)}  \bigcup_{m \in u\left(J(0)- \{0\} \right) } \bigcup_{m' \in F(i,l,i_0)}    C( i_0,m',m)  \]

For $i_0 \in u(0)$ such that $P_0(i_0) \subset [0,1/q_n[$, let $\tau(i,l,i_0,m)$ be a permutation of $u(J(i))$ such that 
 
\[  \tau(i,l,i_0,m)(E'(i,l,i_0,m))=F(i,l,i_0) \] 

Since the $z$-coordinate is on $\varmathbb{T}$, we can consider $i_0 \mod \, u_{n+1}$ on the $z$-coordinate. We have:

\[ S_{\frac{1}{q_n}}    \Delta_n^{'n+1}\left(i/t_n,l/q_n\right)=   \Delta_n^{'n+1}\left(i/t_n,(l+1)/q_n\right)  \]

On the other hand, since $\mathcal{P}_0$ is stable by the circle rotation $R_{1/q_n}$, then:

\[ S_{\frac{1}{q_n}}    \Delta_n^{'n+1}\left(i/t_n,l/q_n\right)=  \bigcup_{i_0 \in u(0)}  \bigcup_{m' \in F(i,l,i_0)} \bigcup_{m \in u\left(J(0)- \{0\} \right) }    C( R_{1/q_n}(i_0),m',m)  \]

\[  =  \bigcup_{i_0 \in u(0)}  \bigcup_{m' \in F(i,l,R_{-1/q_n}(i_0)   )} \bigcup_{m \in u\left(J(0)- \{0\} \right) }    C( i_0,m',m)  \]

Therefore,  $F(i,l,R_{-1/q_n}(i_0) )= F(i,l+1,i_0)$. 

Likewise, $E'(i,l,R_{-1/q_n}(i_0),m )= E'(i,l+1,i_0,m)$. Therefore, we can extend $ \tau(i,l,i_0,m)$ to any $i_0 \in u(0)$ by $\tau(i,l,R_{-1/q_n}(i_0),m )= \tau(i,l+1,i_0,m)$. We have:

\[  \Delta_n^{'n+1}\left(i/t_n,l/q_n\right)  = \bigcup_{i_0 \in u(0)} \bigcup_{m \in u\left(J(0)- \{0\} \right) }  \bigcup_{m' \in \tau(E'(i,l,i_0,m))}    C( i_0,m',m)  \]

\[  \Delta_n^{'n+1}\left(i/t_n,l/q_n\right) = \bigcup_{i_0 \in u(0)} \bigcup_{m \in u\left(J(0)- \{0\} \right) }  \bigcup_{m' \in E'(i,l,i_0,m)}    C( i_0,\tau(i,l,i_0,m)(m'),m)  \]

This completes the definition of $\tau(i,m)$.

\end{proof}

We make a smooth approximation of $\tau$. To that end, we use the following lemma, which generalizes a proposition in \cite{nlbernoulli}:

\begin{lemma}
\label{quasipermutgeneral}
For any $J,J'\subset \varmathbb{N}$, $J \cap J'= \emptyset$, $|J'| \geq 2$, for any permutation $\sigma$ of $\bar{u}(J')$, for any $0<\epsilon<1/(2u_{n+1})$, for any $m \in \bar{u}(J)$, there exists $E_\sigma(m) \subset C(m)$ such that $\mu(C(m) - E_\sigma(m)) \leq \epsilon$, there exists a $J'$-transformation $A(\sigma)$, $J \cup J'$-dependent, smooth, measure-preserving, such that for any $m'\in u(J')$:

\[ A(\sigma)(C(m,m')) \cap E_\sigma(m)   = C(m,\sigma(m')) \cap E_\sigma(m) \]

and: 
\begin{itemize}
\item There exists $A(\sigma)=Id$ on a $\epsilon/2$-neighbourhood of $(C(m))^c$.
\item There exists a fixed function $\rrr(u_{n+1},\epsilon_0,\epsilon'_0) \label{rdeasigma} \in \varmathbb{N}$ such that \[ \|A(\sigma)\|_{n+1} \leq R_{\ref{rdeasigma}}(u_{n+1},\epsilon, |J|+|J'|,\sum_i \epsilon_i) \]
\end{itemize}

\end{lemma}

\begin{proof}

By composition, we can suppose $|J'|=2$, and $\sigma$ is the transposition one one coordinate: $\sigma(i,j)=(i',j)$. By \cite{nlbernoulli}, there exists $\bar{A}(\sigma): (C(j))_{J'} \rightarrow (C(j))_{J'} $ smooth, measure-preserving diffeomorphism, and there exists $E_\sigma(j) \subset (C(j))_{J'} $ such that $ \mu((C(j))_{J'} - _\sigma(j)) \leq \epsilon$, and for any $(x,y) \in (C(j))_{J'} $,

\[  \bar{A}(\sigma) (x,y) \cap E_\sigma(j)= \sigma(x,y) \]

and here exists a fixed function $\rrr(u_{n+1},\epsilon_0) \label{rdeasigmabarre} \in \varmathbb{N}$ such that \[ \|A(\sigma)\|_{n+1} \leq R_{\ref{rdeasigmabarre}}(u_{n+1},\epsilon) \]

From there, the rest of the proof is analogous to the proof of lemma \ref{defazerom}, using a smooth plateau function $\phi$.

\end{proof}

\begin{proof}[Proof of lemma \ref{defai}.]

We let $A_i= \circ_{m \in \bar{u}(J(0))} A(\tau(m))$ with $\tau(i,m)$ coming from lemma \ref{defdetau}, and $A(\tau(i,m))$ coming from lemma \ref{quasipermutgeneral} with $J=J(0),J'=J(i)$.

\end{proof}

\begin{proof}[Proof of proposition \ref{defdea}]

We let $A= A_{t_n-1} \circ ... \circ A_1 \circ A_0$. Conditions 1,2 hold because they hold for any $A_i$. Condition 4 holds because the $\{J(i),i=0,...,t_n-1\}$ partition $\{0,...,t_{n+1}-1 \}$. To show condition 3, let:

\[  E=E_0 \cap (A_0)^{-1}(E_1) \cap ...\cap (A_{t_{n}-2}...A_0)^{-1}(E_{t_{n}-1})  \]

For any $i$, we take $0 < \epsilon_i = \epsilon/t_{n}$. We have: $\mu(M - E) \leq \epsilon$. We also have:

\[ A\left( E \cap \Delta_n(0/t_n,0/q_n) \right)= A( \cap_{i \not\eq 0}  E_i ) \cap A(E_0 \cap  \Delta_n(0/t_n,0/q_n))  \]

\[ = A( \cap_{i \not\eq 0}  E_i ) \cap A_{t_n-1}...A_1 A_0(E_0 \cap  \Delta_n(0/t_n,0/q_n))  \] 

By lemma \ref{defazero}, we get:

\[ A\left( E \cap \Delta_n(0/t_n,0/q_n) \right)= A( \cap_{i \not\eq 0}  E_i ) \cap A_{t_n-1}...A_1 \left( A_0(E_0) \cap  \Delta_n^{'n+1}(0/t_n,0/q_n) \right)  \] 

\[ =  A(E) \cap  A_{t_n-1}...A_1 \left( \Delta_n^{'n+1}(0/t_n,0/q_n) \right)   \]

For any $i=1,...,t_n-1$, $A_i$ is a $J(i)$-transformation. On the other hand, $\Delta_n^{'n+1}(0/t_n,0/q_n)$ is a $J(0)$-cylinder, and $J(i) \cap J(0)= \emptyset$. By claim \ref{transfocylindreintervide},

\[ A_i \left(  \Delta_n^{'n+1}(0/t_n,0/q_n)  \right) = \Delta_n^{'n+1}(0/t_n,0/q_n)   \]

Therefore, \[  A\left(E \cap \Delta_n(0/t_n,0/q_n) \right)= A(E) \cap  \Delta_n^{'n+1}(0/t_n,0/q_n)  \]

Hence proposition \ref{defdea} for $i=0$. Let $t_n-1 \geq  i \geq 1$. Since 

\[  C_{n,0|E_0}^{'n+1}= A_{0|E_0}  \]

and since $E \subset E_0$, then 

\[  A_0\left( E \cap \Delta_n(i/t_n,0/q_n) \right) = A_0(E) \cap C_{n,0}^{'n+1} \left( \Delta_n(i/t_n,0/q_n) \right)  \]

The map $C_{n,0}^{'n+1}$ is $J(0)$-dependent and $\Delta_n(i/t_n,0/q_n)$ is a $(z,x_i)$-cylinder. Therefore, by claim \ref{dependenceavecintersection}, $C_{n,0}^{'n+1} \left( \Delta_n(i/t_n,0/q_n) \right) $ is a $J(0) \cup \{i\}$-cylinder.

Let $j \not\in \{0,i\}$. Since $A(j)$ is a $J(j)$-transformation and $J(j) \cap \{  J(0) \cup \{i\} \}= \emptyset$, then by claim \ref{transfocylindreintervide},

\[ A_j \left(  C_{n,0}^{'n+1} \left( \Delta_n(i/t_n,0/q_n) \right)  \right) = C_{n,0}^{'n+1} \left( \Delta_n(i/t_n,0/q_n) \right)   \]

Therefore, 

\[ A_i A_{i-1}...A_0\left( E \cap \Delta_n(i/t_n,0/q_n)  \right) = A_i...A_0(E) \cap  A_i \left( C_{n,0}^{'n+1} \left( \Delta_n(i/t_n,0/q_n) \right) \right)    \]

Moreover, $A_{i-1}...A_0\left( E  \right) \subset E_i$. Therefore, by proposition \ref{defai},

\[ A_i A_{i-1}...A_0\left( E \cap \Delta_n(i/t_n,0/q_n)  \right) = A_i...A_0(E) \cap  \Delta_n^{'n+1}(i/t_n,0/q_n)   \]

Likewise, $ \Delta_n^{'n+1}(i/t_n,0/q_n)$ is a $\{0\} \cup J(i)$-cylinder and for any $j \not\in \{0,i\}$, $\{ \{0\} \cup J(i) \} \cap J(j)= \emptyset$. Therefore, by claim \ref{transfocylindreintervide},

\[ A_j \left(   \Delta_n^{'n+1}(i/t_n,0/q_n) \right) =   \Delta_n^{'n+1}(i/t_n,0/q_n)  \]

Therefore,

\[  A \left( E \cap \Delta_n(i/t_n,0/q_n) \right)=  A_{t_n-1}...A_0 \left( E \cap \Delta_n(i/t_n,0/q_n) \right)= A(E) \cap  \Delta_n^{'n+1}(i/t_n,0/q_n)   \]

\end{proof}

\subsection{Generation of $\xi_n$ and convergence of the sequence of diffeomorphisms}
By combining lemma \ref{conditionsbnrotisom}, in order to complete the proof of lemma \ref{lemmefonda}, it remains to show that $\xi_n$ generates, that $T_n= B_n^{-1} S_{\frac{p_n}{q_n}} B_n$ converges in the smooth topology. The limit $T$ of $T_n$ is ergodic because it is isomorphic to an ergodic transformation (De La Rue's transformation is ergodic).

\subsubsection{Generation of $\xi_n$}

\begin{lemma}
\label{xingenerates}
$\xi_n$ generates.
\end{lemma}

\begin{proof}
Let $p \geq 0$ and for $x,y \in M$, let $ d_p(x,y)= \max_{0 \leq i \leq p}|x_i-y_i|$. For any $V \subset M$, let the $p$-\textit{diameter} of $V$ be $d_p(V)=\sup_{x,y \in V} d_p(x,y)$. We use the lemma:

\begin{lemma}
\label{lemmedinfini}
Let $V_n \subset M$ be a sequence of subsets such that there exists $r_n \rightarrow + \infty$ such that $d_{r_n}(V_n) \rightarrow_{n \rightarrow + \infty} 0$. For any $k \geq 0$, $\cap_{n \geq k} V_n$ is at most a singleton.
\end{lemma}

\begin{proof}
Let $k \geq 0$ and $x,y \in \cap_{n \geq k} V_n$. Let $M \geq 0$ and $n_0$ such that for any $n \geq n_0$, $r_n \geq M$. For $n \geq n_0$, we have: $d_M(x,y) \leq d_{r_n}(x,y) \rightarrow 0$. Therefore, $x_i=y_i$ for $i=0,...,M$ for any $M$ and therefore, $x=y$.

\end{proof}

The diffeomorphism $A_{n+1|E_{n+1}}$ acts as a permutation, and in particular, it is an isometry. On the other hand, elements of $\eta_{n+1}$ are of the form $\times_{i=0}^{t_{n+1}-1} ([j_i/q_{n+1}, (j_i+1)/q_{n+1}[)_i$, for $j_i=0,...,q_{n+1}-1$. Therefore, for $q_{n+1} \geq \rrr(n,t_{n+1},b_{n+1},q_n,\epsilon) \label{rgener}$, for any $c \in \eta_{n+1}$,

\[ d_{t_{n+1}}\left(A_{n+1}^{-1}(c \cap A_{n+1}(E_{n+1})\right) \leq \frac{1}{2^n \|B_n\|_1 } \]

In particular, 

\[ d_{t_{n+1}}\left(B_{n+1}^{-1}(c \cap A_{n+1}(E_{n+1})\right) \leq \frac{1}{2^n } \]

Let $x \in M$ and $c_n(x)$ be the element of $\eta_n$ to which $x$ belongs. Let \[ G(x)= \bigcap_{n \geq 1} B_{n}^{-1}(c_{n}(B_{n}(x))) \]

Let \[ F= \bigcup_{k \geq 1} \bigcap_{n \geq k} B_{n-1}^{-1}(E_{n}) \]

We have $\mu(E_n) \leq 1/2^n$, and $B_{n}$ is measure-preserving. Therefore, by the Borel-Cantelli lemma, $\mu(F)=1$. We show that for any $x \in F$, $G(x) \cap F$ is a singleton. For any $x \in F$, $x \in G(x) \cap F$, and therefore, $\#( G(x) \cap F) \geq 1$. On the other hand, 

\[  G(x) \cap F= \bigcup_{k \geq 1} \bigcap_{n \geq k} B_n^{-1}(E_n) \bigcap_{n \geq 0} B_n^{-1}(c_n(B_n(x))) \subset \bigcup_{k \geq 1} \bigcap_{n \geq k} B_{n-1}^{-1}(E_n) \cap B_n^{-1}(c_n(B_n(x)))  \]

Let $V_n(c)= B_n^{-1}(c \cap A_n(E_n)$. We have:

\[ G(x) \cap F \subset  \bigcup_{k \geq 1} \bigcap_{n \geq k} V_n(c_n(B_n(x)))  \] 

By lemma \ref{lemmedinfini}, for any integer $k$, $\bigcap_{n \geq k} V_n(c_n(B_n(x)))  $ is at most a singleton. Moreover, $\bigcap_{n \geq k} V_n(c_n(B_n(x)))  $ is an increasing sequence of sets for the inclusion. Therefore, 

\[  \bigcup_{k \geq 1} \bigcap_{n \geq k} V_n(c_n(B_n(x))) \]

is at most a singleton, and $G(x) \cap F =\{x\}$. It shows that $\xi_n$ generates.

\end{proof}

\subsubsection{Convergence}

To complete the proof of lemma \ref{lemmefonda} for $M=[0,1]^{d-1} \times \varmathbb{T}$, we need to show the convergence of $T_n= B_n^{-1} S_{\frac{p_n}{q_n}} B_n $. By the Cauchy criterion, it suffices to show that $\sum_{n \geq 0} d_n(T_{n+1},T_n)$ converges. We combine the estimation of $B_{n+1}$ and the assumption of closeness between $p_{n+1}/q_{n+1}$ and $p_n/q_n$ of lemma \ref{lemmefonda}. We recall the lemma \cite[p.1812]{windsor07}:

\begin{lemma}
\label{faadirot}
Let $k \in \varmathbb{N}$. There is a constant $C(k,d)$ such that, for any $ h \in$ Diff$(\varmathbb{T} \times [0,1]^{d-1})$, $\alpha_1,\alpha_2 \in \varmathbb{R}$, we have:

\[ d_k(hS_{\alpha_1} h^{-1},hS_{\alpha_2} h^{-1} ) \leq C(k,d) \|h\|_{k+1}^{k+1} |\alpha_1-\alpha_2|  \]

\end{lemma}

Since $T_n$ and $T_{n+1}$ are $(z,x_1,...,x_{t_{n+1}})$-transformations and $(z,x_1,...,x_{t_{n+1}})$-dependent, they can be seen as diffeomorphisms of $\varmathbb{T} \times [0,1]^{t_{n+1}-1}$. Moreover, $T_n= B_n^{-1} S_{\frac{p_n}{q_n}} B_n = B_{n+1}^{-1} S_{\frac{p_n}{q_n}} B_{n+1} $. Therefore, we obtain, for a fixed sequence $\rr(n,b_{n+1},q_n,t_{n+1}) \label{r5cste}$ (that depends on $n$ and on the dimension $d$):

\begin{eqnarray*}
d_n(T_{n+1},T_n) = d_n ( B_{n+1}^{-1} S_{\frac{p_{n+1}}{q_{n+1}}} B_{n+1}, B_{n+1}^{-1} S_{\frac{p_n}{q_n}} B_{n+1}) \leq C(n+1,t_{n+1}) \|B_{n+1}\|_{n+1}^{n+1} \left| \frac{p_{n+1}}{q_{n+1}} - \frac{p_n}{q_n} \right|  \\
\leq R_{\ref{r5cste}}(n,b_{n+1},q_n,t_{n+1}) \left| \frac{p_{n+1}}{q_{n+1}} - \frac{p_n}{q_n} \right|  
\end{eqnarray*}

For some choice of the sequence $R_{\ref{r0csterotisom}}(n,b_{n+1},q_n,t_{n+1})$ in lemma \ref{lemmefonda}, this last estimate guarantees the convergence of $T_n$ in the smooth topology.

\bibliography{refdyn}

\newpage

Erratum: lemma 2.3 page 12 is false, because in page 13, we do not have:

\[ \mathcal{B}_{p+1}(q)  = \mathcal{B}_{G,p}(q) \vee \mathcal{B}_{D,p}(q) \]

Instead, we have:

\[ \mathcal{B}_{p+1}(q)  = \mathcal{B}_{G,p}(q) \vee \mathcal{B}_{D,p}(q) \vee \mbox{arg} B_1(q)   \]

However, I still think that the main theorem is true, by using the fact that $\mathcal{B}_{n}(q_n)$ generates, i.e. by applying Corollaire 3.3 page 54 of de la Rue's PhD thesis, and by considering more general partitions of $\varmathbb{T} \times [0,1]^{\varmathbb{N}}$ of this form:

\begin{figure}[h]
\centering
\includegraphics[height=7cm]{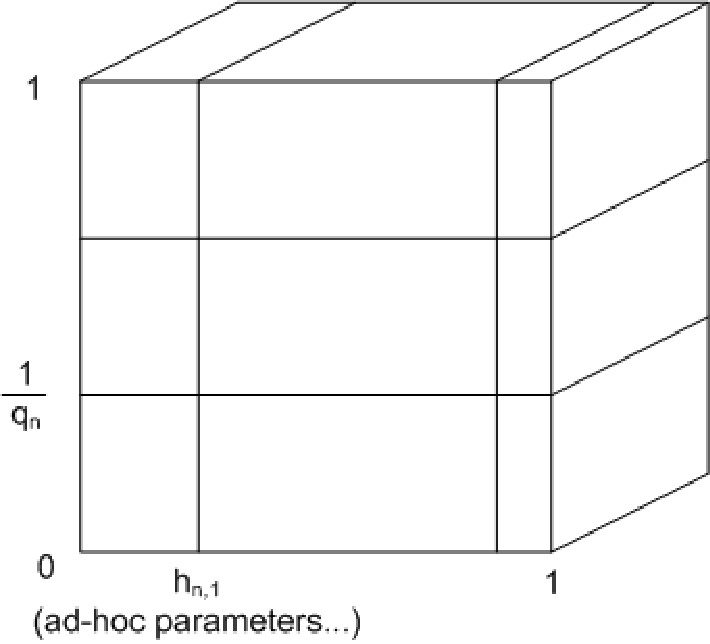}
\end{figure}

Nevertheless, due to lack of time, I am unable to write the details.

\bigskip

I would like to thank Thierry De la Rue for mentioning this mistake.

\end{document}